\documentclass[10pt]{article}
\usepackage[utf8]{inputenc} 
\usepackage[T1]{fontenc}    
\usepackage[hidelinks]{hyperref}       
\usepackage{url}            
\usepackage{booktabs}       
\usepackage{amsfonts}       
\usepackage{nicefrac}       
\usepackage{microtype}      
\usepackage[table]{xcolor}
\usepackage{natbib}
\usepackage{fullpage}
\usepackage{multirow}
\usepackage{soul}           

\newcommand{\Proj}{\textup{Proj}}

\newcommand{\calA}{\mathcal{A}}
\newcommand{\calB}{\mathcal{B}}
\newcommand{\calC}{\mathcal{C}}
\newcommand{\calS}{\mathcal{S}}
\newcommand{\calT}{\mathcal{T}}

\newcommand{\E}{\mathbb{E}}
\renewcommand{\P}{\mathbb{P}}

\newcommand{\dist}{\mathrm{dist}}
\newcommand{\ls}{\left}
\newcommand{\rs}{\right}

\newcommand{\interior}{\mathrm{int}}

\newcommand{\dom}{\mathrm{dom}}
\newcommand{\inner}[2]{\ls\langle #1, \; #2 \rs\rangle}

\newcommand\numberthis{\addtocounter{equation}{1}\tag{\theequation}}
\makeatletter

\newcommand{\Rmnum}[1]{\textup{\expandafter\@slowromancap\romannumeral #1@}}
\newcommand{\cmark}{\checkmark}
\newcommand{\xmark}{\ensuremath{\times}}
\makeatother

\usepackage[table]{xcolor}
\usepackage{tabstackengine}
\usepackage{appendix}
\usepackage{titletoc}
\usepackage{titlesec}
\usepackage{enumerate}
\usepackage{amsmath}
\usepackage{amssymb}
\usepackage{mathtools}
\usepackage{amsthm}
\usepackage{algorithm}
\usepackage{algorithmic}
\usepackage{graphicx}
\usepackage{makecell}
\usepackage{array}
\usepackage{threeparttable}
\usepackage{setspace}
\usepackage{dsfont}
\usepackage{multicol}
\usepackage{authblk}
\usepackage{lipsum}
\usepackage{tabularx}
\usepackage{enumitem}
\usepackage{subcaption}
\allowdisplaybreaks

\theoremstyle{plain}
\newtheorem{theorem}{Theorem}[section]

\newtheorem{lemma}[theorem]{Lemma}

\newtheorem{remark}{Remark}[section]
\theoremstyle{definition}

\newtheorem{assumption}[theorem]{Assumption}

\newtheorem{example}[theorem]{Example}

\title{On the Policy Convergence of Policy Mirror Descent Methods}
\author[1]{Wenye Li}
\author[1]{Ke Wei}
\affil[1]{School of Data Science, Fudan University, Shanghai, China.}
\date{\today}

\begin{document}
\maketitle
\begin{abstract}
We study the policy convergence of unregularized policy mirror descent (PMD) with arbitrary constant step sizes for finite discounted Markov decision processes. We focus on decomposable mirror maps of the form $h(p)=\sum_a \psi(p(a))$, where $\psi$ satisfies standard Legendre-type assumptions. Under these conditions, we prove that the policy sequence generated by PMD converges in the policy domain to a limiting optimal policy, even when the optimal policy set is not a singleton. This result covers a broad class of commonly used mirror maps, including the squared Euclidean mirror map underlying projected Q-ascent, the negative Shannon entropy underlying softmax natural policy gradient, Tsallis entropy, the Hellinger mapping, and the Fermi-Dirac entropy. Although policy convergence has been established previously for specific PMD instances or for regularized variants such as homotopic PMD,  to the best of our knowledge, this is the first systematic and unified policy convergence theory for unregularized PMD under general decomposable mirror maps and arbitrary constant step sizes. Our analysis further reveals that the convergence behavior is governed by the differentiability of $\psi$ at $0$ and $1$, leading to different behaviors, including finite-time termination, asymptotic convergence, and an MDP-dependent dichotomy. When $\psi$ is twice continuously differentiable with strictly positive finite curvature, we further establish local policy convergence rates for the asymptotic convergence cases, covering the standard mirror maps mentioned above.
\end{abstract}

\section{Introduction}
Reinforcement learning (RL) studies sequential decision-making problems and has found successful applications in robotics, scientific discovery, and the alignment of large language models~\citep{robot1,robot2,ref-AlphaFold,ref-AlphaTensor,gpt4,Zhang2025ASO}. The Markov decision process (MDP) provides a standard mathematical framework for RL. Let $\Delta(\mathcal{X})$ be the probability simplex over a finite set $\mathcal{X}$. In this paper, we consider a discounted finite MDP $(\calS, \, \calA, \, P, \, r, \, \gamma)$, where $\calS$ is the state space, $\calA$ is the action space, $P: \calS \times \calA \to \Delta(\calS)$ is the transition model, $r: \calS \times \calA \to [0,1]$ is the reward function, and $\gamma \in [0,1)$ is the discount factor. Under a given policy $\pi: \calS \to \Delta(\calA)$, the state value $V^\pi(s)$ is defined as the expected discounted cumulative reward starting from $s\in\calS$:
\begin{align*}
    V^\pi(s) := \underset{a_t \sim \pi(\cdot|s_t), \, s_{t+1} \sim P(\cdot|s_t, a_t)}{\E} \ls[ \sum_{t=0}^\infty \gamma^t r(s_t, a_t) \, \Big | \, s_0 = s, \, \pi \rs].
\end{align*}
Let $\Pi$ denote the set of all admissible policies. The goal of RL is to find an optimal policy that maximizes the expected state value,
\begin{align}  \label{eq:RL-goal}
    \max_{\pi \in \Pi} \; V^\pi(\mu) := \E_{s\sim\mu} [V^\pi(s)]  
\end{align}
where $\mu \in \Delta(\calS)$ is the initial state distribution. It is well-known that there exists at least one optimal policy $\pi^*$ that maximizes 
$V^\pi(s)$ for every $s\in\calS$, and hence also maximizes~\eqref{eq:RL-goal} 
for any initial distribution $\mu$  \citep[see, e.g.,][]{suttonRL}.

Policy mirror descent (PMD)~\citep{Lan_2021,Xiao_2022,schulman2015trust} instantiates the mirror descent framework~\citep{MD,nemirovski1983problem} in the policy space for maximizing the expected return~\eqref{eq:RL-goal}. Given a mirror map $h$, the Bregman divergence generated by $h$ is defined as:
\begin{align*}
    D_h(p \, \| \, q) := h(p) - h(q) - \inner{\nabla h(q)}{p-q},\quad p\in \text{dom}\,h,~q\in\text{dom}\nabla h.
\end{align*}
Generally, $h$ is taken to be a closed, proper, and convex function with $\text{int}(\text{dom} \, h) \neq \emptyset$. By the convexity of $h$, the Bregman divergence satisfies $D_h(p \, || \, q) \geq 0$. Prior work~\citep{Lan_2021,Xiao_2022} has shown that the mirror descent method equipped with a state-weighted Bregman divergence can be written in the following state-wise update:
\begin{align*}
    \forall\, s\in\calS: \quad \pi^+(\cdot|s) = \underset{p\in\Delta(\calA)}{\arg\max} \,\, \Big\{ \eta\inner{p}{Q^\pi(s,\cdot)} - D_h(p\, \| \, \pi(\cdot|s)) \Big\},
\end{align*}
where $\eta > 0$ is the update step size and $Q^\pi$ is the action value function defined by
\begin{align*}
    Q^\pi(s,a) := \underset{a_t\sim \pi(\cdot|s_t), \, s_{t+1} \sim P(\cdot|s_t, a_t)}{\E} \ls[ \sum_{t=0}^\infty \gamma^t r(s_t, a_t) \, \Big | \, (s_0, a_0) = (s,a) , \, \pi \rs].
\end{align*}

\subsection{Related work}
\paragraph{General PMD.} A basic convergence result for PMD is the global sublinear convergence of the value error at the rate of $O(1/k)$, established by~\citet{Lan_2021,Xiao_2022}. More precisely, if $h$ is Legendre, the state values of PMD converge to optimal value with the sublinear rate of $O(1/k)$ under any constant step size $\eta > 0$. Since many commonly used mirror maps, including the negative entropy and the squared Euclidean mirror map, satisfy the Legendre property, this result applies to a broad class of PMD instances. Beyond constant step sizes, \citet{Xiao_2022} showed that global linear convergence of the value function can be achieved by using geometrically increasing step-size schedules. Furthermore, a $\gamma$-rate global linear convergence was established by~\citet{Johnson_Pike-Burke_Rebeschini_2023} under adaptive step sizes.

Another line of work studies regularized PMD, where the original expected return is replaced by a regularized objective involving an additional policy regularization term, typically weighted by a coefficient $\tau > 0$.\footnote{Unless explicitly stated otherwise, we refer to PMD as the unregularized variant.} For a fixed regularization coefficient, the value functions generated by regularized PMD converge globally linearly to the corresponding regularized optimal value function under any constant step size~\citep{Lan_2021,Zhan_Cen_Huang_Chen_Lee_Chi_2021,Cen_Cheng_Chen_Wei_Chi_2022}. A notable variant is homotopic PMD (HPMD)~\citep{Li_Zhao_Lan_2022}, where the regularization coefficient is gradually decreased so that the regularized problem asymptotically approaches the original unregularized problem. Under adaptive step sizes together with a decreasing regularization coefficient, HPMD achieves global linear convergence toward the optimal unregularized value function. Furthermore, the policies generated by HPMD are shown to converge to the maximum-entropy optimal policy, reflecting the policy-selection effect induced by the vanishing regularization scheme.

\paragraph{Softmax NPG \& PQA.} When $h$ is chosen as the negative entropy $h(p) = \sum_{a\in\calA} p(a) \ln p(a)$, the induced Bregman divergence is exactly the Kullback-Leibler (KL) divergence, and the corresponding PMD update reduces to  the softmax natural policy gradient (softmax NPG) method~\citep{Agarwal_Kakade_Lee_Mahajan_2019,kakade2002npg}:
\begin{align}  \label{eq:softmax-NPG}
    \forall\, s\in\calS: \quad \pi_{k+1}(a|s) = \frac{\pi_k(a|s) \exp\big( \eta \, Q^{\pi_k}(s,a) \big)}{\sum_{a^\prime \in \calA} \pi_k(a^\prime | s) \exp \big ( \eta \, Q^{\pi_k}(s,a^\prime) \big)}.
\end{align}
 Leveraging this explicit formula, instance-specific convergence results have been derived for softmax NPG through direct analyses of the multiplicative update. A representative result states that softmax NPG enjoys local linear convergence at a rate of $O(\exp(-(1-1/\lambda) \eta \Delta))$ under a constant step size, where $\Delta>0$ denotes the MDP-dependent suboptimality gap and $\lambda \to +\infty$ along the iterations~\citep{Khodadadian_Jhunjhunwala_Varma_Maguluri_2021}. Subsequently, \citet{li2025phi} demonstrated that softmax NPG exhibits an exact asymptotic linear convergence rate of $\Theta(\exp(-\eta\Delta))$. Additionally, \citet{pg-liu} established global linear convergence for softmax NPG under a constant step size, although the resulting rate is implicit and does not yield an explicit expression in terms of problem-dependent parameters.

Another widely adopted mirror map is the squared $\ell_2$-norm, $h(p) = \frac{1}{2} \| p \|_2^2$. In this case, the Bregman divergence reduces to the squared $\ell_2$-distance, and the PMD update admits the explicit projected form:
\begin{align}  \label{eq:PQA}
    \forall\, s\in\calS: \quad \pi_{k+1}(\cdot|s) = \Proj_{\Delta(\calA)} \Big \{ \pi_k(\cdot|s) + \eta \, Q^{\pi_k}(s,\cdot) \Big \},
\end{align}
where $\Proj_{\Delta(\calA)}$ denotes the $\ell_2$-projection onto the probability simplex $\Delta(\calA)$. This method is known as projected Q-ascent (PQA)~\citep{Xiao_2022}. Similar to softmax NPG, more refined convergence guarantees can be obtained for PQA by leveraging this explicit update rule. For instance, by exploiting properties of the simplex projection, \citet{ppgliu} showed that PQA terminates in a finite number of iterations under any constant step size.

\paragraph{Policy convergence for special PMD instances and regularized variants.} Most existing PMD analyses focus on value convergence rather than policy convergence, even though value convergence alone does not imply policy convergence when multiple optimal policies exist~\citep{Li_Zhao_Lan_2022}. As already mentioned, \citet{ppgliu} demonstrated that PQA achieves finite-time convergence to an optimal policy for any constant step size. For softmax NPG,   global policy convergence was established in \citet{li2025phi} by showing that the policy trajectory is Cauchy. It is worth emphasizing that these policy convergence guarantees are restricted to specific PMD instances, and their analyses rely heavily on explicit policy update formulas. For a general mirror map $h$, \citet{Lin2022PMD-policy-convergence} showed that PMD with constant step size achieves finite-time convergence provided that $\nabla h$ exists on the entire probability simplex $\Delta(\calA)$ and $h$ is $L$-coercive. These conditions naturally cover the squared Euclidean mirror map that yields PQA, but exclude several commonly used mirror maps, either because they are not differentiable on the boundary of the simplex, as in the negative entropy mirror map used in softmax NPG, or because they fail the $L$-coercivity condition.

In a different direction, \citet{Li_Zhao_Lan_2022} proposed HPMD and showed that it converges to the maximum-entropy optimal policy $\pi_U$ under a decomposable mirror map $h(p) = \sum_i \psi(p_i)$ with $\psi$ being differentiable over $(0,+\infty)$. In particular, when $\psi$ is differentiable at $0$, HPMD reaches the optimal policy set in finitely many iterations, although its iterates may continue to move within the optimal set in order to converge to the maximum-entropy optimal policy $\pi_U$. However, since HPMD is fundamentally a variant of regularized PMD, its theoretical analysis relies crucially on the additional regularization term and the decreasing regularization schedule. These mechanisms are absent in vanilla unregularized PMD with a constant step size. Consequently, a systematic policy convergence theory for unregularized PMD under general mirror maps remains largely unresolved.

\subsection{Main contributions}
As noted above, most existing analyses of PMD focus on value convergence rather than policy convergence. Existing policy convergence guarantees for unregularized PMD are largely limited to specific instances, notably PQA~\citep{ppgliu,Lin2022PMD-policy-convergence} and softmax NPG~\citep{li2025phi}. Consequently, a unified policy convergence theory for unregularized PMD beyond these explicit-update cases remains incomplete. In this paper, we prove policy convergence of unregularized PMD with any constant step size for decomposable mirror maps of the form $h(p)=\sum_i \psi(p_i)$ under the following standard Legendre-type assumption on $\psi$ \citep{Rockafellar1970convex,bauschke1997legendre}.
\begin{assumption}  \label{ass:psi-Legendre}
    We assume that $\psi$ satisfies the following conditions:
    \begin{enumerate}[label=(\alph*)]
        \item \label{ass:psi-Legendre:dom} $\dom \, \psi \supseteq [0,1]$;
        \item $\psi$ is proper, closed, and convex;
        \item \label{ass:psi-Legendre:smooth}  $\psi$ is differentiable on $\interior \, \dom \, \psi$, and for any sequence $\{x_n\} \subseteq \interior \, \dom \, \psi$ converging to $x \in \mbox{bd} \, \dom\, \psi$, there holds $\big| \psi^\prime(x_n) \big| \to \infty$;
        \item \label{ass:psi-Legendre:convex}  $\psi$ is strictly convex on $\interior \, \dom \, \psi$.
    \end{enumerate}
\end{assumption}
Assumption~\ref{ass:psi-Legendre} is satisfied by a broad class of standard mirror maps, including
\begin{itemize}
    \item the squared function (corresponding to PQA): $\psi(x) = \frac{1}{2} x^2$, $\dom \, \psi = \mathbb{R}$, $\interior\, \dom \, \psi = \mathbb{R}$, $\mbox{bd} \, \dom\, \psi = \emptyset$;
    \item Tsallis entropy with $q > 1$: $\psi(x) = \frac{1}{q-1} |x|^q$, $\dom \, \psi = \mathbb{R}$, $\interior\, \dom \, \psi = \mathbb{R}$, $\mbox{bd} \, \dom\, \psi = \emptyset$;
    \item the negative Shannon entropy (corresponding to softmax NPG): $\psi(x) = x\ln x -x$, $\dom \, \psi = [0,+\infty)$, $\interior\, \dom \, \psi = (0, +\infty)$, $\mbox{bd} \, \dom\, \psi = \{ 0 \}$;
    \item Tsallis entropy with $0 < q < 1$: $\psi(x) = \frac{1}{q-1} x^q$, $\dom \, \psi = [0, +\infty)$, $\interior\, \dom \, \psi = (0, +\infty)$, $\mbox{bd} \, \dom\, \psi = \{ 0 \}$;
    \item Hellinger mapping: $\psi(x) = -\sqrt{1-x^2}$, $\dom \, \psi = [-1,1]$, $\interior\, \dom \, \psi = (-1, 1)$, $\mbox{bd} \, \dom\, \psi = \{ -1, 1 \}$;
    \item Fermi-Dirac entropy: $\psi(x) = x\ln x + (1-x) \ln (1-x)$, $\dom \, \psi = [0,1]$, $\interior\, \dom \, \psi = (0,1)$, $\mbox{bd} \, \dom\, \psi = \{ 0, 1 \}$.
\end{itemize}

In summary, we prove that for any decomposable mirror map $h(p)=\sum_i \psi(p_i)$ with $\psi$ satisfying Assumption~\ref{ass:psi-Legendre}, unregularized PMD with any constant step size generates a policy sequence $\{\pi_k\}_{k\ge 0}$ that admits a limit $\pi_\infty$,  and this limit is an optimal policy. This holds even when multiple optimal policies exist.
In particular,we show that the convergence behavior toward this limiting optimal policy depends significantly on the differentiability of $\psi$ at  $0$ and $1$:
\begin{itemize}
    \item \textbf{Case 1:} When $\psi$ is differentiable at both $0$ and $1$, we show that unregularized PMD with a constant step size terminates in a finite number of  iterations. More precisely, there exists a finite iteration $T$ such that $\pi_T$ is optimal and $\pi_k = \pi_T$ for all $k \geq T$. Moreover,  an explicit upper bound on $T$ is also derived. This case includes the squared mirror map $\psi(x) = \frac{1}{2} x^2$, corresponding to PQA, and the Tsallis entropy $\psi(x) = \frac{1}{q-1} |x|^q$ with $q > 1$. It is worth emphasizing that our result does not require the restrictive $L$-coercivity assumption imposed in~\citet{Lin2022PMD-policy-convergence}. This extension is substantive, since $L$-coercivity fails for  the Tsallis entropy when $1 < q < 2$.
    \item \textbf{Case 2:} When $\psi$ is non-differentiable at $0$ but differentiable at $1$, PMD does not generally terminate in finite time; instead, it converges asymptotically to a limiting optimal policy $\pi_\infty$ as $k \to \infty$. In addition, our local analysis yields matching local upper and lower bounds on the value error, as well as a local policy convergence rate characterizing the distance between $\pi_k$ and its limit $\pi_\infty$. Representative examples include the negative entropy $\psi(x) = x\ln x - x$, corresponding to softmax NPG, and the Tsallis entropy $\psi(x) = \frac{1}{q-1} x^q$ with $0 < q < 1$ on $\dom \, \psi = [0, +\infty)$.
    \item \textbf{Case 3:} When $\psi$ is differentiable at $0$ but non-differentiable at $1$, PMD exhibits an MDP-dependent dichotomy determined by whether states with unique optimal actions exist. If no state has a unique optimal action, PMD reaches an optimal policy in finitely many iterations and remains there thereafter. In contrast, if the MDP contains at least one state with a unique optimal action, finite-time termination cannot occur, and policy convergence is necessarily asymptotic. Thus, mirror maps with the same boundary behavior can lead to qualitatively different convergence behaviors depending on the optimal-action structure of the underlying MDP. A representative example is the Hellinger mapping $\psi(x)=-\sqrt{1-x^2}$.
    
    \item \textbf{Case 4:} When $\psi$ is non-differentiable at both $0$ and $1$, as in Case 2, PMD converges asymptotically to a limiting optimal policy. For this case,  local upper and lower bounds on the value error are also established, together with a local policy convergence rate characterizing the distance between $\pi_k$ and its limit $\pi_\infty$. A typical example here is the Fermi--Dirac entropy $\psi(x)=x\ln x+(1-x)\ln(1-x)$.
\end{itemize}

\begin{table}[t!]
\centering
\caption{Summary of convergence results for PMD under arbitrary constant step sizes.
All mirror maps listed below satisfy Assumption~\ref{ass:psi-Legendre}, and thus PMD equipped with any of these mirror maps generates a policy sequence that converges globally to an optimal policy in the policy domain.
Here, $\sigma\in(0,\Delta/2)$ is arbitrary,
$\mathcal S_{=1}:=\{s\in\mathcal S:|\mathcal A_s^*|=1\}$, and
$\mathcal S_{>1}:=\{s\in\mathcal S:|\mathcal A_s^*|>1\}$ ($\calA^*_s$ denotes the optimal action set of $s$, see~\eqref{eq:optimal-action-set} for the formal definition).
}
\label{tab:summary}
\vspace{0.2cm}

\begingroup
\scriptsize
\setlength{\tabcolsep}{2pt}
\renewcommand{\arraystretch}{1.45}
\renewcommand{\tabularxcolumn}[1]{m{#1}}

\begin{tabularx}{\textwidth}{
>{\raggedright\arraybackslash}m{0.18\textwidth}
>{\centering\arraybackslash}m{0.11\textwidth}
>{\centering\arraybackslash}m{0.08\textwidth}
>{\hsize=0.88\hsize\centering\arraybackslash}X
>{\hsize=0.88\hsize\centering\arraybackslash}X
>{\hsize=1.24\hsize\centering\arraybackslash}X
}
\midrule[0.12em]
\textbf{Mirror map}
& \textbf{MDP setting}
& \textbf{Finite-time}
& \textbf{Local upper bound of state value error}
& \textbf{Local lower bound of state value error}
& \textbf{Local policy convergence rate} \\
\midrule[0.12em]

\makecell[l]{\textbf{Case 1:}\\
\(\psi'(0^+)>-\infty,\ \psi'(1^-)<+\infty\)} 
& 
& Theorem~\ref{thm:case-1:finite-time}
& /
& /
& / \\
\midrule[0.08em]

\makecell[l]{Squared norm\\[-0.2em]
\(\psi(x)=\frac12x^2\)}
& /
& \cmark
& /
& /
& / \\
\midrule[0.015em]

\makecell[l]{Tsallis entropy, \(q>1\)\\[-0.2em]
\(\psi(x)=\frac{1}{q-1}|x|^q\)}
& /
& \cmark
& /
& /
& / \\

\midrule[0.12em]
\makecell[l]{\textbf{Case 2:}\\
\(\psi'(0^+)=-\infty,\ \psi'(1^-)<+\infty\)}
& 
& /
& \multicolumn{2}{c}{
Theorem~\ref{thm:case-2-local-convergence}
} 
& Theorem~\ref{thm:case-2-policy-convergence-rate} \\
\midrule[0.08em]

\makecell[l]{Shannon entropy\\[-0.2em]
\(\psi(x)=x\ln x-x\)}
& /
& \xmark
& \(O\!\left(e^{-\eta(\Delta-2\sigma)k}\right)\)
& \(\Omega\!\left(e^{-\eta(\Delta+2\sigma)k}\right)\)
& \(O\!\left(e^{-\eta(\Delta-2\sigma)k}\right)\) \\
\midrule[0.015em]

\multirow[c]{3}{*}{%
  \makecell[l]{Tsallis entropy, \(0<q<1\)\\
  \(\psi(x)=\frac{1}{q-1}x^q\)}
}
& \multirow[c]{3}{*}{/}
& \multirow[c]{3}{*}{\xmark}
& \multirow[c]{3}{*}{%
  \(O\!\left((\nicefrac{1}{k})^{\frac{1}{1-q}}\right)\)
}
& \multirow[c]{3}{*}{%
  \(\Omega\!\left((\nicefrac{1}{k})^{\frac{1}{1-q}}\right)\)
}
& \(\displaystyle s\in\mathcal S_{=1}:\;
O\!\left((\nicefrac{1}{k})^{\frac{1}{1-q}}\right)\)
\\[0.65em]
& & & & &
\(\displaystyle s\in\mathcal S_{>1}:\;
O\!\left((\nicefrac{1}{k})^{\frac{q}{1-q}}\right)\)
\\

\midrule[0.12em]
\makecell[l]{\textbf{Case 3}\\
\(\psi'(0^+)>-\infty,\ \psi'(1^-)=+\infty\)}
&
& Theorem~\ref{thm:case-3-finite-time}
& \multicolumn{2}{c}{Theorem~\ref{thm:case-3-local-convergence}} 
& Theorem~\ref{thm:case-3-policy-convergence-rate} \\
\midrule[0.08em]

\makecell[l]{Hellinger\\[-0.2em]
\(\psi(x)=-\sqrt{1-x^2}\)}
& \(\mathcal S=\mathcal S_{>1}\)
& \cmark
& /
& /
& / \\
\midrule[0.015em]

\multirow[c]{2}{*}{%
  \makecell[l]{Hellinger\\[-0.2em]
  \(\psi(x)=-\sqrt{1-x^2}\)}
}
& \multirow[c]{2}{*}{\(\mathcal S_{=1}\neq\emptyset\)}
& \multirow[c]{2}{*}{\xmark}
& \multirow[c]{2}{*}{%
  \(O\!\left(\nicefrac{1}{k^2}\right)\)
}
& \multirow[c]{2}{*}{%
  \(\Omega\!\left(\nicefrac{1}{k^2}\right)\)
}
& \(\displaystyle s\in\mathcal S_{=1}:\;
O\!\left(\nicefrac{1}{k^2}\right)\)
\\[-0.35em]
& & & & &
\(\displaystyle s\in\mathcal S_{>1}:\;
O\!\left(\nicefrac{1}{k}\right)\)
\\

\midrule[0.12em]
\makecell[l]{\textbf{Case 4}\\
\(\psi'(0^+)=-\infty,\ \psi'(1^-)=+\infty\)}
& 
& /
& \multicolumn{2}{c}{Theorem~\ref{thm:case-4-local-convergence}} 
& Theorem~\ref{thm:case-4-policy-convergence-rate} \\
\midrule[0.08em]

\makecell[l]{Fermi--Dirac entropy\\[-0.2em]
\(\psi(x)=x\ln x\)\\[-0.2em]
\(\qquad +(1-x)\ln(1-x)\)}
& /
& \xmark
& \(O\!\left(e^{-\frac12\eta(\Delta-2\sigma)k}\right)\)
& \(\Omega\!\left(e^{-\eta(\Delta+2\sigma)k}\right)\)
& \(O\!\left(e^{-\frac12\eta(\Delta-2\sigma)k}\right)\) \\

\midrule[0.12em]
\end{tabularx}
\endgroup
\end{table}

Our results for the four categories of mirror maps are summarized in Table~\ref{tab:summary}. To the best of our knowledge, this is the first work to provide a systematic and unified policy convergence analysis across all four categories of Legendre functions. We develop a new analytical framework to establish policy convergence for unregularized PMD with a constant step size, which differs substantially from existing analysis. For example, the analysis for HPMD~\citep{Li_Zhao_Lan_2022} relies  on the regularization effect and requires carefully designed time-varying step sizes and regularization coefficients. However, these mechanisms are not available in the unregularized PMD setting with a constant step size. In contrast, we first establish a general policy convergence criterion for PMD  under a summable value error condition, i.e., $\sum_k \| V^* - V^{\pi_k} \|_\infty < \infty$. We then categorize Legendre functions into four distinct cases based on the differentiability at $0$ and $1$, and develop tailored local convergence analyses by tracking dual-space policy gaps and examining the boundary behavior  of $\psi^\prime$. For Case 1, finite-time termination is established by leveraging the boundedness of $\psi'$. 
For Cases 2 and 4, we derive local upper and lower bounds on the value error and show that the local upper bounds imply
\(
    \sum_{k\ge 0} \|V^* - V^{\pi_k}\|_\infty < \infty,
\)
which then yields policy convergence through the summability criterion. 
For Case 3, we split the analysis according to whether the MDP contains any state with a unique optimal action, and then use distinct arguments for the corresponding scenarios.

\subsection{Organization of this paper}

The rest of this paper is organized as follows. Section~\ref{sec:preliminary} introduces the notation, MDP setting, and preliminary results on policy mirror descent. Section~\ref{sec:main-theorem} establishes our main policy convergence theory for unregularized PMD  with arbitrary constant step sizes. In particular, we divide the analysis into four cases according to the differentiability of the mirror map at   $0$ and $1$, and establish the corresponding finite-time termination, asymptotic convergence, and local convergence rate results. Moreover,  numerical experiments are provided in Section~\ref{sec:numerics} to verify the theoretical findings. Section~\ref{sec:proofs} contains the proofs of the main results, while the proofs of the technical lemmas are presented in Section~\ref{sec:proofs-of-lemmas}.
\section{Notations, settings, and preliminaries}
\label{sec:preliminary}
\subsection{Preliminaries for MDP}
For a discrete set $\mathcal{X}$, we use $|\mathcal{X}|$ to denote its cardinality. In this paper, we focus on finite MDPs, i.e., $|\calS| < \infty$ and $|\calA| < \infty$. For any policy $\pi$, recall  $V^\pi \in \mathbb{R}^{|\calS|}$ denotes the state value  and $Q^{\pi} \in \mathbb{R}^{|\calS||\calA|}$ denotes the action value. For ease of presentation, we assume the rewards are bounded in $[0,1]$. Under this assumption, the value functions are also non-negative and upper bounded,
\begin{align*}
    \forall\, s\in\calS, \; \forall \, a\in\calA: \quad V^\pi(s) \in \Big[ 0, \; \frac{1}{1-\gamma} \Big], \quad Q^{\pi}(s,a) \in \Big[ 0, \; \frac{1}{1-\gamma} \Big].
\end{align*}
One can verify that $V^\pi$ and $Q^\pi$ are the unique solutions to the following Bellman equations,
\begin{alignat*}{2}
    &\forall\, s\in\calS: \quad\, &&V^\pi(s) = \E_{a\sim\pi(\cdot|s) , \; s^\prime\sim P(\cdot|s,a)} [r(s,a) + \gamma \, V^{\pi}(s^\prime)], \numberthis \label{eq:V-Bellman-eq} \\
    &\forall\, s\in\calS, \; \forall\, a\in\calA: \quad &&Q^\pi(s,a) = \E_{s^\prime\sim P(\cdot|s,a), \; a^\prime\sim \pi(\cdot|s^\prime)} [r(s,a) + \gamma \, Q^{\pi}(s^\prime, a^\prime)].
\end{alignat*}
It is clear that $V^\pi(s) = \E_{a\sim\pi(\cdot|s)}[Q^\pi(s,a)]$. Now recall that $\Pi = \{ \pi \, | \, \pi(\cdot|s) \in \Delta(\calA), \; \forall\, s \}$ denotes the set of all admissible policies. The optimal state value functions and optimal action value functions are defined as
\begin{align*}
    \forall\, s\in\calS, \; \forall \, a\in\calA: \quad V^*(s) := \max_{\pi \in \Pi} \; V^\pi(s), \;\; Q^*(s,a) := \max_{\pi\in\Pi} \; Q^\pi(s,a).
\end{align*}
It has been shown that there exists at least one deterministic policy $\pi^*$ such that $V^{\pi^*} = V^*$ and $Q^{\pi^*} = Q^*$, see for example~\citet{suttonRL}. Furthermore, the optimal value functions are the unique solutions to the following optimal Bellman equations,
\begin{alignat*}{2}
    &\forall \, s\in\calS: \quad &&V^*(s) = \max\nolimits_{a \in \calA} \, \E_{s^\prime \sim P(\cdot|s,a)} \, [r(s,a) + \gamma \, V^*(s^\prime)], \\
    &\forall \, s\in\calS, \; \forall\, a\in\calA: \quad &&Q^*(s,a) = \E_{s^\prime \sim P(\cdot|s,a)} [\max\nolimits_{a^\prime \in \calA} [r(s,a) + \gamma \, Q^*(s^\prime, a^\prime)]],
\end{alignat*}
and thus there holds $V^*(s) = \max_{a\in\calA} Q^*(s,a)$. Define the advantage function by
\begin{align*}
    \forall\, \pi \in \Pi, \; \forall\, s\in\calS, \; \forall \, a\in\calA: \quad A^\pi(s,a) := Q^\pi(s,a) - V^\pi(s),
\end{align*}
and similarly define the optimal advantage function by $A^*(s,a) := Q^*(s,a) - V^*(s)$. The optimal Bellman equations imply that $A^*(s,a) \leq 0$ for all $(s,a) \in \calS \times \calA$. It is direct to verify that
\begin{align}  \label{eq:Q-V-A-bounded}
    \forall\, \pi \in \Pi: \quad \big\| Q^* - Q^\pi \big\|_\infty \leq \gamma \, \big\| V^* - V^\pi \big\|_\infty, \;\; \big\| A^* - A^\pi \big\|_\infty \leq \big\| V^* - V^\pi \big\|_\infty.
\end{align}

For each state define the optimal action set
\begin{equation}  \label{eq:optimal-action-set}
\begin{aligned}
    \forall\, s\in\calS: \quad \calA^*_s &:= \big\{ a\in\calA: \;\; Q^*(s,a) = V^*(s) = \max\nolimits_{a^\prime\in\calA} \, Q^*(s,a^\prime) \big\} \\
    &\phantom{:}= \big\{ a\in\calA: \;\; A^*(s,a) = 0 \big\}.
\end{aligned}
\end{equation}
Let $\Pi^*$ denote the set of all optimal policies. It is shown 
\begin{align}  \label{eq:optimal-policies}
    \Pi^* = \big \{ \pi \in \Pi  \, \big | \, \pi(a^\prime|s) = 0, \; \forall\, s\in\calS, \; \forall \, a^\prime \not \in \calA^*_s \big\}.
\end{align}
That is, the support of any optimal policy must lie in the optimal action sets. Therefore, when $|\calA^*_s| > 1$ for some state $s\in\calS$, the optimal policy set $\Pi^*$ contains infinitely many optimal policies, and thus the value convergence $V^{\pi_k} \to V^*$ does not imply the convergence in the policy domain. We further divide the state set $\calS$ into the states with a unique optimal action and those with multiple optimal actions:
\begin{align*}
    \calS_{=1} := \{ s\in\calS: \; \; |\calA^*_s| = 1 \}, \quad \calS_{>1} := \{ s\in\calS: \; \; |\calA^*_s| > 1 \}.
\end{align*}

Define the set of non-trivial states, i.e., states at which not all actions are optimal, by \[\tilde \calS := \{ s\in\calS: \; \calA^*_s \neq \calA \}.\] Throughout this paper, we assume $\tilde \calS \neq \emptyset$. Otherwise, every action is optimal at every state, and hence every admissible policy is optimal. The suboptimality gap is then defined as
\begin{align*}
    \Delta := \min_{s\in\tilde\calS} \, \Delta_s, \quad \mbox{where } \; \Delta_s := V^*(s) - \max_{a^\prime \not \in \calA^*_s} \, Q^*(s,a^\prime).
\end{align*}
The suboptimality gap serves as a problem-dependent measure of how difficult it is to distinguish optimal actions from suboptimal ones, and is closely tied to the tail convergence behavior of PMD~\citep{pg-liu,ppgliu,Li_Zhao_Lan_2022,Khodadadian_Jhunjhunwala_Varma_Maguluri_2021}. Another key metric in our analysis is the suboptimal probabilities,
\begin{align*}
    \forall\, \pi \in \Pi, \; \forall\, s\in\calS: \quad b^{\pi}_s := \sum_{a^\prime \not \in \calA^*_s} \pi(a^\prime | s) = 1 - \sum_{a^* \in \calA^*_s} \pi(a^*|s).
\end{align*}
\begin{remark}
    In~\citet{Li_Zhao_Lan_2022}, the authors use $\dist_{\ell_1}(\pi, \, \Pi^*) := \max_{s\in\calS} \, \inf_{\pi^*\in\Pi^*} \, \big\| \pi(\cdot|s) - \pi^*(\cdot|s) \big\|_1$ to measure the distance between $\pi$ and the optimal policy set $\Pi^*$. It can be verified that
    \begin{align*}
        \dist_{\ell_1}(\pi, \, \Pi^*) = 2\max_{s\in\calS} \, b^\pi_s.
    \end{align*}
\end{remark}
\noindent By~\eqref{eq:optimal-policies}, a policy $\pi$ is optimal if and only if $b^\pi_s = 0$ for all $s\in\calS$. The following lemma plays an important role in our analysis.
\begin{lemma}[\protect{\citet[Lemma~2.4]{pg-liu}}]  \label{lem:sub-optimal-probabilities}
    For any $s \in \calS$,
    \begin{align*}
        \Delta \cdot b^\pi_s \leq V^*(s) -V^\pi(s) \leq \frac{1}{(1-\gamma)^2} \cdot \E_{s^\prime\sim d^\pi_{s}}[b^\pi_{s^\prime}],
    \end{align*}
    where $d^\pi_{s}(s^\prime) := (1-\gamma) \sum_{t=0}^\infty \gamma^t \P[s_t = s^\prime\, | \, s_0 = s ,\, \pi]$ is the visitation measure.
\end{lemma}

\subsection{Policy mirror descent}
In this paper we consider the decomposable mirror map $h(p) = \sum_{a\in\calA} \psi(p(a))$ with $\psi$ satisfying Assumption~\ref{ass:psi-Legendre}. First note that any differentiable convex function is continuously differentiable~\citep{Rockafellar1970convex}. Therefore, Assumption~\ref{ass:psi-Legendre} implies that $\psi^\prime$ is continuous and strictly increasing over $\interior\, \dom \, \psi$.
Moreover, the associated Bregman divergence and domains decompose coordinate-wise as
\begin{align*}
    D_h (p \, \| \, q) = \sum_{a\in\calA} d_\psi(p(a) \, \| \, q(a)), \quad \interior \, \dom \, h = \big( \interior \, \dom \, \psi \big)^{|\calA|}, \quad \dom \, \nabla h = \big( \dom \, \psi^\prime \big)^{|\calA|},
\end{align*}
where
\begin{alignat*}{2}
    &\forall\, x \in \dom \, \psi, \; \forall\, y \in \dom \, \psi^\prime : \quad &&d_\psi(x \, \| \, y) := \psi(x) - \psi(y) - \psi^\prime(y) \cdot (x-y).
\end{alignat*}
Throughout this paper, we adopt the following shorthand notations for convenience:
\begin{align*}
    \forall\, p\in\Delta(\calA), \, \forall\, \pi, \, \tilde \pi \in \Pi: \quad D^p_{\pi}(s) := D_h(p \, \| \, \pi(\cdot|s)), \;\; D^\pi_{\tilde\pi}(s) := D_h(\pi(\cdot|s) \, \| \, \tilde\pi(\cdot|s)).
\end{align*}

Given a constant step size $\eta > 0$, a mirror map $h$, and an initial policy $\pi_0$ satisfying $\pi_0(a|s) \in \interior \, \dom \, \psi$ for all $(s,a) \in \calS\times\calA$, the unregularized PMD algorithm updates the policy as follows:
\begin{align}  \label{eq:PMD}
    \mbox{(PMD)} \quad \forall\, s\in\calS: \quad \pi_{k+1}(\cdot|s) = \underset{p\in\Delta(\calA)}{\arg\max} \; \Big\{ \eta \, \langle p, \; Q^{\pi_k}(s,\cdot) \rangle - D_h(p \, \| \, \pi_k(\cdot|s)) \Big\}.
\end{align}
The following lemma follows from the first-order optimality conditions of the maximization problem defining the PMD policy update~\eqref{eq:PMD}, together with the standard characterization of the normal cone to the probability simplex~\citep{Rockafellar1970convex,beck2017first}.
\begin{lemma}  \label{lem:KKT-cond}
    Let $\{\pi_k\}_{k \geq 0}$ be the policy sequence generated by PMD  with a constant step size $\eta>0$. If the initial policy satisfies $\pi_0(a|s) \in \interior\, \dom \, \psi$ for all $(s,a) \in \calS\times\calA$, then it holds that 
    \begin{align*}
        \forall\, k \geq 0, \; \forall\, (s,a) \in \calS\times\calA: \quad \pi_k(a|s) \in \interior\, \dom \, \psi.
    \end{align*}
    Furthermore,  there exist a scalar $\nu_{s} \in \mathbb{R}$ and a dual vector $\lambda_{s} \in \mathbb{R}^{|\calA|}$ such that
    \begin{align*}
        \forall \, a\in\calA: \quad\psi^\prime(\pi_{k+1}(a|s)) = \psi^\prime(\pi_k(a|s)) + \eta \, Q^{\pi_k}(s,a) + \lambda_{s}(a) + \nu_{s},
    \end{align*}
    where the multipliers satisfy $\lambda_{s}(a) \geq 0$ and the complementary slackness condition $\pi_{k+1}(a|s) \cdot \lambda_{s}(a) = 0$ for all $a\in\calA$.
\end{lemma}
\noindent Note that when $0 \not\in \interior \, \dom \, \psi$ (e.g., the negative Shannon entropy), Lemma~\ref{lem:KKT-cond} implies that 
\begin{align}  \label{eq:KKT-cond-0-not-diff}
    \forall\, k, \; \forall\, (s,a) \in \calS\times\calA: \quad \pi_k(a|s) > 0 \;\; \mbox{and } \; \psi^\prime(\pi_{k+1}(a|s)) = \psi^\prime(\pi_k(a|s)) + \eta \, Q^{\pi_k}(s,a) + \nu_{s}.
\end{align}
Combining Lemma~\ref{lem:KKT-cond} with the three-point identity of Bregman divergences, we obtain the following three-point descent lemma~\citep{Lan_2021,Xiao_2022}.
\begin{lemma}[\protect{\citet[Lemma~6]{Xiao_2022}}]  \label{lem:three-point-descent}
    Let $\{\pi_k\}_{k \geq 0}$ be the policy sequence generated by PMD. Then
    \begin{align*}
        \forall\, p \in \Delta(\calA), \; \forall\, s\in\calS: \quad \eta \, \langle \pi_{k+1}(\cdot|s) - p, \; Q^{\pi_k}(s,\cdot) \rangle \geq D^{\pi_{k+1}}_{\pi_k}(s) + D^p_{\pi_{k+1}}(s) - D^p_{\pi_k}(s).
    \end{align*}
\end{lemma}
\noindent Prior work~\citep{Lan_2021,Xiao_2022} has established that PMD enjoys a dimension-free $O(1/k)$ global sublinear convergence rate for the value error.
\begin{lemma}[\protect{\citet[Lemma~7, Theorem~8]{Xiao_2022}}] \label{lem:Xiao-sublinear}
    Let $\{ \pi_k \}_{k\geq 0}$ and $\{ V^{\pi_k} \}_{k\geq 0}$ be the policy and value sequence generated by PMD  with $\pi_0(\cdot|s) \in \big(\interior\, \dom \, h\big) \cap \Delta(\calA)$ for all $s\in\calS$, respectively. Then there hold
    \begin{alignat*}{2}  
        &\forall\, s\in\calS: \quad &&V^{\pi_{k+1}}(s) \geq V^{\pi_k}(s), \numberthis \label{eq:PMD-monotonicity-V} \\
        &\forall\, s\in\calS, \; \forall \, a\in\calA: \quad && Q^{\pi_{k+1}}(s,a) \geq Q^{\pi_k}(s,a). \numberthis \label{eq:PMD-monotonicity-Q}
    \end{alignat*}
    Furthermore, for arbitrary optimal policy $\pi^* \in \Pi^*$, it holds that 
    \begin{align*}
        \big\| V^* - V^{\pi_k} \big\|_\infty \leq \frac{1}{k+1} \Bigg( \frac{\| D^{\pi^*}_{\pi_0} \|_\infty}{\eta(1-\gamma)} + \frac{1}{(1-\gamma)^2} \Bigg),
    \end{align*}
    where $\| D^{\pi^*}_{\pi_0} \|_\infty = \max_{s\in\calS} \, D^{\pi^*}_{\pi_0}(s) < \infty$.
\end{lemma}
\noindent Lemma~\ref{lem:Xiao-sublinear} immediately implies that $V^{\pi_k} \to V^*$ for PMD if $\pi_0(\cdot|s) \in \big(\interior\, \dom \, h\big) \cap \Delta(\calA)$. Furthermore, it also yields that $b^{\pi_k}_s \to 0$ for all $s\in\calS$ together with Lemma~\ref{lem:sub-optimal-probabilities}.

\section{Policy convergence for unregularized PMD}
\label{sec:main-theorem}
We now turn to our main policy convergence results. Recall that we work with decomposable mirror maps of the form 
$h(p)=\sum_{a\in\calA}\psi(p(a))$, where $\psi$ satisfies 
Assumption~\ref{ass:psi-Legendre}. Moreover, we assume $\pi_0(a|s) \in \interior\, \dom \, \psi$ for all $(s,a) \in \calS\times\calA$.
The first result establishes a general criterion: PMD converges in the policy domain whenever the value errors are summable, i.e., $\sum_{k\geq 0} \| V^* - V^{\pi_k} \|_\infty < \infty$.
\begin{lemma}  \label{lem:summable-error-gives-policy-convergence}
    Let $\{\pi_k\}_{k\geq 0}$ and $\{ V^{\pi_k} \}_{k\geq 0}$ be the policy and  value sequences generated by PMD with a constant step size $\eta>0$. If $\sum_{k \geq 0} \| V^* - V^{\pi_k} \|_\infty < \infty$, then it converges in the policy domain, i.e.,
    \begin{align*}
        \pi_\infty := \lim_{k\to\infty} \, \pi_k \;  \mbox{exists, } \; \mbox{and } \; \pi_\infty \in \Pi^*. 
    \end{align*}
    Furthermore,
    \begin{alignat*}{2}
        &\forall\, s \in \calS, \; \forall\, a^\prime \not\in\calA^*_s: \quad && \pi_\infty(a^\prime | s) = 0; \\
        &\forall \, s\in \calS_{=1}, \; \forall \, a^* \in \calA^*_s: \quad && \pi_\infty(a^*|s) = 1; \numberthis \label{eq:policy-convergence-to-interior} \\
        &\forall\, s\in\calS_{>1}, \; \forall\, a^*\in\calA^*_s: \quad && \pi_\infty(a^*|s) \in \interior \, \dom \, \psi.
    \end{alignat*}
\end{lemma}
\noindent The detailed proof of Lemma~\ref{lem:summable-error-gives-policy-convergence} is given in Section~\ref{sec:pf:lem:summable-error-gives-policy-convergence}, and we provide a proof sketch here for clarity. First, we leverage the three-point-descent lemma to prove that $\big\{ D^{\pi^*}_{\pi_k}(s) \big\}$ is a quasi-Fej{\'e}r sequence for arbitrary optimal policy $\pi^*$, which implies that $\lim_{k\to\infty} D^{\pi^*}_{\pi_k}(s)$ exists for all $s\in\calS$. Next, we show that any clustering point of $\{ \pi_k \}$, denoted $\pi_\infty$, is an optimal policy and $\pi_\infty(a^*|s)$ with $a^*\in\calA^*_s$ lie in $\interior\, \dom\, \psi$ for every state $s$ with multiple optimal actions, i.e., $|\calA^*_s| > 1$. Finally, it is shown that $\lim_{k\to\infty} D^{\pi_\infty}_{\pi_k}(s) = 0$, which implies  $\pi_k \to \pi_\infty$ by standard properties of Bregman divergences generated by Legendre functions \citep{bauschke1997legendre}.

Lemma~\ref{lem:summable-error-gives-policy-convergence} reduces the policy convergence analysis to the task of establishing value-error summability. However, the global sublinear value error bound established by~\citet{Xiao_2022} (Lemma~\ref{lem:Xiao-sublinear}) decays at a rate of $O(1/k)$, which is insufficient to guarantee the summability of the value errors. To verify the summability condition required by Lemma~\ref{lem:summable-error-gives-policy-convergence}, we therefore need sharper local value error estimates. The following $\psi^\prime$-gap quantities are central to our local convergence analysis:
\begin{alignat*}{2} 
    \forall\, s\in\tilde\calS:& \quad && G_k(s) := \min_{a^*\in\calA^*_s} \, \psi^\prime(\pi_k(a^*|s)) - \max_{a^\prime\not\in\calA^*_s} \, \psi^\prime(\pi_k(a^\prime|s)). \numberthis \label{eq:score-gap-definition} \\
    \forall\, s\in\tilde\calS, \; \forall\, a^\prime \not\in \calA^*_s:& \quad && G_k(s,a^\prime) := \max_{a^*\in\calA^*_s} \, \psi^\prime(\pi_k(a^*|s)) - \psi^\prime(\pi_k(a^\prime|s)). \numberthis \label{eq:score-gap-action-definition} 
\end{alignat*}
The different min/max choices in \eqref{eq:score-gap-definition} and \eqref{eq:score-gap-action-definition} reflect their dual roles in the analysis. More precisely, \(G_k(s)\) uses the minimum gap to obtain a uniform bound for controlling suboptimal probabilities, while \(G_k(s,a')\) uses the maximum optimal coordinate to derive a lower bound on the decay of a specific suboptimal probability.

For $\sigma>0$, we further define the local entrance time as 
\begin{align}  \label{eq:local-time-definition}
    T(\sigma) := \min \big\{ k: \;\; \| V^* - V^{\pi_k} \|_\infty \leq \sigma \big\} = \min \big\{ k: \;\; \| V^* - V^{\pi_t} \|_\infty \leq \sigma, \; \forall\, t \geq k \big\},
\end{align}
where the second equality follows from the monotonicity of the value sequence in~\eqref{eq:PMD-monotonicity-V}. Additionally, the finiteness of $T(\sigma)$ is guaranteed by Lemma~\ref{lem:Xiao-sublinear}. The next lemma characterizes the one-step evolution of $G_k(s)$ and $G_k(s,a^\prime)$, and serves as a key building block for our local analysis. It also provides a finite-time termination criterion in terms of the suboptimal probabilities.
\begin{lemma}  \label{lem:gap-metric-increase}
    Consider PMD with a constant step size $\eta>0$. Suppose that $k \geq T(\sigma)$  for $\sigma \in (0, \Delta/2)$. 
    \begin{enumerate}[label=\textup{(\alph*)}]
        \item \label{lem:gap-metric-increase-lower} If $s\in\tilde\calS$ satisfies $b^{\pi_{k+1}}_s > 0$, then
        \begin{align*}
            G_{k+1}(s) \geq G_k(s) + \eta \, (\Delta_s - 2\sigma) \geq G_k(s) + \eta \, (\Delta - 2\sigma).
        \end{align*}
        \item \label{lem:gap-metric-increase-upper} If  $s\in\tilde\calS$ satisfies $b^{\pi_{k+1}}_s < 1$,  then for any $a^\prime \not \in \calA^*_s$,
        \begin{align*}
            G_{k+1}(s, a^\prime) \leq G_k(s, a^\prime) + \eta \, \big( \big[ V^*(s) - Q^*(s,a^\prime) \big] + 2\sigma \big).
        \end{align*}
        \item \label{lem:termination-condition} If $b^{\pi_k}_s = 0$ for some $s\in\tilde\calS$, then for all $t \geq k$ there holds $b^{\pi_t}_s = 0$. If $b^{\pi_k}_s = 0$ for all $s \in \tilde \calS$, then $\pi_k \in \Pi^*$ and $\pi_t = \pi_k$ for all $t \geq k$, i.e., PMD reaches an optimal policy at iteration $k$ and remains unchanged.
    \end{enumerate}
\end{lemma}
\noindent The proof of Lemma~\ref{lem:gap-metric-increase}, given in Section~\ref{sec:pf:lem:gap-metric-increase}, relies on the optimality condition in Lemma~\ref{lem:KKT-cond}. 

As highlighted in our main contributions, the convergence dynamics of PMD are closely tied to the differentiability of $\psi$ at $0$ and $1$. Let $\psi'(0^+)$ and $\psi'(1^-)$ denote the one-sided limits of $\psi'$ from the right and from the left, respectively. Then any function $\psi$ satisfying Assumption~\ref{ass:psi-Legendre} falls into exactly one of the following four regimes:

\begin{itemize}
    \item \textbf{Case 1 (Differentiable at $0$ and $1$):} $\{ 0,1 \} \subseteq \interior \, \dom \, \psi$. Then $\psi^\prime(0^+)=\psi'(0) > -\infty$ and $\psi^\prime(1^-) =\psi'(1)< +\infty$.
    
    \item \textbf{Case 2 (Non-differentiable at $0$, differentiable at $1$):} $0 \in \mathrm{bd} \, \dom\,\psi$ and $1 \in \interior\, \dom \, \psi$. Then $\psi^\prime(1^-) =\psi'(1)< +\infty$ while $\psi^\prime(0^+) = -\infty$.
    
    \item \textbf{Case 3 (Differentiable at $0$, non-differentiable at $1$):} $0 \in \interior\, \dom \, \psi$ and $1 \in \mathrm{bd} \, \dom\,\psi$. Then $\psi^\prime(0^+) =\psi'(0)> -\infty$ while  $\psi^\prime(1^-)= +\infty$.
    
    \item \textbf{Case 4 (Non-differentiable at  $0$ and $1$):} $0\in \mathrm{bd} \, \dom\,\psi$ and $1\in \mathrm{bd} \, \dom\,\psi$. Then $\psi^\prime(0^+) = -\infty$ and $\psi^\prime(1^-) = +\infty$. 
\end{itemize}

Lemma~\ref{lem:gap-metric-increase} is the key ingredient for analyzing the one-step behavior of PMD in each of the four cases. In the asymptotic cases, policy convergence further requires the summability of the value errors, and the following lemma provides the main tool for establishing this summability. Its proof is deferred to Section~\ref{sec:pf:lem:integral-discrimination}.
\begin{lemma}[Integral summability criterion]  \label{lem:integral-discrimination}
    Let $f$ be a function that is non-negative, strictly monotonic decreasing, and continuous over $(0,\epsilon]$ for some $\epsilon > 0$. If $f(\epsilon) < +\infty$, $\lim_{x\downarrow 0} f(x) = +\infty$, and
    \begin{align*}
        \int_0^\epsilon f(x) dx := \lim_{t\downarrow 0} \int_t^\epsilon f(x) dx < \infty,
    \end{align*}
    then
    \begin{align*}
        \sum_{k \geq 0} f^{-1} \big( c_0 + c_1 \cdot k \big) < \infty
    \end{align*}
    for any constants $c_0 \geq f(\epsilon)$, $c_1 > 0$, where $f^{-1}:[f(\epsilon),+\infty)\to(0,\epsilon]$ denotes the inverse function of $f$.
\end{lemma}

\subsection{Finite-time termination for Case 1}

\begin{theorem}[Finite-time termination for Case 1]  \label{thm:case-1:finite-time}
    Let $\{ \pi_k \}_{k \geq 0}$ be the policy sequence generated by PMD with a constant step size $\eta>0$. If $\psi^\prime(0^+) > -\infty$ and $\psi^\prime(1^-) < +\infty$, there exists a termination time $T_{\mathrm{ter}}$ for PMD, i.e.,
    \begin{align*}
        \forall\, t \geq T_{\mathrm{ter}}: \quad \pi_t = \pi_{T_{\mathrm{ter}}} \in \Pi^*.
    \end{align*}
    Furthermore, there holds
    \begin{align*}
        T_{\mathrm{ter}} \leq T(\Delta/4) + \left\lceil \frac{4(\psi^\prime(1^-) - \psi^\prime(0^+))}{\eta\Delta} \right\rceil + 1 \leq \ls\lceil\frac{4}{\Delta} \bigg( \frac{C_\psi}{\eta(1-\gamma)} + \frac{1}{(1-\gamma)^2} \bigg)\rs\rceil + \left\lceil \frac{4(\psi^\prime(1^-) - \psi^\prime(0^+))}{\eta\Delta} \right\rceil,
    \end{align*}
    where $C_\psi := 2\sup_{p\in \Delta(\calA)} |\sum_{a}\psi(p(a))| + 2\sup_{x\in[0,1]} |\psi^\prime(x)| < \infty$.
\end{theorem}
The proof proceeds by contradiction, see Section~\ref{sec:pf:thm:case-1:finite-time} for details. 
Basically, since $\psi'$ is continuous and strictly increasing on an interval containing $[0,1]$, its range over $[0,1]$ is exactly $[\psi'(0^+),\,\psi'(1^-)]$, which is bounded. 
Therefore, the $\psi'$-gap $G_k(s)$ is uniformly bounded between $\psi'(0^+)-\psi'(1^-)$ and $\psi'(1^-)-\psi'(0^+)$. 
On the other hand, Lemma~\ref{lem:gap-metric-increase}~\ref{lem:gap-metric-increase-lower} implies that, if the suboptimal probability $b^{\pi_k}_s$ never vanishes, then $G_k(s)$ must grow to $+\infty$, hence yielding a contradiction.

\begin{example}[Tsallis entropy with $q>1$]
    For Tsallis entropy with $q>1$, $\psi(x) = \frac{1}{q-1} |x|^q$, there holds $\psi^\prime(0) = 0$, $\psi^\prime(1) = \frac{q}{q-1}$, and $C_\psi = \frac{2(q+1)}{q-1}$. Thus by Theorem~\ref{thm:case-1:finite-time}, PMD under Tsallis entropy with $q>1$ terminates in at most
    \begin{align*}
        \ls\lceil \frac{4}{\Delta} \ls( \frac{2(q+1)}{\eta(1-\gamma)(q-1)} + \frac{1}{(1-\gamma)^2} \rs) \rs\rceil + \ls\lceil \frac{4q}{\eta\Delta(q-1)} \rs\rceil
    \end{align*}
    iterations.
\end{example}
\begin{example}[Squared function, PQA]
    When the mirror map is specified as the squared function $\psi(x) = \frac{1}{2} x^2$, PMD reduces to PQA. It is shown by~\citet{ppgliu} that PQA terminates in at most
    \begin{align*}
        \ls\lceil \frac{2}{\Delta} \ls( 1 + \frac{1}{\eta\Delta} \rs) \ls( \frac{1}{\eta(1-\gamma)} + \frac{1}{(1-\gamma)^2} \rs)  - 1\rs\rceil
    \end{align*}
    iterations. In contrast, by  Theorem~\ref{thm:case-1:finite-time}, the termination time of PQA is bounded by
    \begin{align*}
        \ls\lceil \frac{4}{\Delta} \ls( \frac{3}{\eta(1-\gamma)} + \frac{1}{(1-\gamma)^2} \rs) \rs\rceil + \ls\lceil \frac{4}{\eta\Delta} \rs\rceil.
    \end{align*}
    Notably, our upper bound improves over that of~\citet{ppgliu} by removing an extra multiplicative factor of $\Delta^{-1}$.
\end{example}

\begin{remark}
    {The finite-time termination result in Theorem~\ref{thm:case-1:finite-time} should be distinguished from two related finite-time results in the literature.} 
    First, \citet{Lin2022PMD-policy-convergence} showed that unregularized PMD with any constant step size can terminate in finite time if the mirror map $h$ is differentiable over the entire simplex $\Delta(\calA)$ and is $L$-coercive on $\Delta(\calA)$, i.e.,
    \begin{align*}
        \forall\, p,q \in \Delta(\calA): \quad 
        \inner{\nabla h(p)-\nabla h(q)}{p-q}
        \geq
        \frac{1}{L}
        \left\| \nabla h(p)-\nabla h(q) \right\|_2^2 .
    \end{align*}
    {Although their result applies to possibly non-decomposable mirror maps, the $L$-coercivity condition is rather restrictive.} 
    For example, {the Tsallis entropy with $1<q<2$ is not $L$-coercive when $|\mathcal{A}|\geq 2$, whereas it is covered by Theorem~\ref{thm:case-1:finite-time}.}
    Second, for decomposable mirror maps with $\psi$ being differentiable on $[0,+\infty)$, HPMD in~\citet{Li_Zhao_Lan_2022} can reach the optimal policy set $\Pi^*$ in finitely many iterations under carefully scheduled time-varying step sizes and regularization coefficients. 
    However, after reaching $\Pi^*$, the HPMD iterates generally continue to move within the optimal set {so as to asymptotically select the maximum-entropy optimal policy}. 
    Thus, HPMD does not generally terminate in the sense of becoming stationary after reaching $\Pi^*$. 
    In contrast, for unregularized PMD with a constant step size, once the policy reaches the optimal set, it remains fixed thereafter.
\end{remark}

\begin{remark}
   The inverse dependence of the termination time on the suboptimality gap $\Delta$ established in Theorem~\ref{thm:case-1:finite-time} is unavoidable. Consider a two-armed bandit instance specified by
    \begin{align*}
        \calS = \{ s_0 \}, \;\; \calA = \{ a_0, a_1 \}, \;\; \gamma = 0, \;\; r(s_0, a_0) = c, \; \; r(s_0, a_1) = 0
    \end{align*}
    where $c > 0$. It is straightforward to see that the suboptimality gap is exactly $\Delta = c$. Now consider PMD equipped with the squared function $\psi(x) = \frac{1}{2} x^2$, which corresponds to PQA. In this two-action case, the Euclidean projection onto the simplex gives
    \begin{align*}
        \pi_{k+1}(a_0 |s_0) = \min\big\{ 1, \; \pi_k(a_0|s_0) + \eta c /2 \big\},\quad \pi_{k+1}(a_1|s_0) = \max \big\{ 0, \; \pi_k(a_1|s_0) - \eta c/2 \big\}.
    \end{align*}
    Consequently, the exact termination time for this instance is 
    \begin{align*}
        T_{\mathrm{ter}} = \ls\lceil \frac{2\pi_0(a_1|s_0)}{\eta\Delta} \rs\rceil. 
    \end{align*}
\end{remark}

\subsection{Local value error bounds and policy convergence for Case~2}
In Case~2, we have $0 \not\in \interior\, \dom \, \psi$, which implies by Lemma~\ref{lem:KKT-cond} that $\pi_k(a|s) > 0$ for all $k \geq 0$ and $(s,a) \in \calS\times\calA$. Consequently, finite-time termination is impossible as we assume that $\tilde\calS \neq \emptyset$. To establish the asymptotic policy convergence, it suffices to show that $\sum_{k=0}^{\infty} \| V^* - V^{\pi_k} \|_\infty < \infty$ via Lemma~\ref{lem:summable-error-gives-policy-convergence}.
This requires a sharper local value error bound than the global $O(1/k)$ rate.


\begin{theorem}[Local bounds for value error, Case~2]  \label{thm:case-2-local-convergence}
    Let $\{\pi_k\}_{k\geq 0}$ and $\{ V^{\pi_k} \}_{k\geq 0}$ be the policy and value sequences generated by PMD with a constant step size $\eta>0$. If $\psi^\prime(0^+) = -\infty$ and $\psi^\prime(1^-) < +\infty$, for any $\sigma \in (0, \, \Delta/2)$, there exist finite iterations $T_1(\sigma)$ and $T_2(\sigma)$ such that
    \begin{alignat*}{2}
        &\forall\, k \geq T_1(\sigma) : \quad && \big\| V^* - V^{\pi_k} \big\|_\infty \leq \frac{|\calA|}{(1-\gamma)^2} \cdot \varphi^{-1} \big( k \eta\,(\Delta-2\sigma) - C_1(\sigma) \big), \\
        &\forall\, k \geq T_2(\sigma): \quad && \big\| V^* - V^{\pi_k} \big\|_\infty \geq \Delta \cdot \varphi^{-1} \big( k \eta \, (\Delta + 2\sigma) - C_2(\sigma) \big),
    \end{alignat*}
    where $C_1(\sigma)$ and $C_2(\sigma)$ are finite constants. Here, $\varphi(x) := -\psi^\prime(x)$ is a continuous and strictly decreasing function mapping $(0,1]$ onto $[-\psi^\prime(1^-), \; +\infty)$, and $\varphi^{-1}$ denotes the inverse of $\varphi$.
\end{theorem}
The detailed proof of Theorem~\ref{thm:case-2-local-convergence} is deferred to Section~\ref{sec:pf:thm:case-2-local-convergence}. Combining the local upper bound with Lemma~\ref{lem:integral-discrimination}, we obtain that the value errors are summable, i.e., $\sum_{k=0}^\infty \| V^* - V^{\pi_k} \|_\infty < \infty$, which in turn implies the policy convergence through Lemma~\ref{lem:summable-error-gives-policy-convergence}. We formalize this result in the following theorem, with its complete proof provided in Section~\ref{sec:pf:thm:case-2-policy-convergence}. 
\begin{theorem}[Policy convergence for Case~2]  \label{thm:case-2-policy-convergence}
       Let $\{\pi_k\}_{k\geq 0}$ and $\{ V^{\pi_k} \}_{k\geq 0}$ be the policy and value sequences generated by PMD with a constant step size $\eta>0$. If $\psi^\prime(0^+) = -\infty$ and $\psi^\prime(1^-) < +\infty$, the sequence of value errors is summable, namely, $\sum_{k=0}^\infty \| V^* - V^{\pi_k} \|_\infty < \infty$. Consequently, the policy sequence converges to an optimal policy $\pi_\infty$, i.e.,
    \begin{align*}
        \lim_{k \to \infty} \pi_k = \pi_\infty \in \Pi^*.
    \end{align*}
    Moreover, the limiting policy $\pi_\infty$ satisfies the properties in~\eqref{eq:policy-convergence-to-interior}.
\end{theorem}
Furthermore, if $\psi$ is twice continuously differentiable on $\interior\,\dom\,\psi$ with 
$0<\psi''(x)<+\infty$ for all $x\in\interior\,\dom\,\psi$, we can obtain 
a local policy convergence rate under Case~2, characterizing how fast $\pi_k$ 
approaches its limiting optimal policy $\pi_\infty$. This condition is satisfied 
by standard examples such as the negative Shannon entropy and the Tsallis entropy 
with $0<q<1$. The theorem below states this result, with the proof deferred to 
Section~\ref{sec:pf:thm:case-2-policy-convergence-rate}.
\begin{theorem}[Policy convergence rate for Case~2]  \label{thm:case-2-policy-convergence-rate}
    Let $\{ \pi_k \}_{k\geq 0}$ be the policy sequence generated by PMD with a constant step size $\eta>0$.  Let $\pi_\infty$ denote the limiting policy under the conditions  $\psi^\prime(0^+) = -\infty$ and $\psi^\prime(1^-) < +\infty$. If $\psi$ is twice continuously differentiable on $\interior\, \dom \, \psi$ and $0 < \psi^{\prime\prime}(x) < +\infty$ for all $x\in\interior\, \dom \, \psi$, then
    \begin{enumerate}[label=\textup{(\alph*)}]
        \item \label{thm:case-2-policy-convergence-rate-optimal=1} for any $s\in\calS_{=1}$ and $\sigma\in(0, \, \Delta/2)$, 
        \begin{align*}
            \forall\, k \geq T_1(\sigma): \quad \| \pi_k(\cdot|s) - \pi_\infty(\cdot|s)\|_\infty \leq b^{\pi_k}_s \leq |\calA| \cdot \varphi^{-1}\big( k \eta \, (\Delta - 2\sigma) - C_1(\sigma) \big),
        \end{align*}
        where $T_1(\sigma)$ and $C_1(\sigma)$ are the constants defined in Theorem~\ref{thm:case-2-local-convergence};
        \item \label{thm:case-2-policy-convergence-rate-optimal>1} for any $s\in\calS_{>1}$ and $\sigma \in (0, \, \Delta/2)$, there exists a finite iteration $T_3(\sigma)$ such that
        \begin{align*}
            \forall\, k \geq T_3(\sigma): \quad \| \pi_k(\cdot|s) - \pi_\infty(\cdot|s) \|_\infty &\leq \frac{U_s}{L_s} \cdot b^{\pi_k}_s + \frac{2\eta\,U_s}{L_s^2} \sum_{t=k}^\infty \big\| V^* - V^{\pi_t} \big\|_\infty,
        \end{align*}
        where $U_s$ and $L_s$ are strictly positive finite constants depending on the limiting policy $\pi_\infty(\cdot|s)$ and the second derivative $\psi^{\prime\prime}$.
    \end{enumerate}
\end{theorem}

\begin{example}[Negative Shannon entropy, softmax NPG]  \label{example:Shannon}
    Consider the negative Shannon entropy $\psi(x) = x\ln x - x$, under which PMD reduces to softmax NPG. We have $\psi^\prime(x) = \ln x$, $\varphi(x) = -\ln x$, and $\varphi^{-1} (y) = \exp(-y)$. Theorem~\ref{thm:case-2-local-convergence} implies that
    \begin{align*}
        \forall\, k \geq T_1(\sigma): \quad \big\| V^* - V^{\pi_k} \big\|_\infty &\leq \frac{|\calA|}{(1-\gamma)^2} \exp \big( -k\eta \, (\Delta - 2\sigma) + C_1(\sigma) \big) \lesssim \exp(-k\eta\, (\Delta - 2\sigma)), \\
        \forall\, k \geq T_2(\sigma): \quad \big\|V^* - V^{\pi_k}\big\|_\infty &\geq \Delta \cdot \exp \big( - k\eta \, (\Delta + 2\sigma) + C_2(\sigma) \big) \gtrsim \exp(-k\eta\, (\Delta + 2\sigma)). 
    \end{align*}
    Therefore, for the negative Shannon entropy, our local value error bound yields the asymptotic rate of $\exp(-k\eta\Delta)$, exactly matching the rate established in~\citet{li2025phi} for softmax NPG. 
Furthermore, Theorem~\ref{thm:case-2-policy-convergence} recovers and extends the policy convergence result of~\citet{li2025phi} within the broader Case~2 framework. Finally, Theorem~\ref{thm:case-2-policy-convergence-rate} yields a unified local policy convergence rate for all states (overall same rate for $s\in \calS_{=1}$ and $s\in \calS_{>1}$):
    \begin{align*}
        \forall\, s\in\calS: \quad \| \pi_k(\cdot|s) - \pi_\infty(\cdot|s) \|_\infty \lesssim \exp(- k\eta \, (\Delta -2\sigma)). 
    \end{align*}

\end{example}

\begin{example}[Tsallis entropy with $0 < q < 1$]  \label{example:Tsallis-q<1}
    Consider the Tsallis entropy $\psi(x) = \frac{1}{q-1} x^q$ with $0 < q < 1$ and $\dom \, \psi = [0, \, +\infty)$. We have $\psi^\prime(x) = \frac{q}{q-1} x^{q-1}$, $\varphi(x) = \frac{q}{1-q} x^{q-1}$, and the inverse function $\varphi^{-1}(y) = \Big( \frac{q}{(1-q)y} \Big)^{\frac{1}{1-q}}$. Theorem~\ref{thm:case-2-local-convergence} then implies that
    \begin{align*}
        \forall\, k \geq T_{1}(\sigma): \quad \big\| V^* - V^{\pi_k} \big\|_\infty &\leq \frac{|\calA|}{(1-\gamma)^2} \Bigg( \frac{q}{1-q} \frac{1}{k \eta \, (\Delta - 2\sigma) - C_1(\sigma)} \Big)^{\frac{1}{1-q}} \lesssim \Bigg( \frac{1}{k \eta\,(\Delta - 2\sigma)} \Bigg)^{\frac{1}{1-q}}, \\
        \forall\, k \geq T_{2}(\sigma): \quad \big\|V^* - V^{\pi_k}\big\|_\infty &\geq \Delta \cdot \Bigg( \frac{q}{1-q} \frac{1}{k \eta \, (\Delta + 2\sigma) - C_2(\sigma)} \Big)^{\frac{1}{1-q}} \gtrsim \Bigg( \frac{1}{k \eta\,(\Delta + 2\sigma)} \Bigg)^{\frac{1}{1-q}}. 
    \end{align*}
    This indicates that PMD equipped with the Tsallis entropy with $0<q<1$ has a strictly sublinear local value-error convergence rate. Moreover, the asymptotic policy convergence is also guaranteed by Theorem~\ref{thm:case-2-policy-convergence}. Since the second derivative $\psi^{\prime\prime}(x) = q x^{q-2}$ satisfies $0 < \psi^{\prime\prime}(x) < +\infty$ for all $x \in \interior\, \dom \, \psi = (0, \, +\infty)$, Theorem~\ref{thm:case-2-policy-convergence-rate} applies and yields the following local policy convergence rates:
    \begin{alignat*}{2}
        &\forall\, s\in\calS_{=1}: \quad && \| \pi_k(\cdot|s) - \pi_\infty(\cdot|s) \|_\infty \lesssim \Bigg( \frac{1}{k \eta\,(\Delta - 2\sigma)} \Bigg)^{\frac{1}{1-q}}, \\
        &\forall\, s\in\calS_{>1}: \quad && \| \pi_k(\cdot|s) - \pi_\infty(\cdot|s) \|_\infty \lesssim \frac{1}{\Delta - 2\sigma} \Bigg( \frac{1}{k \eta\,(\Delta - 2\sigma)} \Bigg)^{\frac{q}{1-q}}.
    \end{alignat*}
    Thus, the local policy convergence rate exhibits a state-dependent dichotomy: states with a unique optimal action and states with multiple optimal actions have different rate exponents.
\end{example}

\subsection{Policy convergence dichotomy for Case 3}
Recall that in Case~3,  $0 \in \interior\, \dom \, \psi$ but $1 \in \mbox{bd} \, \dom\, \psi$. The range of $\psi^\prime$ over $[0,1)$ is $[\psi^\prime(0^+), \, +\infty)$, bounded from below but unbounded from above. While Case~3 appears to be a symmetric counterpart to Case~2 in terms of  differentiability at $0$ and $1$, its convergence behavior is more subtle. Indeed, it combines features of both finite-time termination and asymptotic convergence, depending on the optimal-action structure of the underlying MDP. Specifically, for each state $s\in\calS$:
\begin{itemize}
    \item If the optimal action is unique, i.e., $|\calA^*_s|=1$, the value 
    convergence of PMD implies that $b^{\pi_k}_s \to 0$, and hence 
    $\pi_k(a^*|s) \to 1$ for the sole optimal action $a^*$. Since 
    $1 \not\in \interior\, \dom \, \psi$, Lemma~\ref{lem:KKT-cond} ensures 
    that $\pi_k(a^*|s)$ cannot reach $1$ in finite time; consequently, PMD 
    exhibits asymptotic behavior analogous to that in Case~2.
    \item If there exist multiple optimal actions, i.e., $|\calA^*_s|>1$, 
    $b^{\pi_k}_s \to 0$ no longer forces $\pi_k(a^*|s)\to 1$ for any single 
    optimal action $a^*\in\calA^*_s$. Thus, it is possible that the probabilities 
    assigned to optimal actions stay away from the boundary point $1$, and the probability sequence $\{ \pi_k(a|s) \}_{k \geq 0}$ may remain within $\interior\, \dom \, \psi$ for all $a\in\calA$. This behavior resembles that of Case~1.
\end{itemize}
Consequently, we adopt a distinct analytical strategy for Case~3. Specifically, we first derive a coarse yet summable bound on the suboptimal probabilities, which implies
\(
    \sum_k \| V^* - V^{\pi_k} \|_\infty < +\infty,
\)
and policy convergence then follows from Lemma~\ref{lem:summable-error-gives-policy-convergence}.

\begin{theorem}[Policy convergence for Case~3]  \label{thm:case-3-policy-convergence}
    Let $\{\pi_k\}_{k\geq 0}$ be the sequence of policies generated by PMD with a constant step size $\eta>0$. If $\psi^\prime(0^+) > -\infty$ and $\psi^\prime(1^-) = +\infty$, the policy sequence converges to an optimal policy $\pi_\infty$, i.e.,
    \begin{align*}
        \lim_{k \to \infty} \pi_k = \pi_\infty \in \Pi^*.
    \end{align*}
    Moreover, the limiting policy $\pi_\infty$ satisfies the properties in~\eqref{eq:policy-convergence-to-interior}. 
\end{theorem}

The detailed proof of Theorem~\ref{thm:case-3-policy-convergence} is deferred to Section~\ref{sec:pf:thm:case-3-policy-convergence}. Using this policy convergence result, we further refine the convergence behavior for specific MDP structures. For instances where every state admits multiple optimal actions (i.e., $\calS_{>1} = \calS$), we show that the dual coordinates $\psi'(\pi_k(a|s))$ are uniformly bounded above by a finite constant $\beta$, and thus $\psi^\prime(\pi_k(a|s)) \in [\psi^\prime(0^+), \, \beta]$. Then, by a contradiction argument analogous to that used in Case~1, we can establish finite-time termination for Case 3, as stated in the following theorem. The proof of this theorem is provided in Section~\ref{sec:pf:thm:case-3-finite-time}.

\begin{theorem}[Finite-time termination for Case~3]  \label{thm:case-3-finite-time}
    Let $\{ \pi_k \}_{k \geq 0}$ be the policy sequence generated by PMD with a constant step size $\eta > 0$. If $\psi^\prime(0^+) > -\infty$, $\psi^\prime(1^-) = +\infty$, and the optimal action is non-unique for all states \textup{(}i.e., $\calS_{>1} = \calS$\textup{)}, then there exists a finite time $T_{\mathrm{ter}}$ such that
    \begin{align*}
        \forall\, t \geq T_{\mathrm{ter}}: \quad \pi_t = \pi_{T_{\mathrm{ter}}} \in \Pi^*.
    \end{align*}
\end{theorem}
In contrast, for MDP instances where at least one state admits a unique optimal action, i.e., $\calS_{=1} \neq \emptyset$, the non-differentiability of $\psi$ at $1$ ensures that $\pi_k(a^*|s) < 1$ for any $s\in\calS_{=1}$. Consequently, the suboptimal probability satisfies $b_s^{\pi_k}>0$ for all finite iterations $k$, so finite-time termination cannot occur. However, this allows the gap-growth estimates in Lemma~\ref{lem:gap-metric-increase} to be applied, leading to local two-sided value error bounds analogous to those in Case~2. We formalize these bounds in the following theorem, with the  proof deferred to Section~\ref{sec:pf:thm:case-3-local-convergence}.

\begin{theorem}[Local bounds for value error, Case~3]  \label{thm:case-3-local-convergence}
    Let $\{\pi_k\}_{k\geq 0}$ and $\{ V^{\pi_k} \}_{k\geq 0}$ be the policy and value sequences  generated by PMD with a constant step size $\eta>0$. If $\psi^\prime(0^+) > -\infty$, $\psi^\prime(1^-) = +\infty$, and there exists at least one state $s\in\calS$ with a unique optimal action \textup{(}i.e., $\calS_{=1} \neq \emptyset$\textup{)}, then for any $\sigma \in (0, \, \Delta/2)$, there exist finite iterations $T_1(\sigma)$ and $T_2(\sigma)$ such that
    \begin{alignat*}{2}
        &\forall\, k \geq T_1(\sigma): \quad && \big\| V^* - V^{\pi_k} \big\|_\infty \leq \frac{1}{(1-\gamma)^2} \chi^{-1}\big( k \eta \, (\Delta_{=1} - 2\sigma) + C_1(\sigma) \big), \\
        &\forall\, k \geq T_2(\sigma): \quad && \big\| V^* - V^{\pi_k} \big\|_\infty \geq \Delta \, \chi^{-1} \big( k\eta \, (\Delta_{=1} + 2\sigma) + C_2(\sigma) \big),
    \end{alignat*}
    where $C_1(\sigma)$ and $C_2(\sigma)$ are finite constants, and $\Delta_{=1} := \min_{s\in\calS_{=1}} \, \Delta_s$. Here, $\chi(x) := \psi^\prime(1-x)$ is a continuous and strictly decreasing function mapping $(0,1]$ onto $[\psi^\prime(0^+), \, +\infty)$, and $\chi^{-1}$ denotes its inverse.
\end{theorem}

For Case~3, we can further establish a local policy convergence rate under the additional regularity condition that $\psi$ is twice continuously differentiable on $\interior\,\dom\,\psi$ and satisfies $0<\psi''(x)<+\infty$ for all $x\in\interior\,\dom\,\psi$. The proof is deferred to Section~\ref{sec:pf:thm:case-3-policy-convergence-rate}.

\begin{theorem}[Policy convergence rate for Case~3]  \label{thm:case-3-policy-convergence-rate}
    Let $\{ \pi_k \}_{k\geq 0}$ be the sequence of policies generated by PMD with a constant step size $\eta>0$. Let $\pi_\infty$ denote the limiting policy under the conditions  $\psi^\prime(0^+) > -\infty$ and $\psi^\prime(1^-) = +\infty$. Assume $\mathcal{S}_{=1}\neq \emptyset$. If $\psi$ is twice continuously differentiable on $\interior\, \dom \, \psi$ and $0 < \psi^{\prime\prime}(x) < +\infty$ for all $x\in\interior\, \dom \, \psi$, then
    \begin{enumerate}[label=\textup{(\alph*)}]
        \item \label{thm:case-3-policy-convergence-rate-optimal=1} for any $s\in\calS_{=1}$ and $\sigma\in(0, \, \Delta/2)$, 
        \begin{align*}
            \forall\, k \geq T_1(\sigma): \quad \| \pi_k(\cdot|s) - \pi_\infty(\cdot|s)\|_\infty \leq b^{\pi_k}_s \leq \chi^{-1}\big( k \eta \, (\Delta_{=1} - 2\sigma) + C_1(\sigma) \big),
        \end{align*}
        where $T_1(\sigma)$ and $C_1(\sigma)$ are the constants defined in Theorem~\ref{thm:case-3-local-convergence};
        \item \label{thm:case-3-policy-convergence-rate-optimal>1} for any $s\in\calS_{>1}$ and $\sigma \in (0, \, \Delta/2)$, there exists a finite time $T_3(\sigma)$ such that
        \begin{align*}
            \forall\, k \geq T_3(\sigma): \quad \| \pi_k(\cdot|s) - \pi_\infty(\cdot|s) \|_\infty &\leq \frac{U_s}{L_s} \cdot b^{\pi_k}_s + \frac{2\eta\,U_s}{L_s^2}(1+|\calA|) \sum_{t=k}^\infty \big\| V^* - V^{\pi_t} \big\|_\infty,
        \end{align*}
        where $U_s$ and $L_s$ are strictly positive finite constants depending on the limiting policy $\pi_\infty(\cdot|s)$ and the second derivative $\psi^{\prime\prime}$.
    \end{enumerate}
\end{theorem}


\begin{example}[Hellinger]  \label{example:Hellinger}
    Consider Hellinger mapping $\psi(x) = -\sqrt{1-x^2}$ with $\dom \, \psi = [-1,1]$. There hold $\psi^\prime(x) = \frac{x}{\sqrt{1-x^2}}$, $\chi(x) = \frac{1-x}{\sqrt{1-(1-x)^2}}$, and thus $\chi^{-1}(y) = \frac{1}{\sqrt{1+y^2} \cdot \big( \sqrt{1+y^2} + y \big)}$. The policy convergence is guaranteed by Theorem~\ref{thm:case-3-policy-convergence}. Furthermore, 
    \begin{itemize}
        \item if $\calS_{>1} = \calS$, then PMD converges to some optimal policy $\pi_\infty \in \Pi^*$ in finite iterations;
        \item if $\calS_{=1} \neq \emptyset$, noticing that $\chi^{-1}(y) \asymp y^{-2}$ as $y \to +\infty$, one obtains that for large enough $k$,
        \begin{align*}
            \bigg( \frac{1}{\eta\, (\Delta_{=1} + 2\sigma) \cdot k} \bigg)^2 \lesssim \big\| V^* - V^{\pi_k} \big\|_\infty \lesssim \bigg( \frac{1}{\eta \, (\Delta_{=1} - 2\sigma) \cdot k} \bigg)^2;
        \end{align*}
       in addition,
       since $\psi^{\prime\prime}(x) = (1-x^2)^{-\nicefrac{3}{2}}$ satisfies  $\psi^{\prime\prime}(x) \in (0, \, +\infty)$ for any $x\in\interior\, \dom \, \psi = (-1,1)$, Theorem~\ref{thm:case-3-policy-convergence-rate} implies that
        \begin{alignat*}{2}
            &\forall\, s\in\calS_{=1}: \quad &&\big\| \pi_k(\cdot|s) - \pi_\infty(\cdot|s) \big\|_\infty \lesssim \bigg( \frac{1}{\eta \, (\Delta_{=1} - 2\sigma) \cdot k} \bigg)^2 , \\
            &\forall\, s\in\calS_{>1}: \quad && \big\| \pi_k(\cdot|s) - \pi_\infty(\cdot|s) \big\|_\infty \lesssim \bigg( \frac{1}{\eta \, (\Delta_{=1} - 2\sigma)^2 \cdot k} \bigg).
        \end{alignat*}
    \end{itemize}
\end{example}

\subsection{Local value bounds and policy convergence for Case~4}
Finally, we consider Case~4, where $\psi$ is non-differentiable at both $0$ and $1$. Consequently, $\dom\,\psi=[0,1]$, and $\psi'$ is well-defined on $\interior\,\dom\,\psi=(0,1)$ with range $(-\infty,+\infty)$. Since $0\in \mathrm{bd}\,\dom\,\psi$, we can use arguments analogous to those in Case~2 to establish local two-sided value error bounds. 
\begin{theorem}[Local bounds for value error, Case~4]  \label{thm:case-4-local-convergence}
    Let $\{\pi_k\}_{k\geq 0}$ and $\{ V^{\pi_k} \}_{k\geq 0}$ be the sequences of policies and value functions generated by PMD with a constant step size $\eta>0$.  If $\psi^\prime(0^+) = -\infty$ and $\psi^\prime(1^-) = +\infty$, for any $\sigma \in (0, \, \Delta/2)$,
    \begin{alignat*}{2}
        &\forall\, k \geq T(\sigma): \quad && \big\| V^* - V^{\pi_k} \big\|_\infty   \leq \frac{|\calA|}{(1-\gamma)^2} \left\{ \varphi^{-1}\Big( \frac{1}{2} \big( k\eta\,(\Delta - 2\sigma) + C_1(\sigma) \big) \Big) + \chi^{-1}\Big( \frac{1}{2} \big( k\eta\,(\Delta - 2\sigma) + C_1(\sigma) \big) \Big) \right\}, \\[.5em]
        &\forall\, k \geq T(\sigma): \quad && \big\| V^* - V^{\pi_k} \big\|_\infty \geq \Delta \, \varphi^{-1} \big( k\eta \, (\Delta + 2\sigma) + C_2(\sigma) \big),
    \end{alignat*}
    where $\varphi(x) := -\psi^\prime(x)$, $\chi(x) := \psi^\prime(1-x)$, and $C_1(\sigma)$, $C_2(\sigma)$ are finite constants. 
\end{theorem}

The proof of Theorem~\ref{thm:case-4-local-convergence} is deferred to Section~\ref{sec:pf:thm:case-4-local-convergence}. 
By combining the local upper bound in Theorem~\ref{thm:case-4-local-convergence} with the integral discrimination argument from Lemma~\ref{lem:integral-discrimination}, we can obtain the summability of the value errors.
Then Lemma~\ref{lem:summable-error-gives-policy-convergence} implies policy convergence, as stated in the following theorem. The proof of this theorem is presented in Section~\ref{sec:pf:thm:case-4-policy-convergence}.

\begin{theorem}[Policy convergence for Case~4]  \label{thm:case-4-policy-convergence}
    Let $\{\pi_k\}_{k\geq 0}$ and $\{ V^{\pi_k} \}_{k\geq 0}$ be the policy and value sequences  generated by PMD with a constant step size $\eta>0$. If $\psi^\prime(0^+) = -\infty$ and $\psi^\prime(1^-) = +\infty$, the sequence of value errors is summable, namely, $\sum_{k=0}^\infty \| V^* - V^{\pi_k} \|_\infty < \infty$. Consequently, the policy sequence converges to an optimal policy $\pi_\infty$, i.e.,
    \begin{align*}
        \lim_{k \to \infty} \pi_k = \pi_\infty \in \Pi^*.
    \end{align*}
    Moreover, the limiting policy $\pi_\infty$ satisfies the properties in~\eqref{eq:policy-convergence-to-interior}.
\end{theorem}

The local policy convergence rate in Theorem~\ref{thm:case-2-policy-convergence-rate} for Case~2 naturally extends to Case~4. The resulting bound is stated below. 

\begin{theorem}[Policy convergence rate for Case~4]  \label{thm:case-4-policy-convergence-rate}
    Let $\{ \pi_k \}_{k\geq 0}$ be the policy sequence generated by PMD with a constant step size $\eta>0$. Let $\pi_\infty$ denote the limiting policy under the conditions $\psi^\prime(0^+) = -\infty$ and $\psi^\prime(1^-) = +\infty$. If $\psi$ is twice continuously differentiable on $\interior\, \dom \, \psi$ and $0 < \psi^{\prime\prime}(x) < +\infty$ for all $x\in\interior\, \dom \, \psi = (0,1)$, then
    \begin{enumerate}[label=\textup{(\alph*)}]
        \item \label{thm:case-4-policy-convergence-rate-optimal=1} for any $s\in\calS_{=1}$ and $\sigma\in(0, \, \Delta/2)$, 
        \begin{align*}
            \forall\, k \geq T(\sigma): \quad \| \pi_k(\cdot|s) - \pi_\infty(\cdot|s)\|_\infty \leq b^{\pi_k}_s \leq |\calA| \cdot &\Big\{ \varphi^{-1}\Big( k \cdot \frac{\eta\,(\Delta - 2\sigma)}{2} + \frac{C_1(\sigma)}{2} \Big) \\
            &+ \chi^{-1}\Big( k \cdot \frac{\eta\,(\Delta - 2\sigma)}{2} + \frac{C_1(\sigma)}{2} \Big) \Big\},
        \end{align*}
        where $C_1(\sigma)$ is the constant defined in Theorem~\ref{thm:case-4-local-convergence};
        \item \label{thm:case-4-policy-convergence-rate-optimal>1} for any $s\in\calS_{>1}$ and $\sigma \in (0, \, \Delta/2)$, there exists a finite iteration $T_3(\sigma)$ such that
        \begin{align*}
            \forall\, k \geq T_3(\sigma): \quad \| \pi_k(\cdot|s) - \pi_\infty(\cdot|s) \|_\infty &\leq \frac{U_s}{L_s} \cdot b^{\pi_k}_s + \frac{2\eta\, U_s}{L_s^2} \sum_{t=k}^\infty \big\| V^* - V^{\pi_t} \big\|_\infty,
        \end{align*}
        where $U_s$ and $L_s$ are strictly positive finite constants depending on the limiting policy $\pi_\infty(\cdot|s)$ and the second derivative $\psi^{\prime\prime}$.
    \end{enumerate}
\end{theorem}
Part~\ref{thm:case-4-policy-convergence-rate-optimal=1} of
Theorem~\ref{thm:case-4-policy-convergence-rate} follows from the uniqueness of
$\pi_\infty(\cdot|s)$ for $s\in\calS_{=1}$, together with the upper bound for $b^{\pi_k}_s$ established in the proof of
Theorem~\ref{thm:case-4-local-convergence} (see~\eqref{eq:Case-4-bound-for-b}). For
Part~\ref{thm:case-4-policy-convergence-rate-optimal>1}, since
$\pi_\infty(a^*|s) \in (0,1)$ by~\eqref{eq:policy-convergence-to-interior}, the proof follows the same argument as that of
Theorem~\ref{thm:case-2-policy-convergence-rate}. We therefore omit the details.

\begin{example}[Fermi-Dirac entropy]  \label{example:Fermi-Dirac}
    The mirror map is the Fermi-Dirac entropy $\psi(x) := x \ln x + (1-x) \ln (1-x)$ with $\dom \, \psi = [0,1]$. A direct calculation gives
    \begin{align*}
        \psi^\prime(x) = \ln \Big( \frac{x}{1-x} \Big), \quad \varphi(x) = \chi(x) = \ln \big( \frac{1-x}{x} \big), \quad \varphi^{-1}(y) = \chi^{-1}(y) = \frac{1}{e^y + 1}.
    \end{align*}
    Notice that for $y \geq 0$ there holds $\frac{e^{-y}}{2}\leq \frac{1}{e^y + 1} \leq e^{-y}$. Thus by Theorem~\ref{thm:case-4-local-convergence}, for all sufficiently large $k$,
    \begin{align*}
        \big\| V^* - V^{\pi_k} \big\|_\infty &\leq \frac{2|\calA|}{(1-\gamma)^2} \cdot \exp \Big( -k \cdot \frac{\eta\, (\Delta - 2\sigma)}{2} - \frac{C_1(\sigma)}{2} \Big) \lesssim \exp\Big( -k \cdot \frac{\eta\, (\Delta - 2\sigma)}{2} \Big), \\
        \big\| V^* - V^{\pi_k} \big\|_\infty &\geq \frac{\Delta}{2} \cdot  \exp\big( - k \cdot  \eta \, (\Delta + 2\sigma) - C_2(\sigma) \big) \gtrsim \exp(-k \cdot \eta \, (\Delta + 2\sigma)).
    \end{align*}
    Therefore, PMD equipped with the Fermi-Dirac entropy exhibits local linear value error convergence, similar to softmax NPG generated by the negative Shannon entropy. Its asymptotic policy convergence follows from Theorem~\ref{thm:case-4-policy-convergence}. Furthermore, Theorem~\ref{thm:case-4-policy-convergence-rate} yields a unified local policy convergence rate for all states (overall same rate for $s\in \calS_{=1}$ and $s\in \calS_{>1}$):
    \begin{align*}
        \forall\, s\in\calS: \quad \| \pi_k(\cdot|s)  - \pi_\infty(\cdot|s)\|_\infty \lesssim \exp(-k \eta \, (\Delta - 2\sigma) /2).
    \end{align*}
\end{example}

\subsection{Numerical experiments}\label{sec:numerics}
In this section, we present numerical experiments on two deterministic Grid World MDPs to illustrate the policy convergence of PMD under the four different cases. 

\paragraph{MDP settings.}
We consider  deterministic Grid World MDPs in our tests, each defined on a finite grid whose cells constitute the state space $\mathcal{S}$. At each non-terminal state, the agent chooses one of five actions, $\calA=\{\mathrm{up},\mathrm{down},\mathrm{left},\mathrm{right},\mathrm{stay}\}$. The action ``stay'' leaves the agent in its current cell. For each of the other four actions, if the corresponding adjacent cell lies within the grid, the agent moves to that cell; otherwise, it remains in its current cell. The agent receives zero reward on every transition that does not enter a goal state. The green cells represent the goal states, which are modeled as absorbing states. Upon entering a goal state, the agent receives a reward of $+1$. Thereafter, every action leaves the agent in the same goal state and yields zero reward. This absorbing-state formulation is equivalent to terminating the episode upon reaching a goal cell.

Two Grid World instances, illustrated in Figure~\ref{fig:Grids}, are designed to exhibit different optimal-action structures, allowing us in particular to demonstrate the MDP-dependent dichotomy arising in Case 3.  In the first instance, shown in Figure~\ref{fig:Grid-A}, each non-goal state in the first row or the first column has a unique optimal action, whereas each of the four non-goal states in the lower-right region has multiple optimal actions. In the second instance, shown in Figure~\ref{fig:Grid-B}, every  state admits multiple optimal actions, i.e., $\calS_{>1}=\calS$. For both MDPs, the discount factor is set to $\gamma=0.9$.

\begin{figure}[htbp]
    \centering
    \begin{subfigure}[b]{0.48\textwidth}
        \centering
        \begin{minipage}[c][4.2cm][c]{\textwidth}
            \centering
            \includegraphics[width=0.5\textwidth]{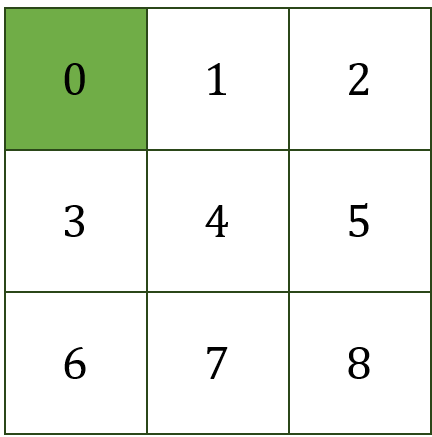}
        \end{minipage}
        \caption{Grid World A.}
        \label{fig:Grid-A}
    \end{subfigure}
    \hfill
    \begin{subfigure}[b]{0.48\textwidth}
        \centering
        \begin{minipage}[c][4.2cm][c]{\textwidth}
            \centering
            \includegraphics[width=0.7\textwidth]{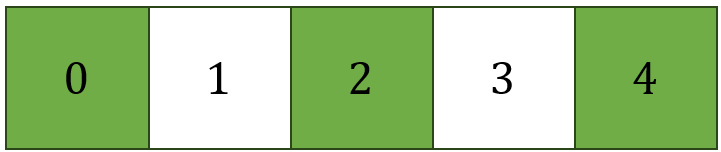}
        \end{minipage}
        \caption{Grid World B.}
        \label{fig:Grid-B}
    \end{subfigure}

    \caption{Two Grid World MDPs that used in our experiment setting. }
    \label{fig:Grids}
\end{figure}

\paragraph{Solving the PMD policy update.}
Except for special instances such as softmax NPG~\eqref{eq:softmax-NPG} and PQA~\eqref{eq:PQA}, the PMD policy update~\eqref{eq:PMD} generally does not admit a closed-form expression. We next describe a procedure for solving the maximization problem~\eqref{eq:PMD} under decomposable mirror map $h(p) = \sum_a \psi(p(a))$. For ease of notation, let $\theta_k(s,a) := \psi^\prime(\pi_k(a|s)) + \eta \, Q^{\pi_k}(s,a)$.  By Lemma~\ref{lem:KKT-cond}, the complete KKT conditions are given by
\begin{align*}\begin{cases}
    \psi^\prime(\pi_{k+1}(a|s)) = \theta_k(s,a) + \nu_s+\lambda_s(a),\\
    \sum_{a\in\mathcal{A}}\pi_{k+1}(a|s)=1,\\
    \lambda_s(a)\geq 0, \,\, \pi_{k+1}(a|s)\cdot\lambda_s(a)=0.
    \end{cases}
\end{align*}
Recall that $\psi'$ is strictly increasing on $\interior\, \dom \, \psi$.
\begin{itemize}
    \item For an action $a$ such that $\theta_k(s,a) + \nu_s>\psi'(0^+)$,  there must hold  $\lambda_s(a)=0$ and $\pi_{k+1}(a|s)>0$. Consequently, 
    \[
    \pi_{k+1}(a|s)=(\psi')^{-1}(\theta_k(s,a) + \nu_s).
    \]
    \item For action $a$ such that $\theta_k(s,a) + \nu_s<\psi'(0^+)$ (implying $\psi'(0)=\psi'(0^+)>-\infty$), there must hold $\lambda_s(a)>0$ and $\pi_{k+1}(a|s)=0$.
    \item For action $a$ such that $\theta_k(s,a) + \nu_s=\psi'(0^+)$ (implying $\psi'(0)=\psi'(0^+)>-\infty$), there must hold $\lambda_s(a)=0$ and $\pi_{k+1}(a|s)=0$.
\end{itemize}
To sum up, one has 
\[
\pi_{k+1}(a|s)=(\psi')^{-1}\left(\max\{\theta_k(s,a) + \nu_s,\; \psi'(0^+)\}\right).
\]
Plugging it back to $\sum_{a\in\mathcal{A}}\pi_{k+1}(a|s)=1$, we obtain that $\nu_s$ is a root of the following equation
\[
f(\nu)=\sum_{a\in\mathcal{A}}(\psi')^{-1}\left(\max\{\theta_k(s,a) + \nu, \;\psi'(0^+)\}\right)-1 = 0.
\]
We first show that the root of $f(\nu)$ exists. Notice that $f(\nu)$ is continuous and non-decreasing with respect to $\nu$. Now let $p(\nu,a) := (\psi')^{-1}\left(\max\{\theta_k(s,a) + \nu, \; \psi'(0^+)\}\right)$ and $a_k$ be the action that $\theta_k(s,a_k) = \max_a \theta_k(s,a)$. For a small enough $\epsilon > 0$, let $\nu_1 := \psi^\prime(\epsilon) - \max_{a} \theta_k(s,a)$. We have $p(\nu_1, a_k) = \epsilon$ and $p(\nu_1, a) \leq p(\nu_1, a_k)$ for all $a\in\calA$ by the monotonicity of $(\psi^\prime)^{-1}$, thus $f(\nu_1) \leq |\calA| \epsilon - 1 < 0$. On the other hand,
\begin{itemize}
    \item if $\psi^\prime(1^-) = +\infty$, then $p(\nu,a) \to 1$ as $\nu \to +\infty$, and thus $f(\nu) \to |\calA| - 1 > 0$ by the continuity of $(\psi^\prime)^{-1}$, i.e., there exists $\nu_2$ such that $f(\nu_2)>0$;
    \item if $\psi^\prime(1^-) < +\infty$, we have $1 \in \interior\, \dom \, \psi$ and thus there exists a small $\epsilon > 0$ such that $\psi^\prime(1+\epsilon) < +\infty$. Let $\nu_2 := \psi^\prime(1+\epsilon) - \max_a \theta_k(s,a)$. It is clear that $p(\nu_2, a_k) = 1+\epsilon \in\interior\, \dom \, \psi$, and thus $f(\nu_2) \geq p(\nu_2, a_k) - 1 \geq \epsilon > 0$.
\end{itemize}
Therefore, there always exist $\nu_1$ and $\nu_2$ such that $f(\nu_1) < 0$ and $f(\nu_2) > 0$. Together with the non-decreasing and continuous properties  of $f(\nu)$, we know that the root of $f(\nu)$ exists. It remains to show that the root is unique. Suppose that $\nu<\nu'$ are two roots.
Since every $p(\nu, a)$ is non-decreasing in $\nu$ and
\[
    \sum_a p(\nu, a)=\sum_a p(\nu', a)=1,
\]
we must have
\[
    p(\nu, a)=p(\nu', a),\qquad \forall a\in\mathcal A.
\]
However, at a root, at least one coordinate $a_0$ satisfies $p(\nu, a_0)>0$.
For this coordinate, the lower-bound truncation is inactive. Since
$(\psi')^{-1}$ is strictly increasing, one has
\[
    p(\nu', a_0)>p(\nu, a_0),
\]
yielding a contradiction. Therefore the root $\nu_s$ is unique. Consequently, $\nu_s$ can be computed by a standard bisection method on the interval $[\nu_1,\nu_2]$, after which the policy update is recovered coordinate-wise from the formula above. For all upcoming numerical experiments, the step size is fixed to $\eta = 0.1$ and the bisection accuracy is set to $|\nu - \nu_s| \leq 10^{-13}$.

\paragraph{Finite-time termination of Case~1.}
To verify the finite-time termination of Case~1, we test PMD with $\psi(x) = \frac{1}{2}x^2$ (corresponding to PQA), $\psi(x) = 2|x|^{3/2}$ (Tsallis entropy with $q = 3/2$), and $\psi(x) = \frac{1}{2} |x|^3$ (Tsallis with $q = 3$) on Grid World A and plot  $\| \pi_{k+1} - \pi_k \|_2$ against the number of iterations, see Figure~\ref{fig:Case-1}. It is obvious that $\| \pi_{k+1} - \pi_k \|_2$ converges to $0$ in finite iterations for all three PMD instances, verifying the finite-time termination results stated in Theorem~\ref{thm:case-1:finite-time}. It is worth noting that $\psi(x) = 2|x|^{3/2}$ is not $L$-coercive, and thus its finite-time termination is not covered by~\citet{Lin2022PMD-policy-convergence}. 

\begin{figure}[htbp]
    \centering
    \begin{subfigure}[b]{0.3\textwidth}
        \centering
        \begin{minipage}[c][4.2cm][c]{\textwidth}
            \centering
            \includegraphics[width=0.9\textwidth]{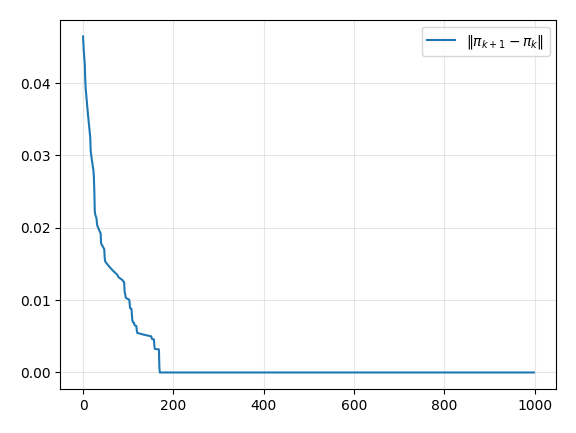}
        \end{minipage}
        \caption{$\psi(x) = \frac{1}{2} x^2$.}
        \label{fig:PQA}
    \end{subfigure}
    \hfill
    \begin{subfigure}[b]{0.3\textwidth}
        \centering
        \begin{minipage}[c][4.2cm][c]{\textwidth}
            \centering
            \includegraphics[width=0.9\textwidth]{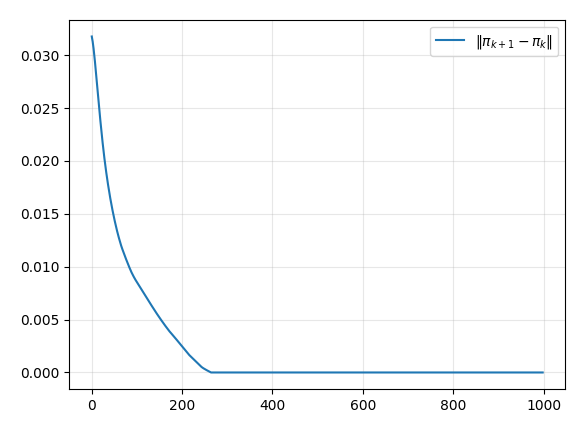}
        \end{minipage}
        \caption{$\psi(x) = 2|x|^{3/2}$.}
        \label{fig:Tsallis-q-1.5}
    \end{subfigure}
    \hfill
    \begin{subfigure}[b]{0.3\textwidth}
        \centering
        \begin{minipage}[c][4.2cm][c]{\textwidth}
            \centering
            \includegraphics[width=0.9\textwidth]{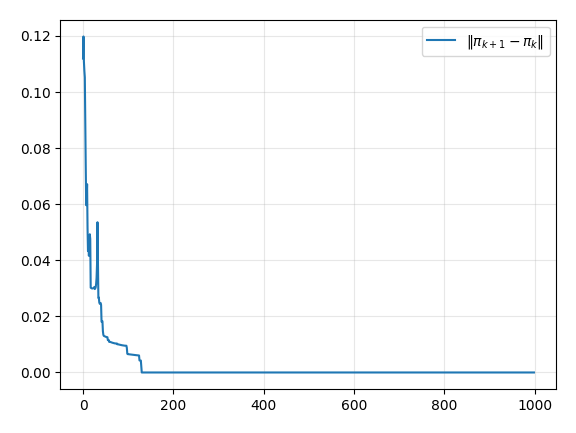}
        \end{minipage}
        \caption{$\psi(x) = \frac{1}{2}|x|^3$.}
        \label{fig:Tsallis-q-3}
    \end{subfigure}
    \caption{The plot of $\| \pi_{k+1} - \pi_k \|_2$ for Case~1, tested on Grid World A.}
    \label{fig:Case-1}
\end{figure}

\paragraph{Policy convergence of Case~2.}
To verify the policy convergence and policy convergence rates in Case~2, we consider PMD instantiated with $\psi(x)=x\ln x-x$ (corresponding to softmax NPG) and $\psi(x)=-2x^{1/2}$ (corresponding to Tsallis entropy with $q=1/2$), and test both methods on Grid World~A (Figure~\ref{fig:Grid-A}). We first run a sufficiently large number of iterations $T$, with $T=4000$ for softmax NPG and $T=10000$ for the Tsallis entropy with $q=1/2$, and use $\pi_{\mathrm{last}}:=\pi_T$ as an approximation to the limiting policy $\pi_\infty$. For the two representative states $s_1=3\in\calS_{=1}$ and $s_2=7\in\calS_{>1}$, we plot the curves of
\[
\log\bigl\|\pi_k(\cdot|s)-\pi_{\mathrm{last}}(\cdot|s)\bigr\|_2
\]
for $k=0,\ldots,\lfloor T/4\rfloor$, as shown in Figure~\ref{fig:Case-2}. The plots show that softmax NPG exhibits linear convergence (Figure~\ref{fig:softmax-NPG}), whereas PMD with the Tsallis entropy for $q=1/2$ exhibits sublinear convergence (Figure~\ref{fig:Tsallis-q-0.5}). Moreover, for the Tsallis entropy, the policy convergence rates differ between states in $\calS_{=1}$ and those in $\calS_{>1}$, in accordance with the state-dependent rate exponents predicted by our theory. These observations support the policy convergence rate results established in Theorem~\ref{thm:case-2-policy-convergence-rate} and Examples~\ref{example:Shannon} and~\ref{example:Tsallis-q<1}.
\begin{figure}[htbp]
    \centering
    \begin{subfigure}[b]{0.45\textwidth}
        \centering
        \begin{minipage}[c][4.2cm][c]{\textwidth}
            \centering
            \includegraphics[width=0.9\textwidth]{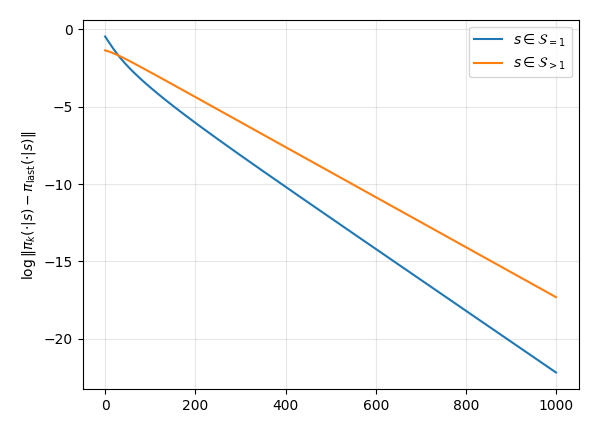}
        \end{minipage}
        \caption{$\psi(x) = x \ln x - x$.}
        \label{fig:softmax-NPG}
    \end{subfigure}
    \hfill
    \begin{subfigure}[b]{0.45\textwidth}
        \centering
        \begin{minipage}[c][4.2cm][c]{\textwidth}
            \centering
            \includegraphics[width=0.9\textwidth]{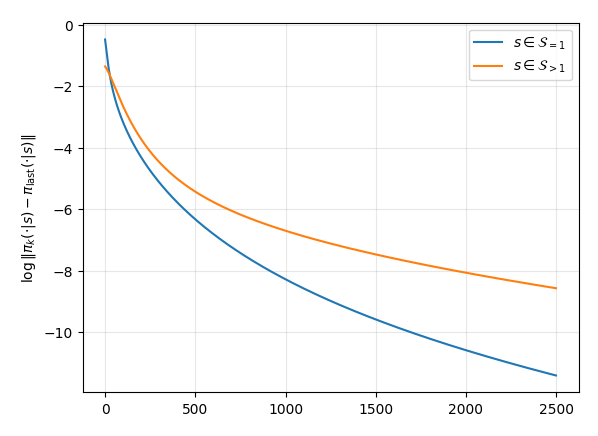}
        \end{minipage}
        \caption{$\psi(x) = -2 x^{1/2}$.}
        \label{fig:Tsallis-q-0.5}
    \end{subfigure}
    \caption{The curves of $\log \| \pi_{k}(\cdot|s) - \pi_{\mathrm{last}}(\cdot|s) \|_2$ for Case~2, tested on Grid World A (Figure~\ref{fig:Grid-A}). The blue curve shows the result on state $s_1 = 3 \in \calS_{=1}$ and the orange one corresponds to state $s_2 = 7 \in \calS_{>1}$.}
    \label{fig:Case-2}
\end{figure}

\paragraph{Policy convergence dichotomy of Case~3.} Since the policy convergence behavior in Case~3 differs depending on whether $\calS_{>1}=\calS$ or $\calS_{>1}\neq\calS$, we test PMD with the Hellinger mapping $\psi(x)=-\sqrt{1-x^2}$ on both Grid Worlds. For Grid World~A, we follow the same experimental setup as in Case~2 and plot
$\log\bigl\|\pi_k(\cdot|s)-\pi_{\mathrm{last}}(\cdot|s)\bigr\|_2$
in Figure~\ref{fig:Hellinger-A}. For Grid World~B, we instead plot $\|\pi_{k+1}-\pi_k\|_2$, as shown in Figure~\ref{fig:Hellinger-B}. The results show that PMD with the Hellinger mapping on Grid World~A, where $\calS_{=1}\neq\emptyset$, exhibits behavior similar to that of the Tsallis entropy with $q=0.5$ (Figure~\ref{fig:Tsallis-q-0.5}), supporting the sublinear local policy convergence rate in Example~\ref{example:Hellinger}. On Grid World~B, where $\calS_{>1}=\calS$, the results demonstrate finite-time termination, similar to Case~1, and confirm Theorem~\ref{thm:case-3-finite-time}.

\begin{figure}[htbp]
    \centering
    \begin{subfigure}[b]{0.45\textwidth}
        \centering
        \begin{minipage}[c][4.2cm][c]{\textwidth}
            \centering
            \includegraphics[width=0.9\textwidth]{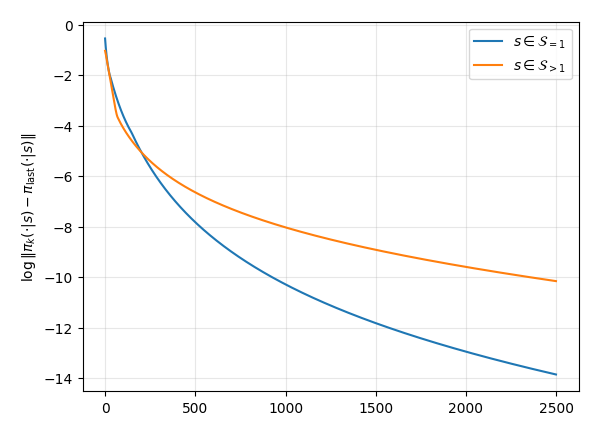}
        \end{minipage}
        \caption{$\psi(x) = -\sqrt{1-x^2}$, Grid World A.}
        \label{fig:Hellinger-A}
    \end{subfigure}
    \hfill
    \begin{subfigure}[b]{0.45\textwidth}
        \centering
        \begin{minipage}[c][4.2cm][c]{\textwidth}
            \centering
            \includegraphics[width=0.9\textwidth]{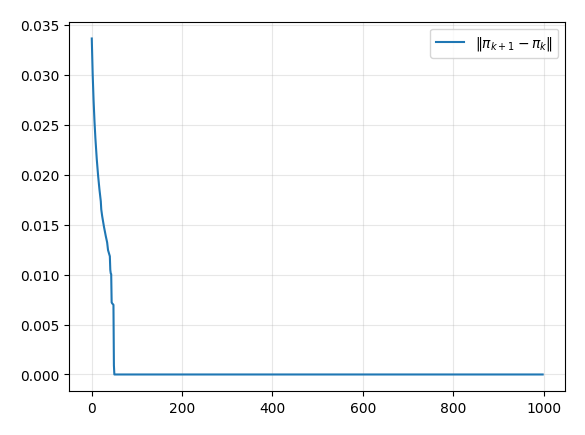}
        \end{minipage}
        \caption{$\psi(x) = -\sqrt{1-x^2}$, Grid World B.}
        \label{fig:Hellinger-B}
    \end{subfigure}
    \caption{Policy convergence results for PMD with Hellinger mapping on Grid Worlds A (left) and B (right).}
    \label{fig:Case-3}
\end{figure}

\paragraph{Policy convergence of Case~4.} We test PMD with Fermi-Dirac entropy $\psi(x) = x\ln x + (1-x) \ln (1-x)$ on Grid World A under the same settings as Case~2. The results are presented in Figure~\ref{fig:Case-4}, which clearly shows the local linear convergence for PMD with the Fermi-Dirac entropy.

\begin{figure}[htbp]
    \centering
    \includegraphics[width=0.45\textwidth]{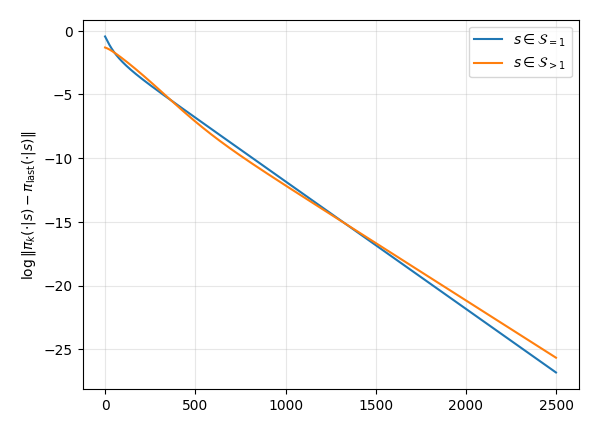}
    \caption{Policy convergence result for PMD with Fermi-Dirac entropy on Grid World A.}
    \label{fig:Case-4}
\end{figure}

\section{Proofs of main theorems}
\label{sec:proofs}

\subsection{Proof of Theorem~\ref{thm:case-1:finite-time}}
\label{sec:pf:thm:case-1:finite-time}
It suffices to show that $\pi_t = \pi_T$ for all $t \geq T$, where $T = T(\Delta/4) + \big\lceil 4(\psi^\prime(1^-) - \psi^\prime(0^+)) / (\eta\Delta) \big\rceil + 1$. We first prove that $b^{\pi_T}_s = 0$ for all $s\in\tilde \calS$ by contradiction. Suppose that there exists a state $s \in \tilde\calS$ such that $b^{\pi_T}_s > 0$. Then by Lemma~\ref{lem:gap-metric-increase}~\ref{lem:termination-condition} there holds $b^{\pi_k}_s > 0$ for all $T(\Delta/4) \leq k \leq T$, and thus one can apply Lemma~\ref{lem:gap-metric-increase}~\ref{lem:gap-metric-increase-lower} to obtain that
\begin{align*}
    G_T(s) &\geq G_{T(\Delta/4)}(s) + \frac{\eta\,\Delta}{2 } \cdot (T - T(\Delta/4)) \\
    &= G_{T(\Delta/4)}(s) + \frac{\eta\,\Delta}{2} \cdot \ls( \ls\lceil \frac{4(\psi^\prime(1^-) - \psi^\prime(0^+))}{\eta\Delta} \rs\rceil + 1 \rs) \\
    &\geq G_{T(\Delta/4)}(s) + 2 \big( \psi^\prime(1^-) - \psi^\prime(0^+) \big ) + \eta\,\Delta / 2.
\end{align*}
On the other hand, since $\pi_k(a|s)\in\interior \, \dom \, \psi$ and $\psi'$ is strictly increasing on $\interior \, \dom \, \psi$, there holds 
\[
\psi'(0^+)\leq \psi'(\pi_k(a|s))\leq \psi'(1^-).
\]
It follows that  $|G_k(s)| \leq \psi^\prime(1^-) - \psi^\prime(0^+)$ for all $k$ and $s$. Consequently,
\begin{align*}
\psi^\prime(1^-) - \psi^\prime(0^+) \geq - \big( \psi^\prime(1^-) - \psi^\prime(0^+) \big) + 2 \big( \psi^\prime(1^-) - \psi^\prime(0^+) \big) + \eta \, \Delta / 2,
\end{align*}
yielding a contradiction. Therefore, we have $b^{\pi_T}_s = 0$ for all $s \in \calS$ and thus $\pi_T \in \Pi^*$. Applying Lemma~\ref{lem:gap-metric-increase}~\ref{lem:termination-condition}, we obtain that
\begin{align*}
    \forall\, t \geq T: \quad \pi_t = \pi_T \in \Pi^*.
\end{align*}
To obtain the upper bound for $T$, leveraging Lemma~\ref{lem:Xiao-sublinear} one obtains that
\begin{align*}
    T(\Delta / 4) \leq \ls\lceil \frac{4}{\Delta} \Bigg( \frac{\| D^{\pi^*}_{\pi_0} \|_\infty}{\eta (1-\gamma)} + \frac{1}{(1-\gamma)^2} \Bigg)  \rs\rceil - 1.
\end{align*}
Notice that 
\begin{align*}
    \forall\, s\in\calS: \quad D_h(\pi^*(\cdot|s) \, \| \, \pi_0(\cdot|s)) &= h(\pi^*(\cdot|s)) - h(\pi_0(\cdot|s)) - \inner{\nabla h(\pi_0(\cdot|s))}{\pi^*(\cdot|s)-\pi_0(\cdot|s)} \\
    &\leq |h(\pi^*(\cdot|s))| + |h(\pi_0(\cdot|s))| + \| \nabla h(\pi_0(\cdot|s)) \|_\infty \cdot \| \pi^*(\cdot|s)-\pi_0(\cdot|s) \|_1 \\
    &\leq 2 \sup_{p\in\Delta(\calA)} |h(p)| + 2 \sup_{q\in\Delta(\calA)} \| \nabla h(q) \|_\infty \\
    &\leq 2\sup_{p\in\Delta(\calA)} \bigg|\sum_{a\in\calA} \psi(p(a)) \bigg| + 2\sup_{x\in[0,1]} |\psi^\prime(x)| := C_\psi,
\end{align*}
where the last line follows from $h(p) = \sum_a \psi(p(a))$. In Case~1, $\psi^\prime$ is well defined over $[0,1]$ and thus $C_\psi < \infty$. We obtain that
\begin{align*}
    T = T(\Delta/4) + \left\lceil \frac{4(\psi^\prime(1^-) - \psi^\prime(0^+))}{\eta\Delta} \right\rceil + 1 \leq \ls\lceil\frac{4}{\Delta} \bigg( \frac{C_\psi}{\eta(1-\gamma)} + \frac{1}{(1-\gamma)^2} \bigg)\rs\rceil + \left\lceil \frac{4(\psi^\prime(1^-) - \psi^\prime(0^+))}{\eta\Delta} \right\rceil.
\end{align*}

\subsection{Proof of Theorem~\ref{thm:case-2-local-convergence}}
\label{sec:pf:thm:case-2-local-convergence}
\textbf{Upper bound:} As $\psi^\prime(0^+) = -\infty$, by Assumption~\ref{ass:psi-Legendre} there holds $0 \notin \interior \, \dom \, \psi$ and thus $\pi_k(a|s) > 0$ for all $k, s, a$ (equation~\eqref{eq:KKT-cond-0-not-diff}). Therefore $b^{\pi_k}_s > 0$ for all $k, \, s\in\tilde{\mathcal{S}}$ and one can apply Lemma~\ref{lem:gap-metric-increase}~\ref{lem:gap-metric-increase-lower} to obtain that
\begin{align*}
    \forall\, k \geq T(\sigma), \; \forall\, s\in\tilde\calS: \quad G_k(s) \geq G_{T(\sigma)}(s) + \eta \,(\Delta-2\sigma) \cdot (k - T(\sigma)),
\end{align*}
which implies that
\begin{align*}
    \forall\, k \geq T(\sigma), \; \forall s\in\tilde\calS: \quad \max_{a^\prime \not\in \calA^*_s} \, \psi^\prime(\pi_k(a^\prime | s)) \leq {\min_{a^*\in\calA^*_s} \, \psi^\prime(\pi_k(a^*|s))} - G_{T(\sigma)}(s) - \eta \, (\Delta - 2\sigma) \cdot (k - T(\sigma)).
\end{align*}
By the strict increasing monotonicity of $\psi^\prime$ and $\pi_k(a|s) \in (0,1]$, there holds $\min_{a^*\in\calA^*_s} \, \psi^\prime(\pi_k(a^*|s)) \leq \psi^\prime(1^-) < +\infty$ and thus
\begin{align*}
    \forall\, k \geq T(\sigma), \; \forall\, s\in\tilde\calS , \; \forall\, a^\prime \not \in \calA^*_s: \quad \psi^\prime(\pi_k(a^\prime | s )) &\leq \underbrace{\psi^\prime(1^-) - G_{T(\sigma)}(s) + \eta \, (\Delta - 2\sigma) \cdot T(\sigma)}_{:= C_s(\sigma)} - k\cdot \eta \, (\Delta - 2\sigma) \\
    &= C_s(\sigma) - k\cdot \eta \, (\Delta - 2\sigma),
\end{align*}
where $|C_s(\sigma)| < \infty$ as $|G_{T(\sigma)}(s)| < \infty$. Noting that $\varphi := - \psi^\prime$, there holds
\begin{align*}
    \forall\, k \geq T(\sigma), \; \forall \, s\in\tilde\calS, \; \forall\, a^\prime \not \in \calA^*_s: \quad \varphi(\pi_k(a^\prime|s)) \geq -C_s(\sigma) + k \cdot \eta \, (\Delta - 2\sigma).
\end{align*}
As $\varphi$ is strictly decreasing over $\interior \, \dom \, \psi$, it has an inverse function, denoted $\varphi^{-1}: \; [-\psi^\prime(1^-), \; +\infty) \to (0,1]$ which is also strictly decreasing. Let $T_1(\sigma) \geq T(\sigma)$ be some large enough time index such that $-C_s(\sigma) + k \cdot \eta \, (\Delta - 2\sigma) \geq - \psi^\prime(1^-)$ for all $k \geq T_1(\sigma)$ and $s\in\tilde\calS$. Then applying $\varphi^{-1}$ to both side gives that
\begin{align*}
    \forall\, k \geq T_1(\sigma), \; \forall\, s\in\tilde\calS,\; \forall\, a^\prime \not \in \calA^*_s: \quad \pi_k(a^\prime | s ) &\leq \varphi^{-1} \big( k \cdot \eta \, (\Delta - 2\sigma) - C_s(\sigma) \big) \\
    &\leq \varphi^{-1} \big( k \cdot \eta \, (\Delta - 2\sigma) - C_1(\sigma) \big),
\end{align*}
where $C_1(\sigma) := \max_{s\in\tilde\calS} C_s(\sigma)$. As $b^{\pi_k}_{s^\prime} = 0$ for $s^\prime \not\in\tilde\calS$, it finally yields that
\begin{align*}
    \forall\, k \geq T_1(\sigma), \; \forall \, s\in\calS: \quad b^{\pi_k}_s = \sum_{a^\prime \not \in \calA^*_s} \pi_k(a^\prime | s) \leq |\calA| \cdot \varphi^{-1} \big( k \cdot \eta \, (\Delta - 2\sigma) - C_1(\sigma) \big).  \numberthis \label{eq:case-2-bk-local-bound}
\end{align*}
Leveraging Lemma~\ref{lem:sub-optimal-probabilities} gives the desired result.

\vspace{1em}

\noindent \textbf{Lower bound:} Pick $s_0 \in \tilde\calS$ and $a^\prime_0 \not\in\calA^*_{s_0}$ such that
\begin{align*}
    V^*(s_0) - Q^{*}(s_0, a_0^\prime) = \Delta.
\end{align*}
In Case~2 we have $0 \not \in \interior \, \dom \, \psi$ and thus $\pi_k(a|s_0) > 0$ for all $k$, $a$ (equation~\eqref{eq:KKT-cond-0-not-diff}). It implies that
\begin{align*}
    \forall\, k: \quad b^{\pi_{k}}_{s_0} < 1.
\end{align*}
Applying Lemma~\ref{lem:gap-metric-increase}~\ref{lem:gap-metric-increase-upper} with $a^\prime = a^\prime_0$ and $s = s_0$ one obtains that
\begin{align*}
    \forall\, k \geq T(\sigma): \quad G_{k+1}(s_0, a^\prime_0) \leq G_k(s_0, a^\prime_0) + \eta \, (\Delta + 2\sigma).
\end{align*}
Iterating from $T(\sigma)$ to $k-1$ gives that
\begin{align*}
    \max_{a^* \in \calA^*_{s_0}} \, \psi^\prime (\pi_{k}(a^*|s_0)) - \psi^\prime(\pi_k(a^\prime_0 | s_0)) \leq \max_{a^* \in \calA^*_{s_0}} \, \psi^\prime(\pi_{T(\sigma)}(a^*|s_0) ) - \psi^\prime(\pi_{T(\sigma)}(a^\prime_0 | s_0)) + \eta \, (\Delta + 2\sigma)\cdot (k - T(\sigma)).
\end{align*}
By Lemma~\ref{lem:sub-optimal-probabilities}, 
\begin{align*}
    \forall\, k \geq T(\sigma): \;\; \| V^* - V^{\pi_k} \|_\infty \leq \sigma \;\; \Longrightarrow \;\; b^{\pi_k}_s \leq \frac{\sigma}{\Delta},\; \forall\, s\in\calS.
\end{align*}
Thus by the strict monotonicity of $\psi^\prime$ we have
\begin{align*}
    \forall \, k \geq T(\sigma): \quad \max_{a^* \in \calA^*_{s_0}} \, \psi^\prime(\pi_k(a^*|s_0)) &= \psi^\prime \bigg( \max_{a^* \in \calA^*_{s_0}} \, \pi_k(a^*|s_0) \bigg) \\
    &\geq \psi^\prime \bigg( \frac{1}{|\calA_{s_0}^*|} \sum_{a^*\in\calA^*_{s_0}} \pi_k(a^*|s_0) \bigg) \\
    &\geq \psi^\prime \bigg(\frac{1}{|\calA|} (1-b^{\pi_k}_{s_0}) \bigg) \\
    &\geq \psi^\prime \bigg( \frac{1}{|\calA|} \Big( 1 - \frac{\sigma}{\Delta} \Big) \bigg).
\end{align*}
Substituting this bound into the preceding inequality yields
\begin{align*}
    \psi^\prime(\pi_k(a^\prime_0|s_0)) &\geq \underbrace{\psi^\prime \bigg( \frac{1}{|\calA|} \Big( 1- \frac{\sigma}{\Delta} \Big) \bigg) - \max_{a^* \in \calA^*_{s_0}} \, \psi^\prime(\pi_{T(\sigma)}(a^*|s_0) ) + \psi^\prime(\pi_{T(\sigma)}(a^\prime_0 | s_0)) + \eta \, (\Delta + 2\sigma) T(\sigma)}_{:=C_2(\sigma)} \\
    &\;\;\;\; - \eta \, (\Delta + 2\sigma) \cdot k \\[.3em]
    &= C_2(\sigma) - \eta \,(\Delta + 2\sigma) \cdot k.
\end{align*}
Let $ T_2(\sigma) \geq T(\sigma)$ be a large enough time index such that $-C_2(\sigma) + \eta \, (\Delta + 2\sigma) \cdot k \geq -\psi^\prime(1^-)$ for all $k \geq  T_2(\sigma)$. Then using $\varphi := -\psi^\prime$ one has
\begin{align*}
    \forall \, k \geq T_2(\sigma): \quad &-\psi^\prime (\pi_k(a^\prime_0 | s_0)) \leq \eta \, (\Delta + 2\sigma) \cdot k - C_2(\sigma) \\[.3em]
    \Longrightarrow \quad & \pi_k(a^\prime_0 | s_0) \geq \varphi^{-1} \big( \eta \, (\Delta + 2\sigma) \cdot k - C_2(\sigma) \big).
\end{align*}
By Lemma~\ref{lem:sub-optimal-probabilities} one obtains the desired result,
\begin{align*}
    \forall\, k \geq T_2(\sigma): \quad \big\|V^* - V^{\pi_k}\big\|_\infty \geq  V^*(s_0) - V^{\pi_k}(s_0) \geq \Delta \cdot b^{\pi_k}_{s_0} \geq \Delta \cdot \varphi^{-1} \big( \eta \, (\Delta + 2\sigma) \cdot k - C_2(\sigma) \big).
\end{align*}

\subsection{Proof of Theorem~\ref{thm:case-2-policy-convergence}}
\label{sec:pf:thm:case-2-policy-convergence}
Set $\sigma = \Delta /4$. Recall the upper bound in Theorem~\ref{thm:case-2-local-convergence},
\begin{align*}
    \forall\, k \geq T_1(\Delta/4): \quad \big\| V^* - V^{\pi_k} \big\|_\infty \leq \frac{|\calA|}{(1-\gamma)^2} \cdot \varphi^{-1} \big( k \cdot {\eta\,\Delta}/{2} - C_1(\Delta/4) \big),
\end{align*}
where $\varphi(\cdot) = -\psi^\prime(\cdot)$. Recall that $\psi'$ is continuous and strictly increasing on $\interior \, \dom\,\psi$. Thus $\varphi$ is strictly decreasing and continuous over $(0,1)$. As $\psi^\prime(0^+) = -\infty$, we have $\varphi(0^+) = +\infty$ and thus there exists $\epsilon > 0$ such that $\varphi(x) \geq 0$ over $(0,\epsilon]$. Furthermore, since $\psi$ is a convex and closed function, there holds $\psi(x)\rightarrow\psi(0)$ when $x\downarrow 0$. Consequently,
\begin{align*}
    \int_0^\epsilon \varphi(x) dx = -\int_0^\epsilon \psi^\prime(x) dx = \psi(0) - \psi(\epsilon) < \infty
\end{align*}
as $[0,1] \subseteq \dom\, \psi$. Setting $T \geq T_1(\Delta/4)$ such that $k \cdot \eta \, \Delta /2 - C_1(\Delta/4) \geq \varphi(\epsilon)$ for all $k \geq T$, applying Lemma~\ref{lem:integral-discrimination} with $f = \varphi$ implies that
\begin{align}  \label{eq:case-2-summability}
    \sum_{k \geq T} \big\| V^* - V^{\pi_k} \big\|_\infty \leq \frac{|\calA|}{(1-\gamma)^2} \sum_{k \geq T} \varphi^{-1} \big( k \cdot \eta \, \Delta / 2 - C_1(\Delta/4) \big) < \infty.
\end{align}
The proof is completed by invoking Lemma~\ref{lem:summable-error-gives-policy-convergence}.

\subsection{Proof of Theorem~\ref{thm:case-2-policy-convergence-rate}}
\label{sec:pf:thm:case-2-policy-convergence-rate}
For part~\ref{thm:case-2-policy-convergence-rate-optimal=1}, since the limiting optimal policy is unique for $s\in \calS_{=1}$, one has $\| \pi_k(\cdot|s) - \pi_\infty(\cdot|s) \|_\infty \leq b^{\pi_k}_s$ and invoking~\eqref{eq:case-2-bk-local-bound} gives the desired result. In the following we prove Theorem~\ref{thm:case-2-policy-convergence-rate}~\ref{thm:case-2-policy-convergence-rate-optimal>1}. Since $0 \not\in \interior\, \dom \, \psi$, by~\eqref{eq:policy-convergence-to-interior} one has $\pi_\infty(a^*|s) > 0$ for all $a^* \in \calA^*_s$. Therefore $\pi_\infty(a^*|s) \in (0,1)$ as there are multiple optimal actions. Now define
\begin{align*}
    \varrho_s := \frac{1}{2}\cdot \min \; \Big \{ \min_{a^*\in\calA^*} \, \pi_\infty(a^*|s), \;\; 1 - \max_{a^*\in\calA^*_s} \, \pi_\infty(a^*|s) \Big\}.
\end{align*}
Then the following region
\begin{align*}
    \mathcal{I}_s := \ls[ \min_{a^*\in\mathcal{A}_s^*}\pi_\infty(a^*|s) - \varrho_s, \; \max_{a^*\in\mathcal{A}_s^*}\pi_\infty(a^*|s) + \varrho_s  \rs] \subsetneqq \interior \, \dom \, \psi.
\end{align*}
Define
\begin{align*}
    U_s := \sup_{x\in \mathcal{I}_s} \; \psi^{\prime\prime}(x), \quad L_s := \inf_{x\in\mathcal{I}_s} \; \psi^{\prime\prime}(x).
\end{align*}
As $\psi^{\prime\prime}$ is continuous and $0 < \psi^{\prime\prime}(x) < +\infty$ for any $x\in\mathcal{I}_s$, one has $U_s$, $L_s \in (0,+\infty)$. Thus for any $x$ , $y \in \mathcal{I}_s$ there holds
\begin{align}  \label{eq:case-2-bilevel-Lipchitz}
    L_s \cdot |x - y| \leq \big|\psi^\prime(x) - \psi^\prime(y) \big| \leq U_s \cdot |x-y|.
\end{align}
Define 
\begin{align*}\psi^{\prime}(\mathcal{I}_s) := \{ \psi^\prime(x): \; x \in \mathcal{I}_s \} = \ls[ \psi^\prime \ls( \min_{a^*\in\mathcal{A}_s^*}\pi_\infty(a^*|s) - \varrho_s \rs) , \;\; \psi^\prime \ls( \max_{a^*\in\mathcal{A}_s^*}\pi_\infty(a^*|s) + \varrho_s\rs)\rs].\end{align*} 
As $\mathcal{I}_s \subsetneqq \interior \, \dom \, \psi$, $(\psi^\prime)^{-1}$ is well defined on $\psi^\prime(\mathcal{I}_s)$. Then for any $u$, $v \in \psi^\prime(\mathcal{I}_s)$ equation~\eqref{eq:case-2-bilevel-Lipchitz} gives that
\begin{align}  \label{eq:case-2-bilevel-Lipchitz-2}
    \frac{1}{U_s} |u - v| \leq \big| (\psi^\prime)^{-1}(u) - (\psi^\prime)^{-1}(v) \big| \leq \frac{1}{L_s} |u - v|.
\end{align}
Now, notice that $0 \not \in \interior \, \dom \, \psi$. By Lemma~\ref{lem:KKT-cond} we know that $\pi_k(a^*|s) > 0$ and for all $a^*\in\calA^*_s$ and $k \geq 0$. Then Lemma~\ref{lem:KKT-cond} says,
\begin{align*}
    \forall\, k \geq 0, \; \forall\, a^* \in \calA^*_s: \quad \psi^\prime(\pi_{k+1}(a^*|s)) - \psi^\prime(\pi_k(a^*|s)) = \eta \, Q^{\pi_k}(s,a^*) + \nu_s.
\end{align*}
Now pick an arbitrary optimal action $a^*\in\calA^*_s$. For any $\tilde a^* \in\calA^*_s$ and $\tilde a^* \neq a^*$, the identity above implies that
\begin{align*}
    \forall\, k \geq 0, \; &\forall\, \tilde a^* \in \calA^*_s: \quad \\
    &\psi^\prime(\pi_{k+1}(a^*|s)) - \psi^\prime(\pi_{k+1}(\tilde a^*|s)) = \psi^\prime(\pi_k(a^*|s)) - \psi^\prime(\pi_k(\tilde a^*|s)) + \eta \, \big[ Q^{\pi_k}(s,a^*) - Q^{\pi_k}(s, \tilde a^*) \big].
\end{align*}
Summing the preceding identity from $j=k$ to $t-1$ yields
\begin{align*}
    \forall\, t >  k \geq 0, \; &\forall\,  \tilde a^* \in \calA^*_s: \quad \\
    &\psi^\prime(\pi_{t}(a^*|s)) - \psi^\prime(\pi_{t}(\tilde a^*|s)) = \psi^\prime(\pi_k(a^*|s)) - \psi^\prime(\pi_k(\tilde a^*|s)) + \eta \sum_{j=k}^{t-1} \big[ Q^{\pi_j}(s,a^*) - Q^{\pi_j}(s, \tilde a^*) \big]. \numberthis \label{eq:policy-rate-case-2-telescoping}
\end{align*}
Note that $\psi^\prime$ is continuous over $\interior\, \dom\, \psi$. As $\pi_t(a^*|s) \in \interior\, \dom \, \psi$ and $\pi_t(a^*|s) \to \pi_\infty(a^*|s) \in \interior \, \dom\, \psi$, the continuity of $\psi^\prime$ gives that $\lim_{t\to\infty} \psi^\prime(\pi_t(a^*|s)) = \psi^\prime(\pi_\infty(a^*|s))$ for all $a^* \in \calA^*_s$. Thus the limit of the left-hand side is $\psi^\prime(\pi_\infty(a^*|s)) - \psi^\prime(\pi_\infty(\tilde a^*|s))$. For the limit of the right-hand side, notice that
\begin{align*}
    \forall\,\tilde a^* \in \calA^*_s: \quad &\phantom{=\,\,\,}\sum_{j=k}^{\infty} \big| Q^{\pi_j}(s,a^*) - Q^{\pi_j}(s, \tilde a^*) \big| \\
    &\leq \sum_{j=k}^\infty \big| Q^*(s,a^*) - Q^{\pi_j}(s,a^*) \big| + \big| Q^*(s, \tilde a^*) - Q^{\pi_j}(s, \tilde a^*) \big| + \big| Q^*(s, a^*) - Q^*(s, \tilde a^*) \big| \\
    &\leq 2\cdot \sum_{j=k}^\infty \big\| Q^* - Q^{\pi_j} \big\|_\infty \leq 2\cdot \sum_{j=k}^\infty \big\| V^* - V^{\pi_j} \big\|_\infty < \infty.  \numberthis \label{eq:policy-rate-case-2-bound-for-e}
\end{align*}
where the last line follows from~\eqref{eq:case-2-summability}. Thus the limit of the right-hand side also exists. Now we take the limit $t \to \infty$ in~\eqref{eq:policy-rate-case-2-telescoping} to obtain
\begin{align*}
    \forall\, k \geq 0, \; &\forall \, \tilde a^* \in \calA^*_s: \\
    &\psi^\prime(\pi_k(\tilde a^*|s)) = \psi^\prime(\pi_k(a^*|s))  + \psi^\prime(\pi_\infty(\tilde a^*|s)) - \psi^\prime(\pi_\infty(a^*|s)) + \eta \sum_{j=k}^\infty \big[ Q^{\pi_j}(s,a^*) - Q^{\pi_j}(s, \tilde a^*) \big].
\end{align*}
Notice that $\pi_k(\tilde a^*|s) \in \interior\, \dom \, \psi$. By $\pi_k(\tilde a^*|s) = (\psi^\prime)^{-1} [\psi^\prime (\pi_k(\tilde a^*|s)) ]$,
\begin{align*}
    \forall\, &k \geq 0, \; \forall \, \tilde a^* \in \calA^*_s: \\
    &\pi_k(\tilde a^* | s) = (\psi^\prime)^{-1} \bigg[ \psi^\prime(\pi_\infty(\tilde a^*|s)) + \underbrace{\psi^\prime(\pi_k(a^*|s)) - \psi^\prime(\pi_\infty(a^*|s))}_{:= c_k(s)} + \eta \underbrace{\sum_{j=k}^\infty \big[ Q^{\pi_j}(s,a^*) - Q^{\pi_j}(s, \tilde a^*) \big]}_{:= e_k(s, \tilde a^*)} \bigg].
\end{align*}
Leveraging the fact $\sum_{a^*\in\calA^*_s} \pi_\infty(a^*|s) - \pi_k(a^*|s) = b^{\pi_k}_s$,
\begin{align*}
    b^{\pi_k}_s &= \pi_\infty(a^*|s) - \pi_k(a^*|s) \\
    &\;\;\;\;+ \sum_{\tilde a^* \in \calA^*_s, \, \tilde a^* \neq a^*} (\psi^\prime)^{-1} \big[ \psi^\prime (\pi_\infty(\tilde a^*|s)) \big] - (\psi^\prime)^{-1} \Big[ \psi^\prime(\pi_\infty(a^*|s)) + c_k(s) + \eta \, e_k(s, \tilde a^*) \Big] \\
    &= (\psi^\prime)^{-1} [\psi^\prime (\pi_\infty(a^*|s))] - (\psi^\prime)^{-1} \big[ \psi^\prime(\pi_\infty(a^*|s)) + c_k(s) \big] \\
    &\;\;\;\; +  \sum_{\tilde a^* \in \calA^*_s, \, \tilde a^* \neq a^*} (\psi^\prime)^{-1} \big[ \psi^\prime (\pi_\infty(\tilde a^*|s)) \big] - (\psi^\prime)^{-1} \Big[ \psi^\prime(\pi_\infty(a^*|s)) + c_k(s) + \eta \, e_k(s, \tilde a^*) \Big] \\
    &= \underbrace{\sum_{\hat a^* \in \calA^*_s} (\psi^\prime)^{-1} [\psi^\prime (\pi_\infty(\hat a^*|s))] - (\psi^\prime)^{-1} \big[ \psi^\prime(\pi_\infty(\hat a^*|s)) + c_k(s) \big]}_{:= (\Rmnum{1})} \\
    &\;\;\;\; + \underbrace{\sum_{\tilde a^* \in \calA^*_s, \, \tilde a^* \neq a^*} (\psi^\prime)^{-1} \big[ \psi^\prime (\pi_\infty(\tilde a^*|s)) + c_k(s) \big] - (\psi^\prime)^{-1} \Big[ \psi^\prime(\pi_\infty(\tilde a^*|s)) + c_k(s) + \eta \, e_k(s, \tilde a^*) \Big]}_{:= (\Rmnum{2})}.
\end{align*}
Thus we obtain that
\begin{align*}
    |(\Rmnum{1})| \leq b^{\pi_k}_s + |(\Rmnum{2})|.
\end{align*}
We next bound the terms $|(\Rmnum{1})|$ and $|(\Rmnum{2})|$ separately. For $|(\Rmnum{1})|$, first notice that for any $s\in\calS$, $\lim_{k\to\infty} \psi^\prime(\pi_k(a^*|s)) = \psi^\prime(\pi_\infty(a^*|s))$, giving that $c_k(s) \to 0$. Furthermore, $\lim_{k\to\infty}e_k(s, \tilde a^*) = 0$ for all $\tilde a^* \in \calA^*_s$. Thus there exists a finite time $T_3(\sigma) \geq T_1(\sigma)$ such that for all $s\in\calS$, \begin{align*}\forall\, k \geq T_3(\sigma), \; \forall\, \tilde a^*\in\calA^*_s: \quad &|c_k(s)| + \eta\,|e_k(s, \tilde a^*)| \\
&\leq \min \, \Bigg\{ \psi^\prime \ls( \min_{a^*\in\mathcal{A}_s^*}\pi_\infty(a^*|s) \rs) - \psi^\prime \ls( \min_{a^*\in\mathcal{A}_s^*}\pi_\infty(a^*|s) - \varrho_s \rs), \\
&\;\;\;\;\;\;\;\;\psi^\prime \ls( \max_{a^*\in\mathcal{A}_s^*}\pi_\infty(a^*|s) + \varrho_s \rs) - \psi^\prime \ls( \max_{a^*\in\mathcal{A}_s^*}\pi_\infty(a^*|s) \rs) \Bigg\}, \end{align*}
so that $\psi^\prime(\pi_\infty(\hat a^*|s)) + c_k(s) \in \psi^\prime(\mathcal{I}_s)$. Then we can apply the two-sided bound~\eqref{eq:case-2-bilevel-Lipchitz-2}
\begin{align*}
    |(\Rmnum{1})| &= \bigg| \sum_{\hat a^* \in \calA^*_s} (\psi^\prime)^{-1} [\psi^\prime (\pi_\infty(\hat a^*|s))] - (\psi^\prime)^{-1} \big[ \psi^\prime(\pi_\infty(\hat a^*|s)) + c_k(s) \big] \bigg| \\
    &= \sum_{\hat a^* \in \calA^*_s} \Big| (\psi^\prime)^{-1} [\psi^\prime (\pi_\infty(\hat a^*|s))] - (\psi^\prime)^{-1} \big[ \psi^\prime(\pi_\infty(\hat a^*|s)) + c_k(s) \big] \Big| \\
    &\geq \frac{1}{U_s} \sum_{\hat a^* \in \calA^*_s} |c_k(s)| = \frac{|\calA^*_s|}{U_s} |c_k(s)|,
\end{align*}
where the second equality is due to the fact that the difference $(\psi^\prime)^{-1} \Big[ \psi^\prime\big( \pi_\infty(\hat a^*|s) \big) \Big] - (\psi^\prime)^{-1} \Big[ \psi^\prime(\pi_\infty(\hat a^*|s)) + c_k(s) \Big]$ consistently takes the opposite sign of $c_k(s)$ for all $\hat a^*\in\calA^*_s$. For $|(\Rmnum{2})|$, similarly applying the two-sided bound~\eqref{eq:case-2-bilevel-Lipchitz-2} yields that
\begin{align*}
    |(\Rmnum{2})| &= \bigg|\sum_{\tilde a^* \in \calA^*_s, \, \tilde a^* \neq a^*} (\psi^\prime)^{-1} \big[ \psi^\prime (\pi_\infty(\tilde a^*|s)) + c_k(s) \big] - (\psi^\prime)^{-1} \Big[ \psi^\prime(\pi_\infty(\tilde a^*|s)) + c_k(s) + \eta \, e_k(s, \tilde a^*) \Big]\bigg| \\
    &\leq \sum_{\tilde a^* \in \calA^*_s, \, \tilde a^* \neq a^*} \bigg| (\psi^\prime)^{-1} \big[ \psi^\prime (\pi_\infty(\tilde a^*|s)) + c_k(s) \big] - (\psi^\prime)^{-1} \Big[ \psi^\prime(\pi_\infty(\tilde a^*|s)) + c_k(s) + \eta \, e_k(s, \tilde a^*) \Big]\bigg| \\
    &\leq \frac{1}{L_s} \sum_{\tilde a^* \in \calA^*_s, \, \tilde a^* \neq a^*} \eta \, \big| e_k(s, \tilde a^*) \big| \leq \frac{2\eta(|\calA^*_s| - 1)}{L_s} \sum_{j=k}^\infty \big\| V^* - V^{\pi_j} \big\|_\infty,
\end{align*}
where the last line follows from~\eqref{eq:policy-rate-case-2-bound-for-e}. Combining together with the bound for $|(\Rmnum{1})|$ we get
\begin{align*}
    \forall\, k \geq T_3(\sigma): \quad  &\frac{|\calA^*_s|}{U_s} |c_k(s)| \leq b_s^{\pi_k} + \frac{2\eta(|\calA^*_s| - 1)}{L_s} \sum_{j=k}^\infty \big\| V^* - V^{\pi_j} \big\|_\infty \\
    \Longrightarrow \quad & |c_k(s)| \leq U_s \cdot b^{\pi_k}_s + \frac{2\eta \,U_s}{L_s} \sum_{j=k}^\infty \big\| V^* - V^{\pi_j} \big\|_\infty.
\end{align*}
Notice that $|c_k(s)| = |\psi^\prime(\pi_k(a^*|s)) - \psi^\prime(\pi_\infty(a^*|s))|$. Applying the two-sided bound~\eqref{eq:case-2-bilevel-Lipchitz},
\begin{align*}
    \forall\, k \geq T_3(\sigma): \quad|\pi_k(a^*|s) - \pi_\infty(a^*|s)| \leq \frac{1}{L_s} |c_k(s)| \leq \frac{U_s}{L_s} \cdot b^{\pi_k}_s + \frac{2\eta \,U_s}{L_s^2} \sum_{j=k}^\infty \big\| V^* - V^{\pi_j} \big\|_\infty.
\end{align*}
As the selection of $a^*\in\calA^*_s$ is arbitrary, we obtain that
\begin{align*}
    \forall\, k \geq T_3(\sigma), \; \forall\, a^*\in\calA^*_s: \quad |\pi_k(a^*|s) - \pi_\infty(a^*|s)| \leq \frac{U_s}{L_s} \cdot b^{\pi_k}_s + \frac{2\eta \,U_s}{L_s^2} \sum_{j=k}^\infty \big\| V^* - V^{\pi_j} \big\|_\infty.
\end{align*}
Noticing that $|\pi_k(a^\prime|s) - \pi_\infty(a^\prime|s)| = \pi_k(a^\prime|s) \leq b^{\pi_k}_s$ for any $a^\prime \not \in \calA^*_s$, the proof is completed by $\| \pi_k(\cdot|s) - \pi_\infty(\cdot|s) \|_\infty = \max_{a\in\calA} |\pi_k(a|s) - \pi_\infty(a|s)|$.

\subsection{Proof of Theorem~\ref{thm:case-3-policy-convergence}}
\label{sec:pf:thm:case-3-policy-convergence}
It suffices to prove that $\sum_k \| V^* - V^{\pi_k} \|_\infty < \infty$ for Case~3 and then apply Lemma~\ref{lem:summable-error-gives-policy-convergence}. To this end, for arbitrary $s \in \tilde\calS$, 
\begin{itemize}
    \item if there exists a finite time $T_0 \geq T(\Delta/4)$ such that $b^{\pi_{T_0}}_{s} = 0$, then by Lemma~\ref{lem:gap-metric-increase}~\ref{lem:termination-condition} we have $b^{\pi_k}_{s} = 0$ for all $k \geq T_0$ and thus $\sum_{k \geq 0} b^{\pi_k}_{s} < +\infty$;
    \item if such finite time does not exist, then we have $b^{\pi_k}_{s} > 0$ for all $k \geq T(\Delta/4)$ and thus Lemma~\ref{lem:gap-metric-increase}~\ref{lem:gap-metric-increase-lower} implies that
    \begin{align}  \label{eq:case-3-psi-score-gap}
        \forall\, k \geq T(\Delta/4): \quad G_{k}(s) \geq G_{T(\Delta/4)}(s) + \eta \, \Delta /2 \cdot (k - T(\Delta/4)).
    \end{align}
    Notice that $1-b^{\pi_k}_{s} \geq \pi_k(a^*|s)$ for any $a^* \in \calA^*_{s}$. Thus the inequality above implies that
    \begin{align*}
        \forall\, k \geq T(\Delta/4): \quad \psi^\prime \big( 1 - b^{\pi_k}_{s} \big) &\geq \psi^\prime \Big( {\min_{a^*\in\calA^*_{s}} \, \pi_k(a^*|s)} \Big) \\
        &= {\min_{a^*\in\calA^*_{s}} \, \psi^\prime(\pi_k(a^*|s))} \\
        &\geq \max_{a^\prime \not \in \calA^*_{s}} \, \psi^\prime(\pi_k(a^\prime|s)) + G_{T(\Delta/4)}(s) - \eta \,T(\Delta/4) \Delta_{s} /2 + \eta \, \Delta_{s} /2 \cdot k \\
        &\geq \underbrace{\psi^\prime(0) + G_{T(\Delta/4)}(s) - \eta \,T(\Delta/4) \Delta_{s} /2}_{:= E_{s}} + \eta \, \Delta_{s} /2 \cdot k.
    \end{align*}
    Define $\chi(x) := \psi^\prime(1-x)$, which is continuous and strictly decreasing over $(0,1]$. Denote its inverse function by $\chi^{-1}: [\psi^\prime(0), \; +\infty) \to (0,1]$. Set $\tilde T \geq T(\Delta/4)$ such that $E_{s} + \eta \, \Delta_{s}/2 \cdot k \geq \psi^\prime(0)$ for all $ k \geq \tilde T$. Then applying the inverse function yields that
    \begin{align*}
        \forall\, k \geq \tilde T: \quad b^{\pi_k}_{s} \leq \chi^{-1}\Big( E_{s} + \eta \, \Delta_{s} / 2 \cdot k \Big).
    \end{align*}
    Noticing that $\chi$ is continuous with $\chi(0^+) = \psi^\prime(1^-) = +\infty$, pick $\epsilon > 0$ such that $\chi(x) \geq 0$ over $(0,\epsilon]$. Notice that
    \[\int_0^\epsilon \chi(x) dx = \int_0^\epsilon \psi^\prime(1-x) dx =  \psi(1) - \psi(1-\epsilon) < \infty\]
    by $\dom \, \psi \supseteq [0,1]$. Picking $T \geq \tilde T$ such that $E_{s} + \eta \, \Delta_{s} /2 \cdot k \geq \chi(\epsilon)$ for all $k \geq T$, applying Lemma~\ref{lem:integral-discrimination} using $f = \chi$ gives that
    \begin{align*}
        \sum_{k \geq T} b^{\pi_k}_{s} \leq \sum_{k \geq T} \chi^{-1} \Big( E_{s} + \eta \, \Delta_{s} / 2 \cdot k \Big) < \infty,
    \end{align*}
    implying that $\sum_k b^{\pi_k}_{s} < \infty$.
\end{itemize}
As $b^{\pi_k}_{s^\prime} = 0$ for all $s^\prime \not \in \tilde \calS$, we conclude that 
\begin{align*}
    \sum_{k \geq 0} b^{\pi_k}_{s} < +\infty, \quad \forall\, s \in \calS.
\end{align*}
By Lemma~\ref{lem:sub-optimal-probabilities},
\begin{align} \label{eq:0-finite-1-infinite:error-summable}
    \sum_{k \geq 0} \big\| V^* - V^{\pi_k} \big\|_\infty \leq \sum_{k \geq 0} \frac{1}{(1-\gamma)^2} \sum_{s\in\calS} b^{\pi_k}_{s} = \frac{1}{(1-\gamma)^2} \sum_{s\in\calS} \sum_{k \geq 0} b^{\pi_k}_{s} <+\infty
\end{align}
where the last equality is due to Fubini's theorem.

\subsection{Proof of Theorem~\ref{thm:case-3-finite-time}}
\label{sec:pf:thm:case-3-finite-time}
We first prove the following auxiliary result.
\begin{lemma}  \label{lem:case-3-psi-prime-bounded}
    Let $\{ \pi_k \}_{k \geq 0}$ be the sequence of policies generated by PMD~\eqref{eq:PMD}. Under Case~3 \textup{(}i.e., $\psi^\prime(0^+) > -\infty$ and $\psi^\prime(1^-) = +\infty$\textup{)}, there exists a constant $\beta$ such that
    \begin{align*}
        \forall\, k\geq 0, \; \forall\, s\in\calS_{>1}, \; \forall\, a\in\calA: \quad \psi^\prime(\pi_k(a|s)) \leq \beta.
    \end{align*}
\end{lemma}
\begin{proof}
    As $\calS_{>1}$ and $\calA$ are finite, it suffices to show that $\sup_{k \geq 0} \, \psi^\prime(\pi_k(a|s)) < +\infty$ for all $s\in\calS_{>1}$ and $a\in\calA$.  For $a^\prime \not \in \calA^*_s$, the global value convergence implies that $\pi_k(a^\prime|s) \to 0$, and thus it is not hard to see that $\sup_{k \geq 0} \psi'(\pi_k(a^\prime|s)) < +\infty$. Now consider the optimal actions. To this end, recall the $\psi^\prime$-score gap of optimal actions,
    \begin{align*}
        \forall\, s\in\calS_{>1}: \quad M_k(s) := \max_{a^* \in \calA^*_s } \, \psi^\prime(\pi_k(a^*|s))  - \min_{\tilde a^* \in\calA^*_s} \, \psi^\prime(\pi_k(\tilde a^*|s)).
    \end{align*}
    For PMD under Case~3, it has been proved that $\sum_k \| V^* - V^{\pi_k} \|_\infty < \infty$ in the proof of Theorem~\ref{thm:case-3-policy-convergence}, see~\eqref{eq:0-finite-1-infinite:error-summable}. From the proof of Lemma~\ref{lem:summable-error-gives-policy-convergence}, we know that $\sum_k \| V^* - V^{\pi_k} \|_\infty < \infty$ gives the boundedness of $M_k(s)$, i.e., $\sup_k M_k(s) < +\infty$, see~\eqref{eq:bounded-M}. By the definition of $M_k(s)$, noticing that $\min_{\tilde a^*\in\calA^*_s} \, \pi_k(\tilde a^*|s) \leq 1/2$, 
    \begin{align*}
        \forall\, s\in\calS_{>1}, \;\forall\, a^*\in\calA^*_s: \quad \psi^\prime(\pi_k(a^*|s)) &\leq \max_{a^*\in\calA^*_s} \, \psi^\prime(\pi_k(a^*|s)) \\
        &= M_k(s) + \min_{\tilde a^*\in\calA^*_s} \, \psi^\prime(\pi_k(\tilde a^*|s)) \\
        &\leq M_k(s) + \psi^\prime  \Big( \min_{\tilde a^*\in\calA^*_s} \, \pi_k(\tilde a^*|s) \Big) \leq M_k(s) + \psi^\prime(1/2).
    \end{align*}
    Taking the supremum,
    \begin{align*}
        \forall\, s\in\calS_{>1}, \;\forall\, a^*\in\calA^*_s: \quad \sup_{k \geq 0} \, \psi^\prime(\pi_k(a^*|s)) \leq \sup_{k \geq 0} \,M_k(s) + \psi^\prime(1/2) < +\infty.
    \end{align*}
    Now just set $\beta := \max_{s\in\calS_{>1}, \, a\in\calA} \, \sup_{k \geq 0} \, \psi^\prime(\pi_k(a|s))$.
    
\end{proof}

\begin{proof}[Proof of Theorem~\ref{thm:case-3-finite-time}]
    For any $s\in\tilde\calS \subseteq \calS_{>1}$, suppose that $b^{\pi_k}_s > 0$ for all $k \geq T(\Delta/4)$. Then by Lemma~\ref{lem:gap-metric-increase}~\ref{lem:gap-metric-increase-lower}, we have $G_k(s) \to +\infty$. However, by Lemma~\ref{lem:case-3-psi-prime-bounded}, $\psi^\prime(\pi_k(a|s)) \in [\psi^\prime(0^+), \; \beta]$, and thus $G_k(s) \leq \beta - \psi^\prime(0^+) < +\infty$, yielding a contradiction. Therefore,  for any $s\in\tilde\calS \subseteq \calS_{>1}$, there exists a finite time $T(\Delta/4) \leq T_{\mathrm{stop}}(s) < +\infty$ such that $b^{\pi_{T_{\mathrm{stop}}(s)}}_s =0$. By Lemma~\ref{lem:gap-metric-increase}~\ref{lem:termination-condition}, we can claim that
    \begin{align*}
        \forall\, s \in \tilde \calS, \; \forall\, k \geq T_{\mathrm{stop}}(s) \quad b^{\pi_k}_s = 0.
    \end{align*}
    For $s\not\in \tilde \calS$, there always holds $b^{\pi_k}_s = 0$ and thus we set $T_{\mathrm{stop}}(s) = 0$. Then
    \begin{align}  \label{eq:case-3:stop-time-in->1-optimal-actions}
        \forall\, s \in \calS_{>1}, \; \forall\, k \geq T_{\mathrm{stop}}(s) \quad b^{\pi_k}_s = 0.
    \end{align}    
    Lemma~\ref{lem:gap-metric-increase}~\ref{lem:termination-condition} implies that PMD terminates in at most $T_{\mathrm{ter}} := \max_{s\in\calS_{>1}} \, T_{\mathrm{stop}}(s)$ iterations, i.e.,
    \begin{align*}
        \forall\, k \geq T_{\mathrm{ter}}: \quad \pi_k = \pi_{T_{\mathrm{ter}}} \in \Pi^*.
    \end{align*}
\end{proof}

\subsection{Proof of Theorem~\ref{thm:case-3-local-convergence}}
\label{sec:pf:thm:case-3-local-convergence}
\textbf{Upper bound:} Pick an arbitrary state $s\in\calS_{=1}$ and denote the optimal action by $a^* \in \calA^*_s = \{ a^* \}$. As $\psi^\prime(1^-) = +\infty$, one has $1 \not\in \interior \, \dom \, \psi$ and thus $\pi_k(a^*|s) < 1$ for all $k\ge 0$. Thus PMD does not terminate at a finite time. Noting that $a^*$ is the unique optimal action, we have $b^{\pi_k}_s = 1 - \pi_k(a^*|s)$, and thus $b^{\pi_k}_s > 0$ holds for all $k \geq 0$. Then following the same argument as for~\eqref{eq:case-3-psi-score-gap} but using $T(\sigma)$ we obtain
\begin{align*}
    \forall\, k \geq T(\sigma): \quad G_k(s) &\geq G_{T(\sigma)}(s) + \eta \, (\Delta_{s} - 2\sigma) \cdot (k - T(\sigma)) \\
    &\geq G_{T(\sigma)}(s) + \eta \, (\Delta_{=1} - 2\sigma) \cdot (k - T(\sigma)),
\end{align*}
where $\Delta_{=1} := \min_{s\in\calS_{=1}}\Delta_s$. Continuing the analysis therein, 
\begin{align}  \label{eq:case-3-b-bound}
    \forall\, k \geq \tilde T(\sigma), \; \forall\, s \in \calS_{=1}: \quad b^{\pi_k}_s \leq \chi^{-1} \big( C_1(\sigma) + \eta \, (\Delta_{=1} - 2\sigma) \cdot k \big),  
\end{align}
where $C_1(\sigma) := \min_{s\in\calS_{=1}}\, E_s(\sigma)$, $E_s(\sigma) := \psi^\prime(0) + G_{T(\sigma)}(s) - \eta \, T(\sigma) \cdot (\Delta_{=1}-2\sigma)$, and $\tilde T(\sigma) \geq T(\sigma)$ is the time step such that $C_1(\sigma) + k\cdot \eta \, (\Delta_{=1} - 2\sigma) \geq \psi^\prime(0)$ for all $k \geq \tilde T(\sigma)$. Following a similar contradiction argument based on Lemmas~\ref{lem:case-3-psi-prime-bounded} and~\ref{lem:gap-metric-increase}~\ref{lem:gap-metric-increase-lower} for equation~\eqref{eq:case-3:stop-time-in->1-optimal-actions}, for any state $\tilde s \in \calS_{>1}$, there exists a finite stopping time $T_{\mathrm{stop}}(\tilde s)$ such that $b^{\pi_k}_{\tilde s} = 0$ for all $k \geq T_{\mathrm{stop}}(\tilde s)$. Setting $T_1(\sigma) := \max\, \big\{ \tilde T(\sigma), \; \{ T_{\mbox{\scriptsize stop}}(\tilde s) \}_{\tilde s \in \calS_{>1}} \big\}$ and using Lemma~\ref{lem:sub-optimal-probabilities} one has
\begin{align*}
    \forall\, k \geq T_1(\sigma), \; \forall\, s^\prime \in \calS: \quad V^*(s^\prime) - V^{\pi_k}(s^\prime)  &\leq \frac{1}{(1-\gamma)^2} \sum_{s\in\calS} d^{\pi_k}_{s^\prime}(s) \cdot b^{\pi_k}_s \\
    &= \frac{1}{(1-\gamma)^2} \sum_{s\in \calS_{=1}} d^{\pi_k}_{s^\prime}(s) \cdot b^{\pi_k}_s \\
    &\leq \frac{1}{(1-\gamma)^2} \sum_{s\in \calS_{=1}} d^{\pi_k}_{s^\prime}(s) \cdot \chi^{-1} \big( C_1(\sigma) + \eta \, (\Delta_{=1} - 2\sigma) \cdot k \big) \\
    &\leq \frac{1}{(1-\gamma)^2}\cdot \chi^{-1} \big( C_1(\sigma) + \eta \, (\Delta_{=1} - 2\sigma) \cdot k \big).
\end{align*}
The arbitrariness of $s^\prime$ gives that
\begin{align*}
    \forall\, k \geq T_1(\sigma): \quad \big\| V^* - V^{\pi_k} \big\|_\infty \leq \frac{1}{(1-\gamma)^2}\cdot \chi^{-1} \big( C_1(\sigma) + \eta \, (\Delta_{=1} - 2\sigma) \cdot k \big).
\end{align*}

\vspace{1em}

\noindent \textbf{Lower bound:} When $k \geq T(\sigma)$ one has $\| V^* - V^{\pi_k} \|_\infty \leq \sigma$, implying that $b^{\pi_k}_{s} \leq \Delta^{-1} \sigma < 1$ when $\sigma < \Delta / 2$ for all $s\in\calS$. Pick $s_0 \in \calS_{=1}$ and $a^\prime_0 \not\in\calA^*_{s_0}$ such that
\begin{align*}
    \Delta_{s_0} = \min_{s\in\calS_{=1}} \; \Delta_s = \Delta_{=1}, \quad V^*(s_0) - Q^*(s_0, a_0^\prime) = \Delta_{s_0} = \Delta_{=1},
\end{align*}
Applying Lemma~\ref{lem:gap-metric-increase}~\ref{lem:gap-metric-increase-upper} to $(s_0, a_0^\prime)$ one obtains
\begin{align*}
    \forall\, k \geq T(\sigma): \quad G_{k+1}(s_0, a_0^\prime) \leq G_k(s_0, a_0^\prime) + \eta \, (\Delta_{=1} + 2\sigma).
\end{align*}
Iterating from $T(\sigma)$ to $k-1$ gives that
\begin{align*}
    \psi^\prime(\pi_k(a^*|s_0)) - \psi^\prime(\pi_k(a^\prime_0|s_0)) \leq \psi^\prime(\pi_{T(\sigma)}(a^*|s_0)) - \psi^\prime(\pi_{T(\sigma)}(a^\prime_0|s_0)) + \eta \,(\Delta_{=1} + 2\sigma) \cdot (k - T(\sigma)).
\end{align*}
Noting that $b^{\pi_k}_{s_0} \leq \Delta^{-1} \sigma$ for $k \geq T(\sigma)$, we have
\begin{align*}
    \forall \, k \geq T(\sigma): \quad \psi^\prime (\pi_k(a^\prime_0 | s_0)) \leq \psi^\prime(b^{\pi_k}_{s_0}) \leq \psi^\prime(\Delta^{-1} \sigma) < \psi^\prime(1/2)
\end{align*}
since $\sigma < \Delta / 2$. Thus
\begin{align*}
    \psi^\prime(1-b^{\pi_k}_{s_0}) = \psi^\prime(\pi_k(a^*|s_0)) &\leq \psi^\prime(1/2) + \psi^\prime(\pi_{T(\sigma)}(a^*|s_0)) - \psi^\prime(\pi_{T(\sigma)}(a^\prime_0|s_0)) + \eta \,(\Delta_{=1} + 2\sigma) \cdot (k - T(\sigma)) \\
    &:= C_2(\sigma) + \eta \, (\Delta_{=1} + 2\sigma) \cdot k.
\end{align*}
Recall that $\chi(x) = \psi^\prime(1-x)$ and $\chi^{-1}: [\psi^\prime(0^+), +\infty ) \to (0,1]$. Take $T_2(\sigma) \geq T(\sigma)$ such that $C_{2}(\sigma) + \eta \, (\Delta_{=1} + 2\sigma) \cdot k \geq \psi^\prime(0^+)$ for all $k \geq T_2(\sigma)$. Applying $\chi^{-1}$ on both sides gives that
\begin{align*}
    \forall\, k \geq T_2(\sigma): \quad b^{\pi_k}_{s_0} \geq \chi^{-1} \big( C_2(\sigma) + \eta \, (\Delta_{=1} + 2\sigma) \cdot k \big).
\end{align*}
Finally, leveraging Lemma~\ref{lem:sub-optimal-probabilities} yields the desired result. 

\subsection{Proof of Theorem~\ref{thm:case-3-policy-convergence-rate}}
\label{sec:pf:thm:case-3-policy-convergence-rate}
Part~\ref{thm:case-3-policy-convergence-rate-optimal=1} follows directly from~\eqref{eq:case-3-b-bound} as $\pi_\infty(\cdot|s) = \pi^*(\cdot|s)$ is unique and thus $\| \pi_k(\cdot|s) - \pi_\infty(\cdot|s) \|_\infty \leq b^{\pi_k}_s$ for $s\in\calS_{=1}$. For Theorem~\ref{thm:case-3-policy-convergence-rate}~\ref{thm:case-3-policy-convergence-rate-optimal>1}, as $0 \in \interior\, \dom \, \psi$, it is possible that $\pi_\infty(a^*|s) = 0$ for optimal action $a^*\in\calA^*_s$. Therefore, we divide the optimal set $\calA^*_s$ into two sets $\calB^*_s$ and $\mathcal{C}^*_s$,
\begin{align*}
    \calB_s^* := \{ a^*\in\calA^*_s: \; \; \pi_\infty(a^*|s) = 0 \}, \quad \calC_s^* := \{ a^*\in\calA^*_s: \;\; \pi_\infty(a^*|s) > 0 \}.
\end{align*}
It is obvious that $\calC^*_s \neq \emptyset$. Furthermore, by Theorem~\ref{thm:case-3-policy-convergence} and equation~\eqref{eq:policy-convergence-to-interior} we know that $|\calC^*_s| > 1$ and $\pi_\infty(a^*|s) \in (0, 1)$ for all $a^*\in\calC^*_s$. Pick a small enough $\epsilon > 0$ such that $-\epsilon \in \interior\, \dom \, \psi$. Define 
\begin{align*}
    \varrho_s := \frac{1}{2} \cdot \min \, \Big\{ \epsilon, \;\; 1 - \max_{a^*\in\calC^*_s} \, \pi_\infty(a^*|s) \Big\} > 0.
\end{align*}
Then define the following interval,
\begin{align*}
    \mathcal{I}_s := \Big[ -\varrho_s, \;\; \max_{a^*\in\calC^*_s} \, \pi_\infty(a^*|s) + \varrho_s \Big] \subsetneq \interior\, \dom \, \psi.
\end{align*}
Similarly, define $U_s := \sup_{x\in\mathcal{I}_s} \psi^{\prime\prime}(x)$, $L_s := \inf_{x\in\mathcal{I}_s} \psi^{\prime\prime}(x)$, and 
\begin{align*}
\psi^\prime(\mathcal{I}_s) := \{ \psi^\prime(x) : \;\; x\in\mathcal{I}_s \} = \Big[ \psi^\prime(-\varrho_s), \;\; \psi^\prime \Big( \max_{a^*\in\calC^*_s} \, \pi_\infty(a^*|s) + \varrho_s \Big) \Big] .
\end{align*}
As $\psi^{\prime\prime}$ is continuous and $0 < \psi^{\prime\prime}(x) < +\infty$ for any $x\in\mathcal{I}_s$, we have $U_s, L_s \in (0, +\infty)$. Additionally, $\psi^\prime$, $(\psi^\prime)^{-1}$ are well defined on $\mathcal{I}_s$, $\psi^\prime(\mathcal{I}_s)$, respectively. It is easy to verify the following two-sided bounds
\begin{alignat*}{2}
    &\forall\, x, \, y \in \mathcal{I}_s: \quad &&L_s \cdot |x-y| \leq \big| \psi^\prime(x) - \psi^\prime(y) \big| \leq U_s \cdot |x-y|, \numberthis \label{eq:eq:case-3-bilevel-Lipchitz} \\
    &\forall\, u, \, v \in \psi^\prime(\mathcal{I}_s): \quad && \frac{1}{U_s} \cdot |u - v| \leq \big| (\psi^\prime)^{-1}(u) - (\psi^\prime)^{-1}(v) \big| \leq \frac{1}{L_s} |u-v|. \numberthis \label{eq:eq:case-3-bilevel-Lipchitz-2}
\end{alignat*}
We proceed by considering the two subcases $\mathcal B_s^*=\emptyset$ and $\mathcal B_s^*\neq\emptyset$. As $\pi_k(a^\prime|s)$ with $a^\prime \not \in \calA^*_s$ can be upper bounded by $b^{\pi_k}_s$, it suffices to establish the upper bound for $|\pi_k(a^*|s) - \pi_\infty(a^*|s)|$ for each situation.

\vspace{1em}

\noindent \textbf{Subcase $\mathcal{B}^*_s = \emptyset$}. The proof for this situation is similar to the one for Theorem~\ref{thm:case-2-policy-convergence-rate}~\ref{thm:case-2-policy-convergence-rate-optimal>1}. As $\calB^*_s = \emptyset$, we have $\calC^*_s = \calA^*_s$. Noticing that $\pi_\infty(a^*|s) > 0$ for all $a^* \in \calC^*_s = \calA^*_s$, there exists a finite time $\tilde T_1$ such that 
\begin{align*}
    \forall\, k \geq \tilde T_1,  \; \forall\, s\in\calS_{>1},\; \forall\, a^*\in\calC^*_s: \quad \pi_k(a^*|s) > 0.
\end{align*}
Now pick an arbitrary optimal action $a^* \in \calA^*_s = \calC^*_s$. As $\pi_k(a^*|s) > 0$ for all $a^*\in\calA^*_s$. Following the same argument in the proof of Theorem~\ref{thm:case-2-policy-convergence-rate}~\ref{thm:case-2-policy-convergence-rate-optimal>1}, we obtain
\begin{align*}
    \forall\, &k \geq \tilde T_1, \; \forall \, \tilde a^* \in \calA^*_s: \\
    &\pi_k(\tilde a^* | s) = (\psi^\prime)^{-1} \bigg[ \psi^\prime(\pi_\infty(\tilde a^*|s)) + \underbrace{\psi^\prime(\pi_k(a^*|s)) - \psi^\prime(\pi_\infty(a^*|s))}_{:= c_k(s)} + \eta \underbrace{\sum_{j=k}^\infty \big[ Q^{\pi_j}(s,a^*) - Q^{\pi_j}(s, \tilde a^*) \big]}_{:= e_k(s, \tilde a^*)} \bigg].  \numberthis \label{eq:policy-rate-case-3-policy-formula}
\end{align*}
Similarly leveraging the fact $\sum_{a^*\in\calA^*_s} \pi_\infty(a^*|s) - \pi_k(a^*|s) = b^{\pi_k}_s$, for $k \geq \tilde T_1$,
\begin{align*}
    b^{\pi_k}_s &= \underbrace{\sum_{\hat a^* \in \calA^*_s} (\psi^\prime)^{-1} [\psi^\prime (\pi_\infty(\hat a^*|s))] - (\psi^\prime)^{-1} \big[ \psi^\prime(\pi_\infty(\hat a^*|s)) + c_k(s) \big]}_{:= (\Rmnum{1})} \\
    &\;\;\;\; + \underbrace{\sum_{\tilde a^* \in \calA^*_s, \, \tilde a^* \neq a^*} (\psi^\prime)^{-1} \big[ \psi^\prime (\pi_\infty(\tilde a^*|s)) + c_k(s) \big] - (\psi^\prime)^{-1} \Big[ \psi^\prime(\pi_\infty(\tilde a^*|s)) + c_k(s) + \eta \, e_k(s, \tilde a^*) \Big]}_{:= (\Rmnum{2})}.
\end{align*}
Thus we obtain that
\begin{align*}
    |(\Rmnum{1})| \leq b^{\pi_k}_s + |(\Rmnum{2})|.
\end{align*}
Similarly, notice that $\lim_{k\to\infty} \psi^\prime(\pi_k(a^*|s)) = \psi^\prime(\pi_\infty(a^*|s))$, giving that $c_k(s) \to 0$. Furthermore, $\lim_{k\to\infty}e_k(s, \tilde a^*) = 0$ for all $\tilde a^* \in \calA^*_s$. Thus there exists a finite time $\tilde T_2$ such that for all $s\in\calS_{>1}$, \begin{align*}\forall\, k \geq \tilde T_2, \; \forall\, \tilde a^*\in\calA^*_s: \quad &\phantom{=\,\,\,}|c_k(s)| + \eta\,|e_k(s, \tilde a^*)| \\
&\leq \min \, \{ \psi^\prime(0) - \psi^\prime(-\varrho_s), \; \psi^\prime(\max_{a^*\in\calC^*_s} \, \pi_\infty(a^*|s) + \varrho_s) - \psi^\prime(\max_{a^*\in\calC^*_s} \, \pi_\infty(a^*|s)) \}, \end{align*}
so that $\psi^\prime(\pi_\infty(\hat a^*|s)) + c_k(s) \in \psi^\prime(\mathcal{I}_s)$. Then repeating the same argument in the proof of Theorem~\ref{thm:case-2-policy-convergence-rate}~\ref{thm:case-2-policy-convergence-rate-optimal>1},
\begin{align*}
    |(\Rmnum{1})| \geq \frac{1}{U_s} \sum_{\hat a^* \in \calA^*_s} |c_k(s)| = \frac{|\calA^*_s|}{U_s} |c_k(s)|,
\end{align*}
\begin{align*}
    |(\Rmnum{2})| \leq \frac{1}{L_s} \sum_{\tilde a^* \in \calA^*_s, \, \tilde a^* \neq a^*} \eta \, \big| e_k(s, \tilde a^*) \big| \leq \frac{2\eta(|\calA^*_s| - 1)}{L_s} \sum_{j=k}^\infty \big\| V^* - V^{\pi_j} \big\|_\infty.
\end{align*}
Thus by $|(\Rmnum{1})| \leq b^{\pi_k}_s + |(\Rmnum{2})|$,
\begin{align*}
    \forall\, k \geq \max\{ \tilde T_1, \, \tilde T_2 \}: \quad  &\frac{|\calA^*_s|}{U_s} |c_k(s)| \leq b_s^{\pi_k} + \frac{2\eta(|\calA^*_s| - 1)}{L_s} \sum_{j=k}^\infty \big\| V^* - V^{\pi_j} \big\|_\infty \\
    \Longrightarrow \quad & |c_k(s)| \leq U_s \cdot b^{\pi_k}_s + \frac{2\eta \,U_s}{L_s} \sum_{j=k}^\infty \big\| V^* - V^{\pi_j} \big\|_\infty.
\end{align*}
Notice that $|c_k(s)| = |\psi^\prime(\pi_k(a^*|s)) - \psi^\prime(\pi_\infty(a^*|s))|$. Applying the two-sided bound~\eqref{eq:eq:case-3-bilevel-Lipchitz},
\begin{align*}
    \forall\, k \geq \max\{ \tilde T_1, \, \tilde T_2 \}: \quad|\pi_k(a^*|s) - \pi_\infty(a^*|s)| \leq \frac{1}{L_s} |c_k(s)| \leq \frac{U_s}{L_s} \cdot b^{\pi_k}_s + \frac{2\eta \,U_s}{L_s^2} \sum_{j=k}^\infty \big\| V^* - V^{\pi_j} \big\|_\infty.
\end{align*}
As the selection of $a^*\in\calA^*_s$ is arbitrary, we obtain that
\begin{align*}
    \forall\, k \geq \max \, \{ \tilde T_1, \tilde T_2 \}, \; \forall\, a^*\in\calA^*_s: \quad |\pi_k(a^*|s) - \pi_\infty(a^*|s)| \leq \frac{U_s}{L_s} \cdot b^{\pi_k}_s + \frac{2\eta \,U_s}{L_s^2} \sum_{j=k}^\infty \big\| V^* - V^{\pi_j} \big\|_\infty.
\end{align*}

\vspace{1em}

\noindent \textbf{Subcase $\calB^*_s \neq \emptyset$}. We first establish the upper bound of $|\pi_k(a^*|s) - \pi_\infty(a^*|s)|$ for $a^* \in \calB^*_s$. Notice that $1 = \sum_{\hat a^*\in \calC^*_s} \pi_\infty(\hat a^*|s) \geq \sum_{\hat a^*\in\calC^*_s} \pi_k(\hat a^*|s)$. Thus for any $k \geq 0$, there exists an action $a^*_k \in \calC^*_s$ such that 
\begin{align*}
    \pi_k(a^*_k | s) \leq \pi_\infty(a^*_k|s).
\end{align*}
As $a^*_k \in \calC^*_s$, we have $\pi_t(a^*_k|s) > 0$ for all $t \geq \tilde T_1$. Then by KKT condition, for $k \geq \tilde T_1$,
\begin{alignat*}{3}
    &\forall\, t \geq k \geq \tilde T_1: \quad &&\psi^\prime(\pi_{t+1}(a^*_k|s)) &&= \psi^\prime(\pi_t(a^*_k|s)) + \eta \, Q^{\pi_t}(s, a^*_k) + \nu_{s}, \\
    &\forall\, t \geq k \geq \tilde T_1, \; \forall\, a^*\in\calB^*_s: \quad &&\psi^\prime(\pi_{t+1}(a^*|s)) &&= \psi^\prime(\pi_t(a^*|s)) + \eta \, Q^{\pi_t}(s, a^*) + \nu_{s} + \lambda_{s}(a^*) \\
    & && &&\geq \psi^\prime(\pi_t(a^*|s)) + \eta \, Q^{\pi_t}(s, a^*) + \nu_{s}.
\end{alignat*}
Subtracting,
\begin{alignat*}{2}
    \forall\, &t \geq k \geq \tilde T_1, \; \forall\, a^* \in \calB^*_s: \\
    & \psi^\prime(\pi_{t+1}(a^*_k|s)) - \psi^\prime(\pi_{t+1}(a^*|s)) &&\leq \psi^\prime(\pi_t(a^*_k|s)) - \psi^\prime(\pi_t(a^*|s)) + \eta \, \big[ Q^{\pi_t}(s, a^*_k) - Q^{\pi_t}(s,a^*) \big]\\
    & && \leq \psi^\prime(\pi_t(a^*_k|s)) - \psi^\prime(\pi_t(a^*|s)) + 2\eta \, \big\| V^* - V^{\pi_t} \big\|_\infty.
\end{alignat*}
Telescoping from $t \geq k$ to $k$,
\begin{align*}
    \forall\, &t > k \geq \tilde T_1, \; \forall\, a^*\in\calB^*_s: \\
    & \psi^\prime(\pi_t(a^*_k|s)) - \psi^\prime(\pi_t(a^*|s)) \leq \psi^\prime(\pi_k(a^*_k|s)) - \psi^\prime(\pi_k(a^*|s)) + 2\eta \sum_{j=k}^{t-1} \big\| V^* - V^{\pi_j} \big\|_\infty.
\end{align*}
It follows that
\begin{alignat*}{2}
    \forall\, k \geq \tilde T_1, \; \forall \, &a^* \in \calB^*_s: \\
    & \psi^\prime(\pi_k(a^*|s)) &&\leq \psi^\prime(\pi_k(a^*_k|s)) - \psi^\prime(\pi_t(a^*_k|s)) + \psi^\prime(\pi_t(a^*|s)) + 2\eta \sum_{j=k}^{t-1} \big\| V^* - V^{\pi_j} \big\|_\infty \\
    & &&\leq \psi^\prime(\pi_k(a^*_k|s)) + \underset{t}{\lim\sup} \; \bigg[ \psi^\prime(\pi_t(a^*|s)) - \psi^\prime(\pi_t(a^*_k|s)) + 2\eta \sum_{j=k}^{t-1} \big\| V^* - V^{\pi_j} \big\|_\infty \bigg] \\
    & &&\stackrel{\mathrm{(limit \; exists)}}{=} \psi^\prime(\pi_k(a^*_k|s)) + \lim_{t} \; \bigg[ \psi^\prime(\pi_t(a^*|s)) - \psi^\prime(\pi_t(a^*_k|s)) + 2\eta \sum_{j=k}^{t-1} \big\| V^* - V^{\pi_j} \big\|_\infty \bigg] \\
    & &&= \psi^\prime(\pi_k(a^*_k|s)) + \psi^\prime(\pi_\infty(a^*|s)) - \psi^\prime(\pi_\infty(a^*_k|s)) + 2\eta \sum_{j=k}^{\infty} \big\| V^* - V^{\pi_j} \big\|_\infty \\
    & &&\leq \psi^\prime(\pi_\infty(a^*|s)) + 2\eta \sum_{j=k}^{\infty} \big\| V^* - V^{\pi_j} \big\|_\infty,
\end{alignat*}
where the last line follows from our selection of $a^*_k$. Noticing that $\pi_k(a^*|s) \to \pi_\infty(a^*|s) = 0$, there exists a finite time $\tilde T_3$ such that $\pi_k(a^*|s) \leq \max_{a^*\in\calC^*_s} \, \pi_\infty(a^*|s) + \varrho_s < 1$ for all $k \geq \tilde T_3$, $s\in\calS_{>1}$, and $a^*\in\calB^*_s$, and then $\pi_k(a^*|s) \in \mathcal{I}_s$. Thus
\begin{alignat*}{2}
    \forall\, &k \geq \max\{ \tilde T_1, \tilde T_3 \}, \; \forall\, a^*\in\calB^*_s: &&\\
    &\pi_k(a^*|s) = |\pi_k(a^*|s) - \pi_\infty(a^*|s)| &&\leq \frac{1}{L_s} \big| \psi^\prime(\pi_k(a^*|s)) - \psi^\prime(\pi_\infty(a^*|s)) \big|  \\
    & &&= \frac{1}{L_s} \big[ \psi^\prime(\pi_k(a^*|s)) - \psi^\prime(\pi_\infty(a^*|s)) \big] \leq \frac{2\eta}{L_s} \sum_{j=k}^\infty \big\| V^* - V^{\pi_j} \big\|_\infty.
\end{alignat*}
Now it remains to bound the optimal actions in $\calC^*_s$. Pick an arbitrary $a^* \in \calC^*_s$. Following the same argument,~\eqref{eq:policy-rate-case-3-policy-formula} also holds for any $a^*\in\calC^*_s$,
\begin{align*}
    \forall\, k \geq \tilde T_1, \; &\forall\, \tilde a^* \in \calC^*_s: \\
    &\pi_k(\tilde a^* | s) = (\psi^\prime)^{-1} \bigg[ \psi^\prime(\pi_\infty(\tilde a^*|s)) + \underbrace{\psi^\prime(\pi_k(a^*|s)) - \psi^\prime(\pi_\infty(a^*|s))}_{:= c_k(s)} + \eta \underbrace{\sum_{j=k}^\infty \big[ Q^{\pi_j}(s,a^*) - Q^{\pi_j}(s, \tilde a^*) \big]}_{:= e_k(s, \tilde a^*)} \bigg].
\end{align*}
Similarly, 
\begin{align*}
    b^{\pi_k}_s &= \sum_{\tilde a^*\in\calC^*_s} \big[ \pi_\infty(\tilde a^*|s) - \pi_k(\tilde a^*|s) \big] + \sum_{\bar a^* \in \calB^*_s} \big [ \pi_\infty(\bar a^*|s) - \pi_k(\bar a^*|s) \big] \\
    &= \pi_\infty(a^*|s) - \pi_k(a^*|s) + \sum_{\tilde a^* \in \calC^*_s, \, \tilde a^* \neq a^*} \big[ \pi_\infty(\tilde a^*|s) - \pi_k(\tilde a^*|s) \big] - \sum_{\bar a^*\in \calB^*_s} \pi_k(\bar a^*|s) \\
    &= \underbrace{\sum_{\hat a^* \in \calC^*_s} (\psi^\prime)^{-1} [\psi^\prime (\pi_\infty(\hat a^*|s))] - (\psi^\prime)^{-1} \big[ \psi^\prime(\pi_\infty(\hat a^*|s)) + c_k(s) \big]}_{:= (\Rmnum{1})} \\
    &\;\;\;\; + \underbrace{\sum_{\tilde a^* \in \calC^*_s, \, \tilde a^* \neq a^*} (\psi^\prime)^{-1} \big[ \psi^\prime (\pi_\infty(\tilde a^*|s)) + c_k(s) \big] - (\psi^\prime)^{-1} \Big[ \psi^\prime(\pi_\infty(\tilde a^*|s)) + c_k(s) + \eta \, e_k(s, \tilde a^*) \Big]}_{:= (\Rmnum{2})} \\
    &\;\;\;\; - \underbrace{\sum_{\bar a^*\in \calB^*_s} \pi_k(\bar a^*|s)}_{:= (\Rmnum{3})}.
\end{align*}
It gives that
\begin{align*}
    |(\Rmnum{1})| \leq b^{\pi_k}_s + |(\Rmnum{2})| + |(\Rmnum{3})|.
\end{align*}
The bounds for $|(\Rmnum{1})|$ and $|(\Rmnum{2})|$ are still valid,
\begin{align*}
    \forall\, k \geq \max \, \{ \tilde T_1, \tilde T_2 \}: \quad  |(\Rmnum{1})| \geq \frac{|\calC^*_s|}{U_s} |c_k(s)|, \quad |(\Rmnum{2})| \leq \frac{2\eta\, (|\calC^*_s| - 1)}{L_s} \sum_{j=k}^\infty \big\| V^* - V^{\pi_j} \big\|_\infty.
\end{align*}
The upper bound for $|(\Rmnum{3})|$ is given by the previous bound for $\bar a^*\in\calB^*_s$,
\begin{align*}
    \forall\, k \geq \max \, \{ \tilde T_1, \tilde T_3 \}: \quad |(\Rmnum{3})| \leq \frac{2\eta |\calB^*_s| }{L_s} \sum_{j=k}^\infty \big\| V^* - V^{\pi_j} \big\|_\infty.
\end{align*}
Let $T_3(\sigma) := \max\{ \tilde T_1, \tilde T_2, \tilde T_3, T_1(\sigma) \}$. Putting all together we obtain
\begin{align*}
    \forall\, k \geq T_3(\sigma): \quad |c_k(s)| \leq U_s \cdot b^{\pi_k}_s + \frac{2\eta \,U_s}{L_s} (1+ |\calA|) \sum_{j=k}^\infty \big\| V^* - V^{\pi_j} \big\|_\infty.
\end{align*}
By applying the two-sided bound and noting the arbitrariness of $a^*\in\calC^*_s$,
\begin{align*}
    \forall \, k \geq T_3(\sigma), \; \forall\, a^*\in\calC^*_s: \quad |\pi_k(a^*|s) - \pi_\infty(a^*|s) | \leq \frac{U_s}{L_s} \cdot b^{\pi_k}_s + \frac{2\eta \,U_s}{L_s^2} (1+ |\calA|) \sum_{j=k}^\infty \big\| V^* - V^{\pi_j} \big\|_\infty.
\end{align*}
In summary, whether or not $\calB^*_s$ is empty, the following upper bounds hold for all $k \geq T_3(\sigma)$:
\begin{alignat*}{2}
    &\forall\, a^\prime \not \in \calA^*_s: \quad &&|\pi_k(a^\prime|s) - \pi_\infty(a^\prime|s)| \leq b^{\pi_k}_s, \\
    &\forall\, a^*\in\calA^*_s: \quad &&|\pi_k(a^*|s) - \pi_\infty(a^*|s) | \leq \frac{U_s}{L_s} \cdot b^{\pi_k}_s + \frac{2\eta \, U_s}{L_s^2} (1+ |\calA|) \sum_{j=k}^\infty \big\| V^* - V^{\pi_j} \big\|_\infty.
\end{alignat*}
This concludes the proof.

\subsection{Proof of Theorem~\ref{thm:case-4-local-convergence} }
\label{sec:pf:thm:case-4-local-convergence}
\textbf{Upper bound:} Notice that $0<b^{\pi_k}_s <1$ holds for $\forall s\in\tilde\calS$ in this case. Using an argument analogous to that in the proof of the upper bound
in Theorem~\ref{thm:case-2-local-convergence}, we obtain
\begin{align*}
    \forall\, k \geq T(\sigma), \; \forall s\in\tilde\calS: \quad {\min_{a^*\in\calA^*_s} \, \psi^\prime(\pi_k(a^*|s))} - \max_{a^\prime \not\in \calA^*_s} \, \psi^\prime(\pi_k(a^\prime | s)) &\geq  G_{T(\sigma)}(s) + \eta \, (\Delta - 2\sigma) \cdot (k - T(\sigma)) \\
    &= \underbrace{G_{T(\sigma)}(s) - \eta \, (\Delta - 2\sigma) \cdot T(\sigma)}_{:= C_s(\sigma)} + \eta \, (\Delta - 2\sigma) \cdot k \\
    &\geq C_1(\sigma) + \eta \, (\Delta - 2\sigma) \cdot k,
\end{align*}
where $C_1(\sigma) := \min_{s\in\tilde\calS} C_s(\sigma)$ and $|C_1(\sigma)| < \infty$. This inequality implies that for any $k \geq T(\sigma)$ and $s\in \tilde \calS$,
\begin{align*}
    {\min_{a^*\in\calA^*_s} \; \psi^\prime(\pi_k(a^*|s))} \geq \frac{1}{2} \big( C_1(\sigma) + \eta \, (\Delta - 2\sigma) \cdot k \big) \;\; \mbox{or} \;\; \max_{a^\prime \not\in \calA^*_s} \, \psi^\prime(\pi_k(a^\prime | s)) \leq - \frac{1}{2} \big( C_1(\sigma) + \eta \, (\Delta - 2\sigma) \cdot k \big).
\end{align*}
By the monotonicity of $\psi^\prime$,
\begin{align*}
    \psi^\prime(1-b^{\pi_k}_s) \geq \psi^\prime\big( {\min\nolimits_{a^*\in\calA^*_s} \, \pi_k(a^*|s)} \big) = \min_{a^*\in\calA^*_s} \, \psi^\prime(\pi_k(a^*|s)).
\end{align*}
Consequently, for any $k \geq T(\sigma)$ and $s\in \tilde \calS$,
\begin{align*}
    \psi^\prime(1-b^{\pi_k}_s) \geq \frac{1}{2} \big( C_1(\sigma) + \eta \, (\Delta - 2\sigma) \cdot k \big) \;\; \mbox{or} \;\;\psi^\prime(\pi_k(a^\prime|s)) \leq -\frac{1}{2} \big( C_1(\sigma) + \eta \, (\Delta - 2\sigma) \cdot k \big), \; \forall\, a^\prime \not \in \calA^*_s.
\end{align*}
As $\varphi(x) := -\psi^\prime(x)$ and $\chi(x) := \psi^\prime(1-x)$ are both strictly decreasing functions from $(0,1)$ to the whole real line, applying their inverse functions to the above two inequalities yields
\begin{align*}
    \forall\, s\in\tilde\calS, \; k \geq T(\sigma): \quad & b^{\pi_k}_s \leq {\chi^{-1} \bigg( \frac{C_1(\sigma)}{2} + k \cdot \frac{\eta\, (\Delta - 2\sigma)}{2} \bigg)} \;\; \mbox{or} \;\; b^{\pi_k}_s \leq |\calA| \cdot {\varphi^{-1} \bigg( \frac{C_1(\sigma)}{2} + k \cdot \frac{\eta \, (\Delta - 2 \sigma)}{2} \bigg)} \end{align*}
    Notice that $\chi^{-1}, \varphi^{-1}: (-\infty, +\infty) \to (0,1)$, 
    \begin{align} \label{eq:Case-4-bound-for-b}
    b^{\pi_k}_s \leq |\calA| \cdot \chi^{-1} \bigg( \frac{C_1(\sigma)}{2} + k \cdot \frac{\eta\, (\Delta - 2\sigma)}{2} \bigg) + \varphi^{-1} \bigg( \frac{C_1(\sigma)}{2} + k \cdot \frac{\eta \, (\Delta - 2 \sigma)}{2} \bigg).
\end{align}
Since $b_s^{\pi_k}=0$ when $s\notin\tilde{\mathcal{S}}$,  applying Lemma~\ref{lem:sub-optimal-probabilities} gives the main result.

\vspace{1em}

\noindent \textbf{Lower bound:} As in the proof of the lower bound in Theorem~\ref{thm:case-2-local-convergence}, pick $s_0 \in \tilde\calS$ and $a^\prime_0 \not\in\calA^*_{s_0}$ such that
\begin{align*}
    V^*(s_0) - Q^{*}(s_0, a_0^\prime) = \Delta.
\end{align*}
As $0 \not\in \interior\, \dom \, \psi$, by Lemma~\ref{lem:KKT-cond} there still holds $\pi_k(a|s_0) > 0$ for all $k$, $a$, and thus $b^{\pi_k}_{s_0} < 1$. By Lemma~\ref{lem:gap-metric-increase}, there holds
\begin{align*}
    \max_{a^* \in \calA^*_{s_0}} \, \psi^\prime (\pi_{k}(a^*|s_0)) - \psi^\prime(\pi_k(a^\prime_0 | s_0)) \leq G_{T(\sigma)}(s_0, a_0^\prime) + \eta \, (\Delta + 2\sigma)\cdot (k - T(\sigma)),\quad \forall\, k\geq T(\sigma).
\end{align*}
In addition,  $b^{\pi_k}_{s_0} \leq \Delta^{-1} \| V^* - V^{\pi_k} \|_\infty \leq \sigma / \Delta < 1/2$ for all $k \geq T(\sigma)$ by Lemma~\ref{lem:sub-optimal-probabilities}. It follows that \[\max_{a^* \in \calA^*_{s_0}} \, \pi_k(a^*|s_0) \geq \frac{1}{|\calA^*_{s_0}|} (1-b^{\pi_k}_{s_0}) \geq \frac{1}{2|\calA|}.\] Thus by the monotonicity of $\psi^\prime$,
\begin{gather*}
    \max_{a^*\in\calA^*_{s_0}} \, \psi^\prime (\pi_k(a^*|s_0)) = \psi^\prime \Big( \max_{a^*\in\calA^*_{s_0}} \, \pi_k(a^*|s_0)  \Big) \geq \psi^\prime\bigg( \frac{1}{2|\calA|} \bigg).
\end{gather*}
Furthermore,
\begin{gather*}
    \psi^\prime(\pi_k(a^\prime_0 |s_0)) \leq \psi^\prime \bigg( \sum_{a^\prime \not \in \calA^*_{s_0}} \pi_k(a^\prime|s_0) \bigg) = \psi^\prime \big( b^{\pi_k}_{s_0} \big).
\end{gather*}
Putting them together we obtain
\begin{align*}
    \forall\, k \geq T(\sigma): \quad &\psi^\prime \bigg( \frac{1}{2|\calA|} \bigg) - \psi^\prime \big( b^{\pi_k}_{s_0} \big) \leq G_{T(\sigma)}(s_0, a_0^\prime)- \eta \, (\Delta + 2\sigma)\cdot T(\sigma) + \eta \, (\Delta + 2\sigma) \cdot k  \\
    \Longrightarrow \quad & -\psi^\prime (b^{\pi_k}_{s_0}) \leq \underbrace{G_{T(\sigma)}(s_0, a_0^\prime)- \eta \, (\Delta + 2\sigma)\cdot T(\sigma) - \psi^\prime\bigg( \frac{1}{2|\calA|} \bigg)}_{:= C_2(\sigma)} + \eta \, (\Delta + 2\sigma) \cdot k
\end{align*}
Applying the inverse function of $\varphi$ gives that
\begin{align*}
    \forall\, k \geq T(\sigma): \quad b^{\pi_k}_{s_0} \geq \varphi^{-1} \big( C_2(\sigma) + \eta \, (\Delta + 2\sigma) \cdot k \big).
\end{align*}
By Lemma~\ref{lem:sub-optimal-probabilities} one obtains the desired result.

\subsection{Proof of Theorem~\ref{thm:case-4-policy-convergence} }
\label{sec:pf:thm:case-4-policy-convergence}
Similar to the proof of Theorem~\ref{thm:case-2-policy-convergence}, we set $\sigma = \Delta / 4$ and Theorem~\ref{thm:case-4-local-convergence} yields
\begin{align*}
    \forall\, k \geq T(\Delta/4): \quad \big\| V^* - V^{\pi_k} \big\|_\infty &\leq \frac{|\calA|}{(1-\gamma)^2} \cdot \bigg\{ \varphi^{-1} \bigg( \frac{C_1(\Delta/4)}{2} + k \cdot \frac{\eta \, \Delta}{4}  \bigg)+\chi^{-1} \bigg( \frac{C_1(\Delta/4)}{2} + k \cdot \frac{\eta \, \Delta}{4}  \bigg) \bigg\}.
\end{align*}
Notice that both $\varphi$ and $\chi$ are continuous and strictly decreasing in Case~4 with $\varphi(0^+) = +\infty$ and $\chi(0^+) = +\infty$. Thus there exists $\epsilon > 0$ such that $\varphi(x) \geq 0$ and $\chi(x) \geq 0$ on $(0,\epsilon]$. The integrals are finite as well,
\begin{gather*}
    \int_0^\epsilon \varphi(x) dx = \lim_{t\downarrow 0} \int_t^\epsilon \varphi(x) dx =  \lim_{t\downarrow 0} \big[ \psi(t) - \psi(\epsilon) \big] = \psi(0) - \psi(\epsilon) < \infty, \\
    \int_0^\epsilon \chi(x) dx = \lim_{t\downarrow 0} \int_t^\epsilon \chi(x)dx =  \lim_{t\downarrow 0} \big[ \psi(1-t) - \psi(1-\epsilon) \big] = \psi(1) - \psi(1-\epsilon) < \infty.
\end{gather*}
Setting $T \geq T(\Delta/4)$ such that $k \cdot \eta \, \Delta / 4 + C_1(\Delta/4) / 2 \geq \varphi(\epsilon) + \chi(\epsilon)$ for all $k \geq T$, Lemma~\ref{lem:integral-discrimination} implies that
\begin{align*}
    \sum_{k \geq T} \varphi^{-1} \bigg( \frac{C_1(\Delta/4)}{2} + k \cdot \frac{\eta\,\Delta}{4} \bigg) < \infty, \quad \sum_{k \geq T} \chi^{-1} \bigg( \frac{C_1(\Delta/4)}{2} + k \cdot \frac{\eta\,\Delta}{4} \bigg) < \infty,
\end{align*}
giving that
\begin{align*}
    \sum_{k \geq T} \big\| V^* - V^{\pi_k} \big\|_\infty \leq \frac{|\calA|}{(1-\gamma)^2} \bigg( \sum_{k \geq T} \varphi^{-1} \bigg( \frac{C_1(\Delta/4)}{2} + k \cdot \frac{\eta\,\Delta}{4} \bigg) + \sum_{k \geq T} \chi^{-1} \bigg( \frac{C_1(\Delta/4)}{2} + k \cdot \frac{\eta\,\Delta}{4} \bigg)  \bigg) < \infty.
\end{align*}
The proof is completed by invoking Lemma~\ref{lem:summable-error-gives-policy-convergence}. 


\section{Proofs of technical lemmas}
\label{sec:proofs-of-lemmas}

\subsection{Proof of Lemma~\ref{lem:summable-error-gives-policy-convergence}}
\label{sec:pf:lem:summable-error-gives-policy-convergence}
It is not hard to see that $\psi$ satisfying Assumption~\ref{ass:psi-Legendre} is both essentially smooth and essentially strictly convex, and thus it is Legendre and so is $h$~\citep{bauschke1997legendre,bauschke2001essential}. The following two lemmas for Legendre functions in~\citet{bauschke1997legendre} will be utilized in the proof of Lemma~\ref{lem:summable-error-gives-policy-convergence}.

\begin{lemma}[Theorems~3.8.(iii) and 3.9~\protect{\citep{bauschke1997legendre}}]  \label{lem:thm-3.8-iii-3.9-of-bauschke}
    Suppose that $h: \mathbb{R}^{n} \to \mathbb{R}$ is Legendre. Then
    \begin{align*}
        \left. \begin{array}{r}
             \{x_n\} \; \subseteq \; \dom \, h, \;\; x_n \to x \in \dom \, h, \\
             \{y_n\} \; \subseteq \; \interior\,\dom \, h, \;\; y_n \to y \in \dom \, h, \\
             \{ x,y \} \cap \interior \, \dom \, h \neq \emptyset , \; D_h(x_n \, \| \, y_n) \to 0
        \end{array} \right\} \Rightarrow D_h(x \, \| \, y) = 0 \; \mbox{and } \; x=y.
    \end{align*}
\end{lemma}

\begin{lemma}[Proposition~3.3~\protect{\citep{bauschke1997legendre}}]  \label{lem:pro-3.3-of-baischke}
    Suppose that $\psi: \mathbb{R} \to \mathbb{R}$ is Legendre. Then
    \begin{align*}
        \left. \begin{array}{r}
             \{y_n\} \subseteq \interior \, \dom \, \psi, \\
             y_n \to y \in \dom \, \psi,
        \end{array} \right\} \Rightarrow d_\psi(y \, \| \, y_n) \to 0.
    \end{align*}
\end{lemma}

\noindent The following lemma states the continuity of $V^\pi$ with respect to $\pi$.
\begin{lemma}[Continuity of $V^\pi$]  \label{lem:V-continuity}
    The state value function $V^\pi$ is continuous with respect to the policy $\pi$.
\end{lemma}

\begin{proof}
   This is a standard fact for finite discounted MDPs, and we  provide a proof for completeness. Define
    \begin{alignat*}{2}
        \forall\, \pi \in \Pi: \quad &r^\pi \in \mathbb{R}^{|\calS|}: \, &&r^\pi(s) = \sum_{a\in\calA} \pi(a|s) \cdot r(s,a), \\
        & P^\pi \in \mathbb{R}^{|\calS|\times|\calS|}: \quad && P^{\pi}(s, s^\prime) = \sum_{a\in\calA} \pi(a|s) \cdot P(s^\prime|s,a).
    \end{alignat*}
    Further define the following Bellman operator,
    \begin{align*}
        \forall\, \pi \in \Pi, \; \forall\, V \in \mathbb{R}^{|\calS|}: \quad \calT^\pi V := r^\pi + \gamma \, P^{\pi} V.
    \end{align*}
    Bellman equation~\eqref{eq:V-Bellman-eq} gives that
    \begin{align*}
        \forall\, \pi \in \Pi: \quad \calT^\pi V^\pi = V^\pi.
    \end{align*}
    Noting that $\| P^\pi \|_\infty \leq 1$, the Bellman operator is $\gamma$-contraction with respect to the $\ell_\infty$-norm,
    \begin{align*}
        \forall\, \pi \in \Pi, \; \forall\, V, \, V^\prime \in \mathbb{R}^{|\calS|}: \quad \big\| \calT^\pi V - \calT^\pi V^\prime \big\|_\infty \leq \gamma \, \big\| V - V^\prime \big\|_\infty.
    \end{align*}
    Additionally, it is straightforward to verify that
    \begin{align*}
        \forall\, \pi, \, \pi^\prime \in \Pi: \quad \big\| P^{\pi} - P^{\pi^\prime} \big\|_\infty \leq |\calS||\calA| \cdot \big\| \pi - \pi^\prime \big\| _\infty, \;\; \big\| r^\pi - r^{\pi^\prime} \big\|_\infty \leq |\calA| \cdot \big\| \pi - \pi^\prime \big\|_\infty.
    \end{align*}
    For arbitrary $\pi, \, \tilde \pi \in \Pi$, one has
    \begin{align*}
        \big\| V^\pi - V^{\tilde \pi} \big\|_\infty 
        &= \big\| \calT^\pi V^\pi - \calT^{\tilde \pi} V^{\tilde \pi} \big\|_\infty \\
        &= \big\| \calT^\pi V^\pi - \calT^\pi V^{\tilde \pi} +\calT^{\pi} V^{\tilde\pi} - \calT^{\tilde \pi} V^{\tilde \pi} \big\|_\infty \\
        &\leq \big\| \calT^\pi V^\pi - \calT^\pi V^{\tilde \pi} \big\| + \big\| \calT^{\pi} V^{\tilde\pi} - \calT^{\tilde \pi} V^{\tilde \pi} \big\|_\infty \\
        &\leq \gamma \, \big\| V^\pi - V^{\tilde\pi} \big\|_\infty + \big\| \calT^{\pi} V^{\tilde\pi} - \calT^{\tilde \pi} V^{\tilde \pi} \big\|_\infty.
    \end{align*}
    Furthermore,
    \begin{align*}
        \big\| \calT^{\pi} V^{\tilde\pi} - \calT^{\tilde \pi} V^{\tilde \pi} \big\|_\infty &= \big\| r^\pi - r^{\tilde\pi} + \gamma \big( P^{\pi} - P^{\tilde \pi} \big) V^{\tilde\pi} \big\|_\infty \\
        &\leq \big\| r^\pi - r^{\tilde\pi} \big\|_\infty + \frac{\gamma}{1-\gamma} \big\| P^{\pi} - P^{\tilde\pi} \big\|_\infty \\
        &\leq |\calA| \cdot \| \pi - \tilde \pi \|_\infty + \frac{\gamma |\calS||\calA|}{1-\gamma} \cdot \| \pi - \tilde\pi \|_\infty .
    \end{align*}
    Therefore,
    \begin{align*}
        \big\| V^\pi - V^{\tilde \pi} \big\|_\infty \leq \frac{1}{1-\gamma} \Big( |\calA| + \frac{\gamma|\calS||\calA|}{1-\gamma} \Big) \cdot \big\| \pi - \tilde\pi \big\|_\infty,
    \end{align*}
    which completes the proof.
\end{proof}

Finally, for every $s\in\calS_{>1}$, we define the $\psi^\prime$-score gap inside the optimal action set,
\begin{align*}
    \forall\, k: \quad M_k(s) := \max_{a^*\in\calA^*_s} \, \psi^\prime(\pi_k(a^*|s)) - \min_{\tilde a^* \in \calA^*_s} \, \psi^\prime(\pi_k(\tilde a^*|s)).
\end{align*}

\begin{proof}[Proof of Lemma~\ref{lem:summable-error-gives-policy-convergence}]
    The proof can be decomposed into four steps.
    
    \noindent \textbf{Step 1: Existence of $\lim_{k\to\infty} D^{\pi^*}_{\pi_k}(s)$. } For any fixed optimal policy $\pi^* \in \Pi^*$ and any state $s\in\calS$, let $p = \pi^*(\cdot|s)$ and apply the three-point-descent lemma (Lemma~\ref{lem:three-point-descent}),
    \begin{align*}
        \eta \, \inner{\pi_{k+1}(\cdot|s) - \pi^*(\cdot|s)}{Q^{\pi_k}(s,\cdot)} \geq D^{\pi_{k+1}}_{\pi_k}(s) + D^{\pi^*}_{\pi_{k+1}}(s) - D^{\pi^*}_{\pi_k}(s).
    \end{align*}
    It follows that
    \begin{align*}
        D^{\pi^*}_{\pi_{k+1}}(s) - D^{\pi^*}_{\pi_k}(s) &\leq \eta  \underbrace{\inner{\pi_{k+1}(\cdot|s) - \pi^*(\cdot|s)}{Q^*(s,\cdot)}}_{\leq 0} - \underbrace{D^{\pi_{k+1}}_{\pi_k}(s)}_{\geq 0} + \eta \inner{\pi_{k+1}(\cdot|s) - \pi^*(\cdot|s)}{Q^{\pi_k}(s,\cdot) - Q^*(s,\cdot)} \\
        &\leq \eta \inner{\pi_{k+1}(\cdot|s) - \pi^*(\cdot|s)}{Q^{\pi_k}(s,\cdot) - Q^*(s,\cdot)} \\
        &\leq 2\eta \cdot \big\| Q^* - Q^{\pi_k} \big\|_\infty \leq 2\gamma\eta \cdot \delta_k,
    \end{align*}
    where $\delta_k := \| V^* - V^{\pi_k} \|_\infty$ and the last inequality is due to~\eqref{eq:Q-V-A-bounded}. Now define $x_k := D^{\pi^*}_{\pi_k}(s)$, $y_k := 2\gamma\eta \cdot \delta_k$, and $z_k := x_k - \sum_{t=0}^{k-1}y_t$. Notice that
    \begin{align*}
        z_{k+1} - z_{k} = x_{k+1} - x_k - y_k = D^{\pi^*}_{\pi_{k+1}}(s) - D^{\pi^*}_{\pi_{k}}(s) - 2\gamma\eta\cdot \delta_k \leq 0.
    \end{align*}
    Thus $\{ z_k \}_{k\geq 0}$ is non-increasing. On the other hand, as $\sum_{k=0}^\infty y_k = 2\gamma \eta \sum_{k=0}^\infty \delta_k < +\infty$,
    \begin{align*}
        z_k = x_k - \sum_{t=0}^{k-1}y_t \geq x_k - \sum_{t=0}^\infty y_t > -\infty,
    \end{align*}
    thus $\{ z_k \}_{k \geq 0}$ is bounded from below. Then by monotone convergence theorem, the limit $z_\infty := \lim_{k\to\infty} z_k$ exists. Since $x_k = z_k + \sum_{t=0}^{k-1} y_t$, the limit of $x_k$ also exists, given by $x_\infty := \lim_{k\to\infty} x_k = z_\infty + \sum_{t=0}^\infty y_t$. Therefore we conclude that
    \begin{align}  \label{eq:pf:Bregman-divergence-limit}
        \lim_{k\to\infty} D^{\pi^*}_{\pi_k}(s) \;\; \mbox{exists for any } s\in\calS, \; \pi^* \in \Pi^*.
    \end{align}

    \noindent \textbf{Step 2: Every accumulation point is an optimal policy.} Notice that $\pi_k$ lies in $\Pi$, which is bounded and closed. By Bolzano-Weierstrass theorem, there exists a converging sub-sequence $\{ \pi_{k_i} \}_{i \geq 0}$, i.e.,
    \begin{align*}
        \pi_{k_i} \stackrel{i\to\infty}{\to} \pi_\infty
    \end{align*}
    for some $\pi_\infty \in \Pi$. As the value function $V^\pi$ is a continuous function w.r.t. the policy $\pi$ (Lemma~\ref{lem:V-continuity}), the value sequence $\{ V^{\pi_{k_i}} \}_{i \geq 0}$ also converges. By $\sum_k \delta_k < \infty$, we have $\delta_k \to 0$ and thus  $V^{\pi_k} \rightarrow V^*$. As $\{ V^{\pi_{k_i}} \}_{i \geq 0}$ is a converging sub-sequence, its limit is also the optimal value, i.e., $V^{\pi_{k_i}} \stackrel{i\to\infty}{\to} V^*$. This implies that $V^{\pi_\infty} = V^*$, and thus $\pi_\infty \in \Pi^*$. 

    Now recall the definition of $\calS_{=1} := \{ s\in\calS: \; |\calA^*_s| =1 \}$ and $\calS_{>1} := \{ s\in\calS: \; |\calA^*_s| >1 \}$. As $\pi_\infty \in \Pi^*$, we immediately have
    \begin{alignat*}{2}
        &\forall\, s \in \calS, \; \forall\, a^\prime \not\in\calA^*_s: \quad &&\pi_\infty(a^\prime | s) = 0, \\
        &\forall \, s\in \calS_{=1}, \; \forall \, a^* \in \calA^*_s: \quad &&\pi_\infty(a^*|s) = 1.
    \end{alignat*}
    For any optimal action of $s\in\calS_{>1}$, we claim that \[\pi_\infty(a^*|s) \in \interior \, \dom \, \psi\]
    and will prove this via contradiction. 
    
    For Case~1, we have $[0,1] \subseteq \interior \, \dom \, \psi$ and thus $\pi_\infty(a^*|s) \in \interior\, \dom \, \psi$ holds automatically. Now consider Cases~2, 3, 4. Suppose that there exists $s\in\calS_{>1}$ and $a^*_0 \in \calA^*_s$ such that $\pi_\infty(a^*_0|s) \in \mbox{bd} \, \dom \, \psi$. In this context, it implies that $\pi_\infty(a^*_0|s) \in \{ 0,1 \}$. If $\pi_\infty(a^*_0|s) = 0$ (which implies that $0 \in \mbox{bd} \, \dom \, \psi$), then $\pi_{k_i}(a^*_0|s) \to 0$ and $\psi^\prime(\pi_{k_i}(a^*_0|s)) \to -\infty$ as $0 \in \mbox{bd} \, \dom \, \psi$. On the other hand, when $k_i$ is large enough there holds $b^{\pi_{k_i}}_s < 1/2$, yielding that $\max_{a^*\in\calA^*_s} \pi_{k_i}(a^*|s) \geq 1/(2|\calA|)$. Thus when $k_i$ is large enough,
    \begin{align*}
        M_{k_i}(s) &= \max_{a^* \in \calA^*_s} \, \psi^\prime(\pi_{k_i}(a^*|s)) - \min_{\tilde a^* \in \calA^*_s} \, \psi^\prime(\pi_{k_i}(\tilde a^*|s)) \\
        &\geq \max_{a^* \in \calA^*_s} \, \psi^\prime(\pi_{k_i}(a^*|s)) - \psi^\prime(\pi_{k_i}(a^*_0|s))  \\
        &= \psi^\prime \big(\max_{a^* \in \calA^*_s}\pi_{k_i}(a^*|s)\big) - \psi^\prime(\pi_{k_i}(a^*_0|s)) \geq \underbrace{\psi^\prime\Big( \frac{1}{2|\calA|} \Big)}_{> -\infty} - \underbrace{\psi^\prime(\pi_{k_i}(a^*_0|s))}_{\to-\infty} \to +\infty,
    \end{align*}
    where in the second equality we use the strict increasing property of $\psi^\prime$. If $\pi_\infty(a^*_0|s) = 1$ (implying that $1 \in \mbox{bd} \, \dom\, \psi$), then we have $\pi_\infty(\hat a^*|s) = 0$ for any $\hat a^* \in\calA^*_s$ and $\hat a^* \neq a^*_0$. As $1 \in \mbox{bd} \, \dom \, \psi$, there holds $\psi^\prime(\pi_{k_i}(a^*_0|s)) \to +\infty$. On the other hand, for any $\hat a^* \in\calA^*_s$ and $\hat a^* \neq a^*_0$, there holds $\pi_{k_i}(\hat a^*|s) \to 0$ and thus $\psi^\prime(\pi_{k_i}(\hat a^*|s)) \leq \psi^\prime(1/2) <+\infty$ when $k_i$ is large enough. Therefore,
    \begin{align*}
        M_{k_i}(s)&= \max_{a^* \in \calA^*_s} \, \psi^\prime(\pi_{k_i}(a^*|s)) - \min_{\tilde a^* \in \calA^*_s} \, \psi^\prime(\pi_{k_i}(\tilde a^*|s)) \\
        &\geq \psi^\prime(\pi_{k_i}(a^*_0|s)) - \psi^\prime(\pi_{k_i}(\hat a^*|s)) \geq \psi^\prime(\pi_{k_i}(a^*_0|s)) - \psi^\prime(1/2) \to +\infty.
    \end{align*}
    Therefore, we always have $M_{k_i}(s) \to +\infty$ if there exists an optimal action $a^*_0$ such that $\pi_\infty(a^*_0|s) \in \mbox{bd} \, \dom \, \psi$.
    To obtain the contradiction, recall that $\| V^* - V^{\pi_k} \|_\infty \to 0$ by $\sum_k \| V^* - V^{\pi_k} \|_\infty < \infty$. Then by Lemma~\ref{lem:sub-optimal-probabilities}, there exists a finite time $T$ such that $b^{\pi_k}_s < 1/2$ for all $k \geq T$, and thus $\max_{a^*\in\calA^*_s} \pi_k(a^*|s) > 0$. Now, pick
    \begin{align*}
        a^*_{k+1} \in \underset{a^*\in\calA^*_s}{\arg\max} \; \pi_{k+1}(a^*|s), \quad \tilde a^*_{k+1} \in \underset{a^*\in\calA^*_s}{\arg\min} \; \pi_{k+1}(\tilde a^*|s).
    \end{align*}
    As $\psi^\prime$ is monotonic increasing over $\interior \, \dom \, \psi$, one has $\psi^\prime(\pi_{k+1}(a^*_{k+1}|s)) = \max_{a^*\in\calA^*_s} \, \psi^\prime(\pi_{k+1}(a^*|s))$ and $\psi^\prime(\pi_{k+1}(\tilde a^*_{k+1} | s)) = \min_{a^*\in\calA^*_s} \, \psi^\prime(\pi_{k+1}(a^*|s))$. Since $\pi_{k+1}(a^*_{k+1}|s) > 0$, by Lemma~\ref{lem:KKT-cond},
    \begin{align*}
        &\psi^\prime \big( \pi_{k+1}(a^*_{k+1} | s) \big) - \psi^\prime\big( \pi_k(a^*_{k+1}|s) \big) = \eta \, Q^{\pi_k}(s,a^*_{k+1}) + \nu_s, \\
        &\psi^\prime \big( \pi_{k+1}(\tilde a^*_{k+1} |s) \big) - \psi^\prime \big( \pi_k(\tilde a^*_{k+1}|s) \big) \geq \eta \, Q^{\pi_k}(s, \tilde a^*_{k+1}) + \nu_s.
    \end{align*}
    It follows that
    \begin{align*}
        &\phantom{=\,\,\,}\psi^\prime \big( \pi_{k+1}(a^*_{k+1} | s) \big) - \psi^\prime \big( \pi_{k+1}(\tilde a^*_{k+1} |s) \big) \\
        &\leq \psi^\prime\big( \pi_k(a^*_{k+1}|s) \big) - \psi^\prime \big( \pi_k(\tilde a^*_{k+1}|s) \big) + \eta \, \big[ Q^{\pi_k}(s,a^*_{k+1}) - Q^{\pi_k}(s,\tilde a^*_{k+1}) \big] \\
        &\leq \psi^\prime\big( \pi_k(a^*_{k+1}|s) \big) - \psi^\prime \big( \pi_k(\tilde a^*_{k+1}|s) \big) + \eta \, \big[ Q^{*}(s,a^*_{k+1}) - Q^*(s,\tilde a^*_{k+1}) \big] + 2\eta \, \big\| Q^* - Q^{\pi_k} \big\|_\infty \\
        &\leq \psi^\prime\big( \pi_k(a^*_{k+1}|s) \big) - \psi^\prime \big( \pi_k(\tilde a^*_{k+1}|s) \big) + 2\eta \, \delta_k.
    \end{align*}
    Then there holds
    \begin{align*}
        M_{k+1}(s) &= \max_{a^*, \, \tilde a^* \in \calA^*_s} \; \Big\{ \psi^\prime(\pi_{k+1}(a^*|s)) - \psi^\prime(\pi_{k+1}(\tilde a^* | s)) \Big\} \\
        &= \psi^\prime\big(\pi_{k+1}(a^*_{k+1} |s)\big) - \psi^\prime\big( \pi_{k+1}(\tilde a^*_{k+1}|s) \big) \\[.3em]
        &\leq \psi^\prime\big( \pi_k(a^*_{k+1}|s) \big) - \psi^\prime \big( \pi_k(\tilde a^*_{k+1}|s) \big) + 2\eta \, \delta_k \\
        &\leq \max_{a^*, \, \tilde a^* \in \calA^*_s} \; \Big\{ \psi^\prime(\pi_{k}(a^*|s)) - \psi^\prime(\pi_{k}(\tilde a^* | s)) \Big\} + 2\eta \, \delta_k =M_k(s) + 2\eta \, \delta_k.
    \end{align*}
    By telescoping,
    \begin{align}  \label{eq:0-finite-1-infinite:upper-bound-for-M}
        \forall\, k \geq T: \quad M_k(s) \leq M_T(s) + 2 \eta \sum_{t=T}^{k-1} \delta_t \leq M_T(s) +  2\eta \sum_{t=0}^\infty \delta_t.
    \end{align}
    Therefore we obtain that 
    \begin{align}  \label{eq:bounded-M}
        \forall\, s\in \calS_{>1}:\quad \sup _{k \geq 0} \; M_k(s) \leq \max_{t \leq T} M_t(s) + \sup_{k > T} \, M_k(s) \leq \sum_{t=0}^T M_t(s) + M_T(s) + 2\eta \sum_k \delta_k < \infty.
    \end{align}
    This however contradicts $M_{k_i}(s) \to +\infty$. Therefore it is concluded that 
    \begin{align*}
        \forall\, s\in\calS_{>1}, \; \forall\, a^*\in\calA^*_s: \quad \pi_\infty(a^*|s) \in \interior \, \dom \, \psi.
    \end{align*}

    \noindent \textbf{Step 3: Showing that $\lim_{k\to\infty} D^{\pi_\infty}_{\pi_k}(s) = 0$.} Recall that $\pi_{k_i} \to \pi_\infty$. As $h$ is decomposable,
    \begin{align*}
        \forall\, s \in \calS: \quad D_h(\pi_\infty(\cdot|s) \, || \, \pi_{k_{i}}(\cdot|s)) &= \sum_{a\in\calA} d_\psi(\pi_\infty(a|s) \, || \, \pi_{k_i}(a|s)).
    \end{align*}
    By Lemma~\ref{lem:KKT-cond}, $\pi_{k_i}(a|s) \in \interior \, \dom \, \psi$. Then by Lemma~\ref{lem:pro-3.3-of-baischke}, for any $s\in\calS$ and $a\in\calA$,
    \begin{align*}
        \left . \begin{array}{r}
              \pi_{k_i}(a|s) \in \interior \, \dom \, \psi,  \\
              \pi_{k_i}(a|s) \to \pi_\infty(a|s) \in [0,1] \subseteq \dom \, \psi,
        \end{array} \right \} \;\; \Rightarrow \;\; d_\psi(\pi_\infty(a|s) \, || \, \pi_{k_i}(a|s)) \to 0.
    \end{align*}
    Therefore
    \begin{align*}
        \forall\, s\in\calS: \quad \lim_{i\to\infty} D^{\pi_\infty}_{\pi_{k_i}}(s) = 0.
    \end{align*}
    By Step 1, we know that $\lim_{k\to\infty} D^{\pi^*}_{\pi_k}(s)$ exists for all $s\in\calS$ and any $\pi^* \in \Pi^*$. As $\pi_\infty \in \Pi^*$, the limit $\lim_{k\to\infty} D^{\pi_\infty}_{\pi_k}(s)$ also exists for all $s\in\calS$. Since one of its subsequences $\big\{ D^{\pi_\infty}_{\pi_{k_i}}(s) \big\}$ converges to zero, the complete sequence converges to zero as well, i.e.,
    \begin{align*}
        \forall \, s\in\calS: \quad \lim_{k\to\infty} D^{\pi_\infty}_{\pi_k}(s) = 0.
    \end{align*}

    \noindent \textbf{Step 4: Deduce that $\pi_k \to \pi_\infty$.} As $h$ is decomposable,
    \begin{align*}
        &\forall\, s\in\calS : \quad \lim_{k\to\infty} \, D_h(\pi_\infty(\cdot|s) \, || \, \pi_k(\cdot|s)) = 0 \\
        \Longrightarrow \quad & \forall\, s\in\calS, \; \forall \, a\in\calA: \quad \lim_{k\to\infty} \, d_\psi (\pi_\infty(a|s) \, || \, \pi_k(a|s)) = 0.
    \end{align*}
    Since $b^{\pi_k}_s \to 0$ for all $s\in\calS$, one immediately has
    \begin{alignat*}{2}
        &\forall\, s\in\calS, \; \forall\, a^\prime \not\in\calA^*_s : \quad &&\pi_k(a^\prime | s) \to 0 = \pi_\infty(a^\prime | s), \\
        &\forall\, s \in \calS_{=1} , \; \forall a^* \in \calA^*_s: \quad &&\pi_k(a^*|s) \to 1 = \pi_\infty(a^*|s).
    \end{alignat*}
    It remains to consider the case that $(s,a^*)$ with $s\in\calS_{>1}$ and $a^* \in \calA^*_s$.  Suppose that $\pi_{k}(a^*|s) \not\to \pi_\infty(a^*|s)$. Then there exists a subsequence $\{ \pi_{k_j} \}_{j\geq0}$ satisfying
    \begin{align*}
        \forall\, j :\quad \big| \pi_{k_j}(a^*|s) - \pi_\infty(a^*|s) \big| \geq \epsilon > 0.
    \end{align*}
    Applying Bolzano-Weierstrass theorem again, there exists a converging sub-sub-sequence $\big\{ \pi_{k_{j_l}} \big\}_{l \geq 0}$. Letting $\hat\pi_\infty := \lim_{l\to\infty} \, \pi_{k_{j_l}}$, it is clear that $\hat\pi_\infty \neq \pi_\infty$. As we have shown that $D^{\pi_\infty}_{\pi_k}(s) \to 0$, there holds
    \begin{align*}
        \forall \, s\in\calS: \quad \lim_{l \to \infty} \, D^{\pi_\infty}_{\pi_{k_{j_l}}}(s) = 0,
    \end{align*}
    implying that
    \begin{align*}
        \forall\, s\in\calS_{>1}, \; \forall\, a^*\in\calA^*_s: \quad \lim_{l \to \infty} \, d_\psi (\pi_\infty(a^*|s) \, || \, \pi_{k_{j_l}}(a^*|s)) = 0.
    \end{align*}
    Notice that $\pi_{k_{j_l}}(a^*|s) \in \interior \, \dom \, \psi$ by Lemma~\ref{lem:KKT-cond}. Since $\pi_\infty(a^*|s) \in \interior \, \dom \, \psi$, Lemma~\ref{lem:thm-3.8-iii-3.9-of-bauschke} implies that
    \begin{align*}
        \left . \begin{array}{r}
            \pi_\infty(a^*|s) \in \interior \, \dom\, \psi, \\
            \pi_{k_{j_l}}(a^*|s) \in \interior \, \dom \, \psi, \; \pi_{k_{j_l}}(a^*|s) \to \hat\pi_\infty(a^*|s), \\
            \{ \pi_\infty(a^*|s), \; \hat\pi_\infty(a^*|s) \} \cap \interior \, \dom \, \psi \neq \emptyset, \; d_\psi(\pi_\infty(a^*|s) \, \| \, \pi_{k_{j_l}}(a^*|s)) \to 0
        \end{array} \right \} \;\; \Rightarrow \;\; \pi_\infty(a^*|s)  = \hat\pi_\infty(a^*|s).
    \end{align*}
    This contradicts $\big| \pi_{k_j}(a^*|s) - \pi_\infty(a^*|s) \big| \geq \epsilon > 0$, which concludes the proof.
\end{proof}

\subsection{Proof of Lemma~\ref{lem:gap-metric-increase}}
\label{sec:pf:lem:gap-metric-increase}
\begin{proof}[Proof of Lemma~\ref{lem:gap-metric-increase}~\ref{lem:gap-metric-increase-lower}]
    Pick $a^\prime_{k+1} \not\in\calA^*_s$ and $a^*_{k+1} \in \calA^*_s$ such that
    \begin{align*}
        \psi^\prime (\pi_{k+1}(a^\prime_{k+1} | s)) = \max_{a^\prime\not\in\calA^*_s} \, \psi^\prime (\pi_{k+1}(a^\prime|s)), \quad \psi^\prime(\pi_{k+1}(a^*_{k+1}|s)) = \min_{a^* \in \calA^*_s} \, \psi^\prime (\pi_{k+1}(a^*|s)).
    \end{align*}
    Note that $\psi^\prime$ is strictly increasing on $\interior \, \dom \, \psi$. Thus we know that $\pi_{k+1}(a^\prime_{k+1}|s) = \max_{a^\prime \not \in \calA^*_s} \, \pi_{k+1}(a^\prime | s)$ and $\pi_{k+1}(a^*_{k+1}|s) = \min_{a^* \in \calA^*_s} \pi_{k+1}(a^*|s)$. As $b^{\pi_{k+1}}_s> 0$, one has $\pi_{k+1}(a^\prime_{k+1}|s) > 0$. Thus, the application of Lemma~\ref{lem:KKT-cond} yields that
    \begin{align*}
        &\psi^\prime (\pi_{k+1}(a^*_{k+1} |s )) - \psi^\prime (\pi_{k}(a^*_{k+1} | s)) \geq \eta \, Q^{\pi_k}(s,a^*_{k+1}) + \nu_s, \\
        &\psi^\prime(\pi_{k+1}(a^\prime_{k+1}|s)) - \psi^\prime(\pi_k(a^\prime_{k+1}|s)) = \eta\, Q^{\pi_k}(s, a^\prime_{k+1}) + \nu_s.
    \end{align*}
    Consequently, for $k \geq T(\sigma)$, there holds $\| Q^* - Q^{\pi_k} \|_\infty \leq \sigma$ and thus
    \begin{align*}
        &\phantom{=\,\,\,}\psi^\prime (\pi_{k+1}(a^*_{k+1} |s )) - \psi^\prime(\pi_{k+1}(a^\prime_{k+1}|s)) \\
        &\geq \psi^\prime (\pi_{k}(a^*_{k+1} | s)) - \psi^\prime(\pi_k(a^\prime_{k+1}|s)) + \eta \, \big[ Q^{\pi_k}(s, a_{k+1}^*) - Q^{\pi_k}(s, a^\prime_{k+1}) \big] \\
        &\geq \psi^\prime (\pi_{k}(a^*_{k+1} | s)) - \psi^\prime(\pi_k(a^\prime_{k+1}|s)) + \eta \, \big[ Q^*(s,a^*_{k+1}) - Q^*(s, a^\prime_{k+1}) \big]
         - 2\eta \, \big\| Q^* - Q^{\pi_k} \big\|_\infty \\
        &\geq \psi^\prime (\pi_{k}(a^*_{k+1} | s)) - \psi^\prime(\pi_k(a^\prime_{k+1}|s)) + \eta \big( \Delta_s - 2\sigma \big).
    \end{align*}
    Therefore,
    \begin{align*}
        G_{k+1}(s) &= \min_{a^* \in \calA^*_s} \, \psi^\prime(\pi_{k+1}(a^*|s)) - \max_{a^\prime \not \in \calA^*_s} \, \psi^\prime(\pi_{k+1}(a^\prime | s)) \\
        &= \psi^\prime(\pi_{k+1}(a^*_{k+1} | s)) - \psi^\prime(\pi_{k+1}(a^\prime_{k+1}| s)) \\[.5em]
        &\geq \psi^\prime(\pi_{k}(a^*_{k+1} | s)) - \psi^\prime(\pi_{k}(a^\prime_{k+1}| s)) + \eta \,(\Delta_s - 2\sigma) \\[.5em]
        &\geq \min_{a^* \in \calA^*_s} \, \psi^\prime(\pi_k(a^*|s)) - \max_{a^\prime \not \in \calA^*_s} \,\psi^\prime(\pi_{k}(a^\prime| s)) + \eta \,(\Delta_s - 2\sigma) = G_k(s) + \eta \,(\Delta_s -2 \sigma).
    \end{align*}    
\end{proof}

\begin{proof}[Proof of Lemma~\ref{lem:gap-metric-increase}~\ref{lem:gap-metric-increase-upper}]
    Pick $a^*_{k+1} \in \calA^*_s$ such that
    \begin{align*}
        \psi^\prime (\pi_{k+1}(a^*_{k+1}|s)) = \max_{a^*\in\calA^*_s} \, \psi^\prime (\pi_{k+1}(a^*|s)).
    \end{align*}
    As $b^{\pi_{k+1}}_s < 1$, we know that $\pi_{k+1}(a^*_{k+1} | s) = \max_{a^*\in\calA^*_s} \pi_{k+1}(a^*|s) > 0$. By Lemma~\ref{lem:KKT-cond}, for any $a^\prime \not \in \calA^*_s$,
    \begin{align*}
        &\psi^\prime(\pi_{k+1}(a^*_{k+1}|s)) - \psi^\prime(\pi_k(a^*_{k+1}|s)) = \eta \, Q^{\pi_k}(s, a^*_{k+1}) + \nu_s, \\
        &\psi^\prime(\pi_{k+1}(a^\prime | s)) - \psi^\prime(\pi_k(a^\prime|s)) \geq \eta \, Q^{\pi_k}(s, a^\prime) + \nu_s.
    \end{align*}
    It follows that for $k \geq T(\sigma)$,
    \begin{align*}
        \psi^\prime(\pi_{k+1}(a^*_{k+1}|s)) - \psi^\prime(\pi_{k+1}(a^\prime | s)) &\leq \psi^\prime(\pi_k(a^*_{k+1}|s)) - \psi^\prime(\pi_k(a^\prime|s)) + \eta \, \big[ Q^{\pi_k}(s,a^*_{k+1}) - Q^{\pi_k}(s,a^\prime) \big] \\
        &\leq \psi^\prime(\pi_k(a^*_{k+1}|s)) - \psi^\prime(\pi_k(a^\prime|s)) + \eta \, \big[ Q^*(s,a^*_{k+1}) - Q^*(s,a^\prime) + 2\sigma \big] \\
        &=\psi^\prime(\pi_k(a^*_{k+1}|s)) - \psi^\prime(\pi_k(a^\prime|s)) + \eta \, \big[ V^*(s) - Q^*(s,a^\prime) + 2\sigma \big].
    \end{align*}
    Thus,
    \begin{align*}
        G_{k+1}(s,a^\prime) &= \max_{a^*\in\calA^*_s} \, \psi^\prime(\pi_{k+1}(a^*|s)) - \psi^\prime(\pi_{k+1}(a^\prime|s)) \\
        &= \psi^\prime(\pi_{k+1}(a^*_{k+1}|s)) - \psi^\prime(\pi_{k+1}(a^\prime|s)) \\[.5em]
        &\leq \psi^\prime(\pi_k(a^*_{k+1}|s)) - \psi^\prime(\pi_k(a^\prime|s)) + \eta \, \big[ V^*(s) - Q^*(s,a^\prime) + 2\sigma \big] \\[.5em]
        &\leq \max_{a^*\in\calA^*_s} \, \psi^\prime (\pi_k(a^*|s)) - \psi^\prime(\pi_k(a^\prime|s)) + \eta \, \big[ V^*(s) - Q^*(s,a^\prime) + 2\sigma \big] \\
        &= G_k(s,a^\prime) + \eta \, \big[ V^*(s) - Q^*(s,a^\prime) + 2\sigma \big].
    \end{align*}
\end{proof}

\begin{proof}[Proof of Lemma~\ref{lem:gap-metric-increase}~\ref{lem:termination-condition}]
\phantom{sdf}

    \noindent\textbf{First part:} Since $\pi_k(a|s)\in \interior\, \dom \, \psi$ by Lemma~\ref{lem:KKT-cond},  $b_s^{\pi_k}=0$ implies that $\pi_k(a^\prime|s) = 0$ for all $a^\prime \not\in\calA^*_s$ and thus $0\in \interior\, \dom \, \psi$. Now suppose that $b^{\pi_{k+1}}_s > 0$. Then there exists a suboptimal action $a^\prime \not \in \calA^*_s$ such that $\pi_{k+1}(a^\prime|s) > 0$. As $b^{\pi_k}_s = 0$, one has
    \begin{align*}
        \sum_{a^* \in \calA^*_s} \pi_k(a^*|s) = 1 , \quad \sum_{a^*\in\calA^*_s} \pi_{k+1}(a^*|s) \leq 1 - \pi_{k+1}(a^\prime|s) < 1.
    \end{align*}
    It implies that there exists at least one optimal action $a^* \in \calA^*_s$ such that $\pi_{k+1}(a^*|s) < \pi_k(a^*|s)$. Let
    \begin{align*}
        \epsilon := \min \, \Big\{ \frac{1}{2} \pi_{k+1}(a^\prime|s), \;\; \frac{1}{2} \big( \pi_k(a^*|s) - \pi_{k+1}(a^*|s) \big)  \Big\} > 0.
    \end{align*}
    Define the following probability distribution $p \in \Delta(\calA)$,
    \begin{align*}
        p(a) := \begin{cases}
            \pi_{k+1}(a|s) , & \mbox{if } \; a \neq a^\prime, \; a\neq a^*, \\
            \pi_{k+1}(a^\prime | s) - \epsilon > 0 , & \mbox{if } \; a = a^\prime, \\
            \pi_{k+1}(a^*|s) + \epsilon < \pi_k(a^*|s) , & \mbox{if } \; a = a^*.
        \end{cases}
    \end{align*}
    For any $q \in \Delta(\calA)$, let $f(q) := \inner{q}{Q^{\pi_k}(s,\cdot)} - \frac{1}{\eta} D_h(q \, \| \, \pi_{k}(\cdot|s))$. By the PMD policy update there holds $f(\pi_{k+1}(\cdot|s)) = \max_{q \in \Delta(\calA)} f(q)$. Now consider the difference $f(p) - f(\pi_{k+1}(\cdot|s))$:
    \begin{align*}
        &\phantom{=\,\,\,}f(p) - f(\pi_{k+1}(\cdot|s)) \\
        &= \sum_{a\in\calA} \big( p(a) - \pi_{k+1}(a|s) \big) \cdot Q^{\pi_k}(s,a)  - \frac{1}{\eta} \Big( D_h(p \, \| \, \pi_k(\cdot|s)) - D_h (\pi_{k+1}(\cdot|s) \, \| \, \pi_k(\cdot|s)) \Big) \\
        &= \sum_{a\in\calA} \big( p(a) - \pi_{k+1}(a|s) \big) \cdot Q^{\pi_k}(s,a) - \frac{1}{\eta} \sum_{a\in\calA} \Big( d_\psi (p(a) \, \| \, \pi_k(a|s)) - d_\psi (\pi_{k+1}(a|s) \, \| \, \pi_k(a|s)) \Big).
    \end{align*}
    For the first term, noting that $k \geq T(\sigma)$,
    \begin{align*}
        \sum_{a\in\calA} \big( p(a) - \pi_{k+1}(a|s) \big) \cdot Q^{\pi_k}(s,a) = \epsilon \cdot \big( Q^{\pi_k}(s,a^*) - Q^{\pi_k}(s,a^\prime) \big) \geq (\Delta - 2\sigma) \cdot \epsilon.
    \end{align*}
    For the divergence term, as $\pi_k(a|s)$ and $\pi_{k+1}(a|s)$ both lie in $\interior \, \dom \, \psi$, one has $p(a) \in \interior\, \dom \, \psi$ for all $a\in\calA$. Therefore,
    \begin{align*}
        &\phantom{=\,\,\,}\sum_{a\in\calA} \Big( d_\psi (p(a) \, \| \, \pi_k(a|s)) - d_\psi (\pi_{k+1}(a|s) \, \| \, \pi_k(a|s)) \Big) \\
        &= d_{\psi} \big( \pi_{k+1}(a^\prime|s) - \epsilon \, \| \, \pi_k(a^\prime|s) \big) - d_{\psi} \big( \pi_{k+1}(a^\prime|s) \, \| \, \pi_k(a^\prime|s) \big) \\
        &\quad\quad\quad\quad + d_{\psi} \big( \pi_{k+1}(a^*|s) + \epsilon \, \| \, \pi_k(a^*|s) \big) - d_{\psi} \big( \pi_{k+1}(a^*|s) \, \| \, \pi_k(a^*|s) \big) \\[.5em]
        &= \psi \big( \pi_{k+1}(a^\prime |s ) - \epsilon \big) - \psi\big( \pi_{k+1}(a^\prime | s) \big) + \epsilon \cdot \psi^\prime \big( \pi_k(a^\prime|s) \big) \\
        &\quad\quad\quad\quad + \psi\big( \pi_{k+1}(a^*|s) + \epsilon \big) - \psi \big( \pi_{k+1}(a^*|s) \big) - \epsilon \cdot \psi^\prime \big( \pi_k(a^*|s) \big).
    \end{align*}
    Noticing that $\psi$ is continuous on $[0,1]$ and is differentiable on $\interior \, \dom \, \psi$, one can apply the mean value theorem to obtain that
    \begin{align*}&\phantom{=\,\,\,}\sum_{a\in\calA} \Big( d_\psi (p(a) \, \| \, \pi_k(a|s)) - d_\psi (\pi_{k+1}(a|s) \, \| \, \pi_k(a|s)) \Big)\\
        &= \epsilon \cdot \ls[-\psi^\prime (\xi_1) +  \psi^\prime \big( \pi_k(a^\prime|s) \big) + \psi^\prime(\xi_2) - \psi^\prime \big( \pi_k(a^*|s) \big) \rs],
    \end{align*}
    where 
    \begin{align*}
        &0 = \pi_k(a^\prime|s) < \frac{1}{2}\pi_{k+1}(a^\prime|s) \leq \pi_{k+1}(a^\prime|s) - \epsilon \leq  \xi_1 \leq \pi_{k+1}(a^\prime|s) , \\
        &\pi_{k+1}(a^*|s) \leq \xi_2 \leq \pi_{k+1}(a^*|s) + \epsilon \leq \frac{1}{2} \big( \pi_k(a^*|s) + \pi_{k+1}(a^*|s) \big)  < \pi_k(a^*|s).
    \end{align*}
    By the strict increasing monotonicity of $\psi^\prime$, we obtain that
    \begin{align*}
        \sum_{a\in\calA} \Big( d_\psi (p(a) \, \| \, \pi_k(a|s)) - d_\psi (\pi_{k+1}(a|s) \, \| \, \pi_k(a|s)) \Big) \leq 0.
    \end{align*}
    Putting it all together,
    \begin{align*}
        f(p) - f(\pi_{k+1}(\cdot|s)) \geq (\Delta - 2\sigma)\cdot  \epsilon > 0,
    \end{align*}
    which implies that $f(p) > f(\pi_{k+1}(\cdot|s))$, contradicting $f(\pi_{k+1}(\cdot|s)) = \max_{q\in\Delta(\calA)} f(q)$. Therefore there must hold $b^{\pi_{k+1}}_s = 0$, which further gives that $b^{\pi_t}_s = 0$ for all $t \geq k$.

    \vspace{1em}

    \noindent \textbf{Second part:} First noticing that $b^{\pi_k}_{s^\prime} = 0$ holds automatically for $s^\prime \not\in\tilde \calS$, there holds $b^{\pi_k}_s = 0$ for all $s\in\calS$. By Lemma~\ref{lem:sub-optimal-probabilities}, there holds $V^{\pi_k} = V^*$ and thus $\pi_k \in \Pi^*$. Next it suffices to show $\pi_{k+1}=\pi_k$. If this is not true, then there exists $s\in\calS$ and $a\in\calA$ such that $\pi_{k+1}(a|s) \neq \pi_k(a|s)$. Recall the function
    \begin{align*}
        f(p) = \inner{p}{Q^{\pi_k}(s,\cdot)} - \frac{1}{\eta} D_h \big( p \, \| \, \pi_k(\cdot|s) \big).
    \end{align*}
    As $\pi_k \in \Pi^*$,
    \begin{align*}
        f(p) &= \inner{p}{Q^*(s,\cdot)} - \frac{1}{\eta} D_h \big( p \, \| \, \pi_k(\cdot|s) \big) \\
        &= \underbrace{\inner{p}{A^*(s,\cdot)} - \frac{1}{\eta} D_h \big( p \, \| \, \pi_k(\cdot|s) \big)}_{:= g(p)} + V^*(s).
    \end{align*}
    By the PMD policy update we know that $\pi_{k+1}(\cdot|s) = \arg\,\max_{p\in\Delta(\calA)} g(p)$. Note that
    \begin{align*}
        g\big( \pi_{k+1}(\cdot|s) \big) &= \inner{\pi_{k+1}(\cdot|s)}{A^*(s,\cdot)} - \frac{1}{\eta} D_h \big( \pi_{k+1}(\cdot|s) \, \| \, \pi_k(\cdot|s) \big) \\
        &\leq -\frac{1}{\eta} D_h \big( \pi_{k+1}(\cdot|s) \, \| \, \pi_k(\cdot|s) \big) \\
        &= -\frac{1}{\eta} \sum_{a^\prime\in\calA} d_{\psi} \big( \pi_{k+1}(a^\prime|s) \, \| \, \pi_k(a^\prime|s) \big) \\
        &\leq -\frac{1}{\eta} d_{\psi} \big( \pi_{k+1}(a|s) \, \| \, \pi_k(a|s) \big).
    \end{align*}
    As $\pi_{k+1}(a|s)$, $\pi_k(a|s)$ lie in $\interior \, \dom \, \psi$ and $\psi$ is strictly convex on $\interior\, \dom \, \psi$, one has $d_{\psi} \big( \pi_{k+1}(a|s) \, \| \, \pi_k(a|s) \big) > 0$ as $\pi_{k+1}(a|s) \neq \pi_k(a|s)$. Thus we obtain that $g\big( \pi_{k+1}(\cdot|s) \big) < 0$. In contrast, it is clear that $g\big( \pi_k(\cdot|s) \big) = 0$, and thus $g\big( \pi_{k+1}(\cdot|s) \big) < g\big( \pi_{k}(\cdot|s) \big)$ which leads to a contradiction.
\end{proof}

\subsection{Proof of Lemma~\ref{lem:integral-discrimination}}
\label{sec:pf:lem:integral-discrimination}
As $f$ is strictly decreasing and continuous with $f(\epsilon) < \infty$ and $\lim_{x \downarrow 0} f(x) = +\infty$, its inverse function $f^{-1}: [f(\epsilon), \infty) \to (0, \epsilon]$ is well-defined, non-negative, continuous and strictly decreasing. 
To establish the integrability of $f^{-1}$, we employ the layer-cake representation for the non-negative function $f$ on the interval $(0, \epsilon]$:
\begin{equation*}
    \int_0^\epsilon f(x) dx = \int_0^\infty \mu\big(\{x \in (0, \epsilon] : f(x) > y\}\big) dy,
\end{equation*}
where $\mu$ denotes the Lebesgue measure. For $0 \le y \le f(\epsilon)$, the strict monotonicity of $f$ implies $f(x) \ge f(\epsilon) \ge y$ for all $x \in (0, \epsilon]$, yielding a constant subset measure of $\epsilon$. For $y > f(\epsilon)$, the condition $f(x) > y$ is  equivalent to $0 < x < f^{-1}(y)$, which yields a measure of $f^{-1}(y)$. Therefore, by splitting the integral according to whether $y \le f(\epsilon)$ or $y > f(\epsilon)$, we obtain
\begin{equation*}
    \int_0^\epsilon f(x) dx = \int_0^{f(\epsilon)} \epsilon \, dy + \int_{f(\epsilon)}^\infty f^{-1}(y) dy = \epsilon f(\epsilon) + \int_{f(\epsilon)}^\infty f^{-1}(y) dy.
\end{equation*}
Rearranging the terms immediately demonstrates the finiteness of the integral of the inverse function over its infinite tail:
\begin{equation*}
    \int_{f(\epsilon)}^\infty f^{-1}(y) dy = \int_0^\epsilon f(t) dt - \epsilon f(\epsilon) < \infty.
\end{equation*}
Finally, define the sequence $y_k = c_0 + c_1 k$ for $k \ge 0$. Since $c_0 \ge f(\epsilon)$ and $c_1 > 0$, the sequence is strictly increasing and bounded below by $f(\epsilon)$. 
The monotonicity of $f^{-1}$ ensures that $f^{-1}(y_{k+1}) \le f^{-1}(y)$ for any $y \in [y_k, y_{k+1}]$. Integrating this over the interval of width $c_1$ gives $c_1 f^{-1}(y_{k+1}) \le \int_{y_k}^{y_{k+1}} f^{-1}(y) dy$. 
Summing this inequality for $k = 0, \dots, N$, we establish
\begin{equation*}
    c_1 \sum_{k=1}^{N+1} f^{-1}(y_k) \le \sum_{k=0}^{N} \int_{y_k}^{y_{k+1}} f^{-1}(y) dy = \int_{y_0}^{y_{N+1}} f^{-1}(y) dy \le \int_{f(\epsilon)}^{\infty} f^{-1}(y) dy < \infty.
\end{equation*}
The proof is complete since the first term $f^{-1}(c_0)$ is also finite.

\bibliographystyle{plainnat}
\bibliography{refs}

\end{document}